\documentclass[a4paper,11pt, reqno]{amsart}
\usepackage{amsfonts, color}
\usepackage{xcolor}
\usepackage{geometry}
\usepackage[normalem]{ulem}
\usepackage{relsize}
\usepackage{amstext, amsthm, amssymb, amsmath, mathtools}
\usepackage{mathrsfs}
\usepackage{graphicx}
\usepackage{mathrsfs} 
\usepackage{hyperref}
\usepackage{pgfmath} 
\usepackage[capitalize,nameinlink,noabbrev]{cleveref}
\usepackage{cancel}
\usepackage[font=small, labelfont=bf, labelsep=colon, justification=centering]{caption}
\usepackage{float}
\usepackage{enumitem}
\usepackage{ upgreek }
\usepackage{ mathrsfs }
\usepackage{url}
\usepackage{soul}
\usepackage{subcaption}
\usepackage{layout}
\usepackage{tikz-cd}
\usepackage[linesnumbered,lined,ruled,commentsnumbered]{algorithm2e}
\usepackage{tikz}
\usetikzlibrary{patterns, shapes, decorations.pathmorphing, fit}

\usepackage{graphicx}
\usepackage{caption}
\usepackage{calc}
\usepackage{ragged2e}
\usepackage{setspace}
\newcommand{\FundingLogos}{%
  \raisebox{0pt}{\includegraphics[height=1.5cm]{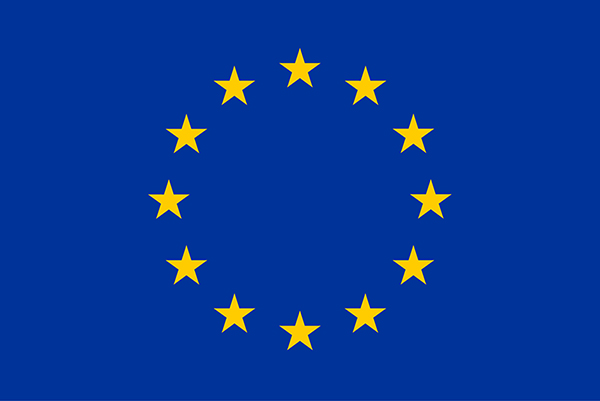}}%
  \hspace{1em}%
  \raisebox{0pt}{\includegraphics[height=1.5cm]{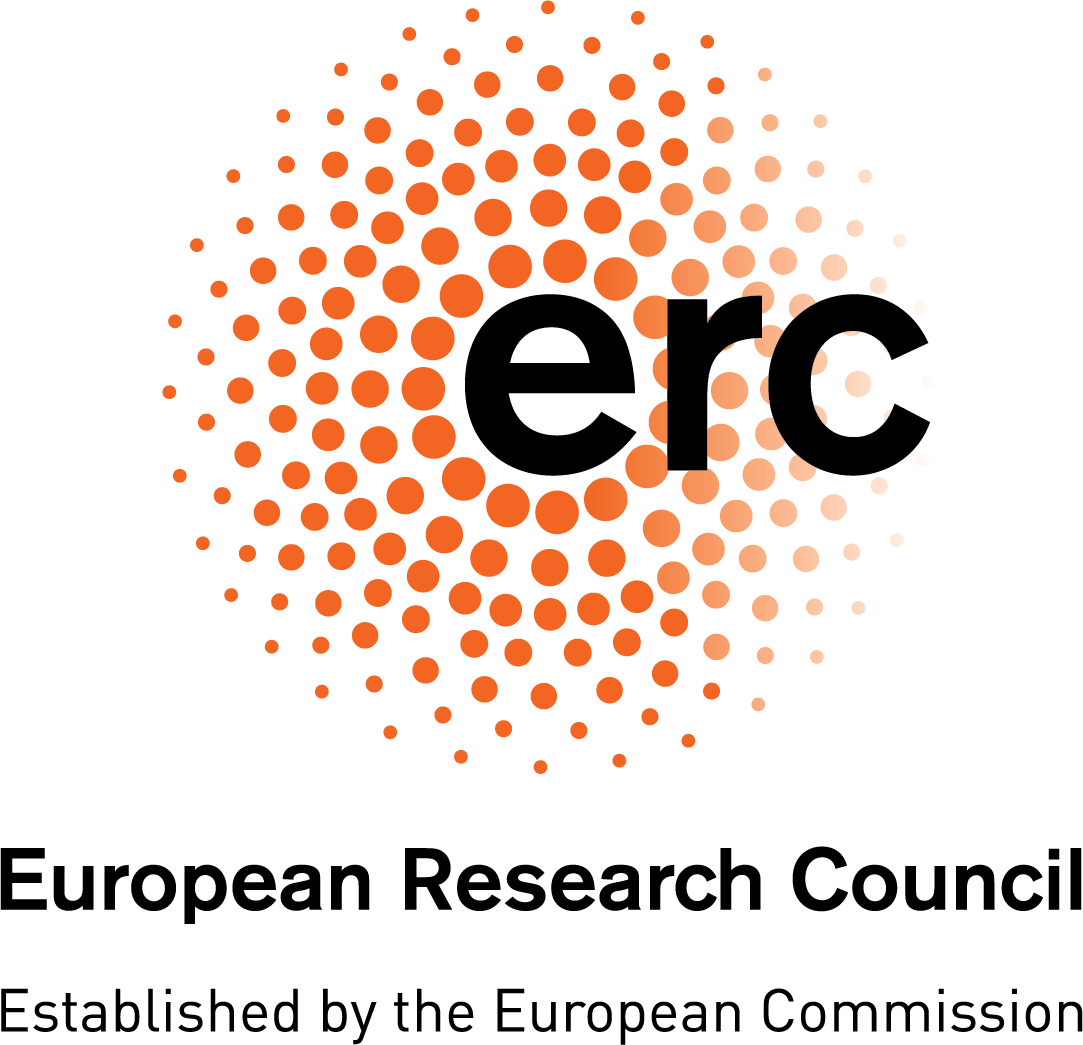}}%
}

\usepackage{tcolorbox}
\tcbset{colback=white,boxrule=0.5pt,arc=2pt,left=4pt,right=4pt,top=4pt,bottom=4pt}

\DeclareMathOperator*{\argmin}{arg\,min}

\DeclareMathOperator{\im}{Im}
\DeclareMathOperator{\dom}{Dom}

\numberwithin{equation}{section}
\theoremstyle{plain}
\begingroup
\newtheorem{theorem}{Theorem}
\newtheorem{lemma}[theorem]{Lemma}
\newtheorem{proposition}[theorem]{Proposition}
\newtheorem{corollary}[theorem]{Corollary}
\endgroup
\numberwithin{theorem}{section}

\theoremstyle{definition}
\begingroup
\newtheorem{definition}{Definition}

\newtheorem{assumption}{Assumption}
\newtheorem{axiom}{Axiom}
\endgroup

\theoremstyle{remark}
\newtheorem{remark}{Remark}

\crefname{axiom}{Axiom}{Axioms}
\crefname{property}{Property}{Properties}
\crefname{assumption}{Assumption}{Assumptions}

\newcommand{\R}{\mathbb{R}}
\newcommand{\F}{\mathcal{F}}

\newcommand{\e}{\varepsilon}
\newcommand{\weak}{\rightharpoonup}

\newcommand{\NN}{\mathbb{N}}

\newcommand{\X}{\mathcal{X}}

\newcommand{\Ii}{\mathbb{I}}
\newcommand{\supp}{\mathrm{supp}}
\newcommand{\E}{\mathcal{E}}
\newcommand{\Xx}{\mathbb{X}}
\newcommand{\Z}{\mathbb{Z}}
\newcommand{\Cost}{\mathcal{L}}
\newcommand{\Speed}{\upvartheta}
\newcommand{\Pp}{\mathscr{P}}
\newcommand{\Price}{\mathscr{P}}
\newcommand{\diff}{\,\mathrm{d}}
\newcommand{\Crit}{\mathscr{C}}

\newcommand{\Cc}{\mathfrak{c}}

\newcounter{property} 

\newenvironment{properties}
{
    \begin{list}{\textit{Property \arabic{property}:\,}}
    {
        \usecounter{property} 
        \setlength{\leftmargin}{2em} 
        \setlength{\labelwidth}{1em} 
        \setlength{\itemindent}{0em} 
    }
}
{
    \end{list}
}
\crefname{property}{Property}{Properties}

\usepackage[foot]{amsaddr}

\title[Balanced quasistatic evolutions of critical points in metric spaces]{Balanced quasistatic evolutions of critical points in metric spaces}
\author[S.~Almi]{Stefano Almi$^{1}$}
\email{stefano.almi@unina.it}

\author[M.~Fornasier]{Massimo Fornasier$^{2, 3, 4}$}
\email{massimo.fornasier@cit.tum.de}

\author[J.~Klemenc]{Jona Klemenc$^{2,3}$}
\email{jona.klemenc@tum.de}
\author[A.~Scagliotti]{Alessandro Scagliotti$^{2,3}$}
\email{scag@ma.tum.de}

\address{$^1$ Dipartimento di Matematica e Applicazioni ``R. Caccioppoli", Universit\`a di Napoli Federico II, via Cintia, 80126 Napoli, Italy}
\address{$^2$CIT School, Technical University of Munich, Garching bei M\"unchen, Germany.}
\address{$^3$Munich Center for Machine Learning (MCML), Munich, Germany.}
\address{$^4$Munich Data Science Institute (MDSI), Technical University of Munich, Garching bei M\"unchen, Germany}

\begin{document}

\begin{abstract}
Quasistatic evolutions of critical points of time-dependent energies exhibit piecewise smooth behavior, making them useful for modeling continuum mechanics phenomena like elastic-plasticity and fracture. Traditionally, such evolutions have been derived as vanishing viscosity and inertia limits, leading to balanced viscosity solutions. However, for nonconvex energies, these constructions have been realized in Euclidean spaces and assume non-degenerate critical points. 
In this paper, we take a different approach by decoupling the time scales of the energy evolution and of the transition to equilibria.
Namely, starting from an equilibrium configuration, we let the energy evolve, while keeping frozen the system state; then, we update the state by freezing the energy, while letting the system transit via gradient flow or an approximation of it (e.g., minimizing movement or backward differentiation schemes). 
This approach has several advantages. It aligns with the physical principle that systems transit through energy-minimizing steady states. It is also fully constructive and computationally implementable, with physical and computational costs governed by appropriate action functionals. Additionally, our analysis is simpler and more general than previous formulations in the literature, as it does not require non-degenerate critical points. Finally, this approach extends to evolutions in locally compact metric path spaces, and our axiomatic presentation allows for various realizations.
\end{abstract}

\maketitle

\section{Introduction}
Quasistatic evolutions of critical points driven by time-dependent energies are characterized by piecewise smooth behaviors. For this reason, they have been used to model time-dependent phenomena in continuum mechanics, such as linearly elastic perfect plasticity or cohesive and brittle fracture \cite{dalmaso2006quasistatic,dalmaso2008vanishing,efendiev2006rate,fiaschi2009vanishing}. 
Traditionally, proposed constructions have been mostly relying on a vanishing  viscosity and inertia limit, yielding solutions characterized by an energy balance, known as \emph{balanced viscosity solutions} \cite{Zanini2007,  Agostiniani_2012,AR17, Scilla_2018_2, Scilla_2018,Scilla_2019} (see Section \ref{sec:relatedwork} for more details);
these derivations for quite general  nonconvex energies have been developed in Euclidean spaces and under the restrictive assumption of non-degeneracy of critical points of the energy. 
 In such constructions, two time-scales coexist: One corresponding to the evolution of the driving energy, and the other corresponding to the aspiration of the system towards equilibria.

In this paper, we follow an alternative route, which disentangles the two time-scales.
Starting from an equilibrium configuration, we define the \emph{discrete quasistatic evolution} by letting the energy evolve, while keeping frozen the system state. After a small amount of time $\delta>0$, the state is in general out of equilibrium. We amend this situation by freezing the energy, while letting the system transit via gradient flow---or some discrete approximation of it, like minimizing movement or backward differentiation scheme. Afterwards, we iterate this procedure for the whole evolution horizon, obtaining a piecewise constant curve $t\mapsto \eta^\delta (t)$.
For vanishing time discretizations $\delta \to 0$, it turns out that the limiting trajectories indeed enjoy the same energy balance properties as balanced viscosity solutions.

There are multiple advantages of our approach, which was introduced and explored first for time-evolving constraints in \cite{ACFS17} (see \cite{Almi-Belz-AMPA, Almi-Belz-Negri, Almi-Negri2020, Bourdin, Knees-NegriM3AS} for some applications to phase-field fracture evolution). First of all, it follows the physical principle that a system moves through action-minimizing transitions towards a steady state. 
Secondly, it allows us to effortlessly combine a discrete time scale for the system transition---e.g., when considering the minimizing movement scheme---with a continuous time scale for the energy evolution.
Third, the separation of the time scales leads to a disentanglement of the technical-proving challenges shared between previous approaches and ours: Taking \cite{AR17} as a seminal example, the two careful arguments concerning functional compactness and trajectory surgery coexist in the proof of \cite[Proposition 4.1 and Proposition 4.5]{AR17}. In comparison, when deriving the limiting curve, we retrieve compactness by relying on the properties of the system transitions in a purely axiomatic way. 
After having applied the functional compactness arguments, we establish the validity of the axioms by trajectory surgery, in a separate step. The resulting framework leaves open the possibility of further system transitions beyond the three which we have studied here, i.e., gradient flow, minimizing movement and backward differentiation scheme (see below for more details).
Lastly, our construction immediately suggests a numerical implementation approach, whereas constructions through viscosity and inertia limits face the numerical difficulties associated with exploding velocity fields.

Let us now introduce the results of the paper more formally.
In the setting of this paper, we consider a locally compact metric path space $(\Xx, d)$, and a sufficiently well-behaved energy $E \colon [0, T] \times \Xx \to \R$ in charge of driving the system. We seek to construct a \emph{balanced quasistatic evolution} $\hat \eta\colon [0, T] \to \Xx$. Being a ``quasistatic evolution'', we demand that 
\begin{equation}\label{eq:crit_point_evol}
    |\partial E_t|\big(\hat \eta(t)\big) = 0 \quad \mbox{for every } t\in [0,T] \setminus J,
\end{equation}
where $J=\{t \in [0,T]: \hat \eta^+(t) \neq \hat \eta^-(t)\}$ is the countable jump set of $\hat \eta$.
Moreover, the property of being ``balanced'' refers to complying with the energy balance
\begin{equation}\label{eq:e_b_0}
     E_t(\hat\eta^+(t)) - E_s(\hat\eta^-(s)) = \int_s^t \partial_t E_\tau(\hat\eta(\tau)) \diff \tau - \bar \mu([s, t]), \quad \forall s < t \in (0,T),
\end{equation}
for a suitable positive Radon measure $\mu \in \mathcal M^+([0,T])$ supported on $J$, together with the fact that, for each $t \in J$, we have that
\begin{equation}\label{eq:jump_characterization_intro}
    \mu(\{t\}) = E_t\big(\hat \eta^-(t)\big) - E_t\big(\hat \eta^+(t)\big) = c_t\big(\hat \eta^-(t), \hat \eta^+(t)\big),
\end{equation}
Here, $c_t\colon \Xx \times \Xx \to \R$ is defined by minimizing energy-dissipation integrals, and is tailored on the energy landscape at time $t$ and on the system transition rule. Furthermore, \eqref{eq:jump_characterization_intro} enforces that the balanced quasistatic evolution $\hat \eta$ does not jump over energy barriers.
On this point, the work \cite{ACFS17} provided a construction of quasistatic evolutions in presence of a fixed energy and a time-varying linear constraint. Unfortunately, in that case it was not possible to achieve the energy balance. Instead, the authors established an inequality of the form
\begin{equation}\label{eq:enineq}
    E_t \big(\hat \eta^+(t) \big)  \leq E_s \big(\hat \eta^-(s) \big) +
    \int_s^t \partial_t E_\tau \big(\hat \eta(\tau) \big) \diff \tau
\end{equation}
for every $0< s\leq t < T$.
In contrast, we bring to completion the program outlined above, even in the presence of degenerate critical points.
For such degenerate critical points not to interfere with the limiting construction,
we have to set this limiting construction in a quotient space $\mathcal{X}$ of the path space $[0, T] \times \Xx$,
where we identify points in the same connected component of the set of critical points of $E_t$,
with the limiting curve $\hat \eta$ taking values in $\mathcal{X}$.
In the simplified theorem, we assume \cref{ass:conn_comp,ass:der_time,ass:metric_space,ass:path+Lipsch,ass:PL+time_der_slope,ass:reg_E,ass:unif_coerc,ass:t_der_quotient}, which ensures that all the functions $E$, $|\partial E|$, $\partial_t E$ and $c$ factor through $\mathcal{X}$; we denote the resulting maps as
$\hat E$, $|\partial \hat E|$, $\partial_t \hat E$ and $\hat c$ (\cref{fig:quotient-space-diagram}).
\begin{figure}[htb]
    \centering
\begin{tikzcd}
    & {[0, T] \times \Xx} \arrow[d, "q" description, two heads] \arrow[rrd, "{E, |\partial E|, \partial_t E, c}" description] &  &    \\
{[0, T]} \arrow[r, "\hat{\eta}^{\delta_n}"] \arrow[ru, "\textrm{id} \times \eta^{\delta_n}"] & \mathcal{X} \arrow[rr, "{\hat E, |\partial \hat E|, \partial_t \hat E,\hat c}" description, dashed]                     &  & \R
\end{tikzcd}
    \caption{The spaces involved in the definition of the quasistatic evolution $\hat \eta$.
    If $E$ does not have degenerate critical points, both $q$ is the identitity and
    $\hat \eta$ takes values in $[0, T] \times \Xx$.}
    \label{fig:quotient-space-diagram}
\end{figure}
{
\renewcommand{\thetheorem}{\arabic{theorem}~(simplified)}
\begin{theorem}
    Let us assume \cref{ass:conn_comp,ass:der_time,ass:metric_space,ass:path+Lipsch,ass:PL+time_der_slope,ass:reg_E,ass:unif_coerc,ass:t_der_quotient}.
    Furthermore, let the family of mappings $\bar{\omega}_t\colon \Xx \to \Xx$ indexed by $t \in [0, T]$ be the transition rule as in \cref{def:discr_evol}, corresponding to an action $c_t$ and
    complying with \cref{ax:decreasing,ax:gradient,ax:action_triangular_time_stable}.
    Then, for any positive vanishing sequence $(\delta_n)_{n \in \NN}$ and for the corresponding discrete quasistatic evolutions
    $\eta^{\delta_n}$, we can---without relabeling---extract a subsequence
    such that:
    \begin{itemize}
        \item The compositions {$q \circ (\mathrm{id} \times \eta^{\delta_n}) \eqqcolon \hat \eta^{\delta_n}$} converge pointwise to a piecewise continuous limiting curve $\hat\eta\colon [0, T] \to \X$.
        \item There exists a positive Radon measure $\bar\mu \in \mathcal{M}([0, T])$ such that $\mu^{\delta^n} \weak^* \bar\mu$, where $\mu^{\delta_n}$ is as defined in \cref{eq:def_mu_delta}. Moreover, the set $J \coloneqq \mathrm{supp}\, \bar\mu$ consists of countably many points.
        \item The left and right limits $\hat\eta^-(t)$ and $\hat\eta^+(t)$ of $\hat\eta$ exist for every $t \in (0, T)$, and so do the limits $\hat\eta^+(0)$ and $\hat\eta^-(T)$.
        \item The limiting curve $\hat\eta$ fulfills, for all $0 \leq s < t \leq T$, the energy balance identity
        \begin{equation*}
             \hat E_t(\hat\eta^+(t)) - \hat E_s(\hat\eta^-(s)) = \int_s^t \partial_t \hat E(\hat\eta(\tau)) \diff \tau - \bar \mu([s, t]).
        \end{equation*}
        \item $|\partial \hat E_t|\big(\hat\eta(t)\big) = 0$ for all $t \in [0, T]$.
        \item The limiting curve $\hat\eta$ is continuous in $[0, T] \setminus J$, and for every $t \in J$ we have that 
        \begin{equation*}
            \hat E_t(\hat\eta^-(t)) - \hat E_t(\hat\eta^+(t)) = \bar \mu(\{t\}) = \hat c_t\big(\hat\eta^-(t), \hat\eta^+(t)\big).
        \end{equation*}
    \end{itemize}
\end{theorem}
}
As is apparent in the theorem, we are pursuing an axiomatic construction. Comparing \cref{ax:decreasing,ax:gradient} with \cref{ax:action_triangular_time_stable}, the latter looks technically more involved---indeed, it distills the aforementioned argument of trajectory surgery---, constituting a potential obstruction in finding complying transition rules. To remedy this difficulty, we propose the notion of \emph{actions generated by curves} in \cref{def:generated_by_curves} and we prove that it suffices for conforming to \cref{ax:action_triangular_time_stable}. Moreover, we provide three examples of transition rules and their corresponding actions, namely the gradient flow (GF), the minimizing movement scheme (MMS) and the backward differentiation formula (BDF2). While the variational nature of the gradient flow is well established \cite{greenbook} and indeed used, e.g., in \cite{AR17}, the formulation of the action for MMS has been introduced later in \cite{Scilla_2018}. Finally, we propose here an action for the BDF2 scheme that is, to the best of our knowledge, novel. 
Our analysis of the BDF2 is inspired by~\cite{Matthes-Plazotta}, which is the first work describing the BDF2 as a numerical algorithm for the approximation of gradient flows in a metric space.  By proving that those three actions are actually generated by curves, we guarantee that GF, MMS and BDF2 are suitable for constructing balanced quasistatic evolutions.
{However, we would like to emphasize the distinction between these actions. While the action associated with the GF (gradient flow) models a time-continuous transition, those of MMS and BDF2 are based on purely algorithmic (discrete-time) iterations, see Figure \ref{fig:quasistatic}. Therefore, we understand the GF action as ``physical" and those of MMS and BDF2 as ``computational". 
}

Finally, in Section~\ref{sec:num}, we illustrate our theoretical findings with a simple numerical experiment on a model of elastic rod fracture. Despite its simplicity, the purpose of this example is to demonstrate the full computability of our approach and its correspondence to physical principles. 

\begin{figure}[H]
    \centering
    \includegraphics[width=0.7\textwidth]{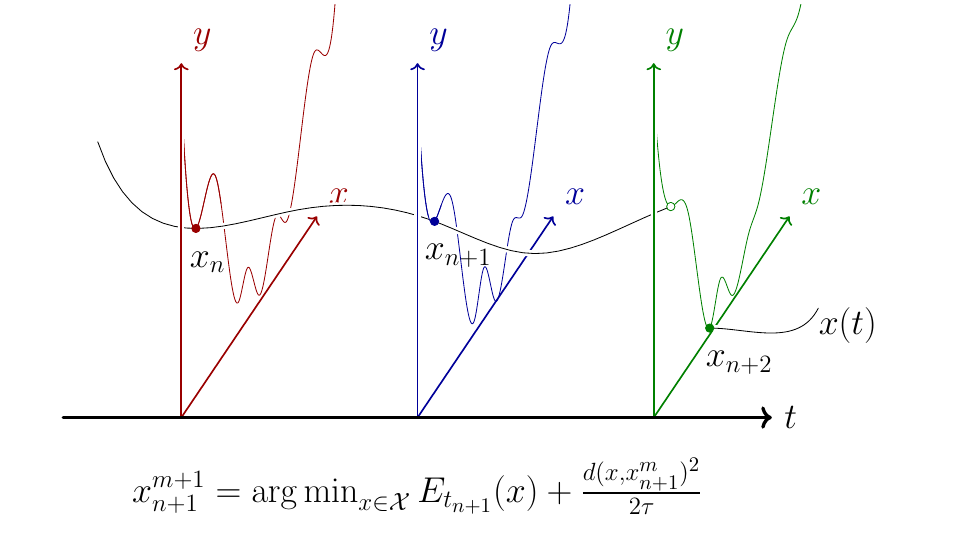}
    \caption{We display here a scheme explaining how a quasistatic evolution of critical point is constructed through MMS transition rule.}
    \label{fig:quasistatic}
\end{figure}

{\subsection{Related work} \label{sec:relatedwork}

Our results are inspired by the research initiated by C.~Zanini in \cite{Zanini2007}. In that work, the author constructs unique evolutions of critical points of smooth, nonconvex energy functions, where the critical points satisfy certain transversality conditions. These evolutions are obtained as limits of vanishing viscosity solutions to the singularly perturbed gradient flow equation:  
\[
\varepsilon \dot{x}(t) = - \nabla_x E_t(x(t)).
\]
The construction of these solutions involves patching together smooth branches of equilibrium solutions (slow dynamics) with heteroclinic solutions of the gradient flow (fast dynamics).

Building upon the concept of balanced viscosity solutions introduced in the context of rate-independent evolutions \cite{MR2525194,MR2887927,MR3531671}, V.~Agostiniani and R.~Rossi further proved in \cite{AR17} that vanishing viscosity evolutions of isolated critical points fulfill an energy balance equations as in \eqref{eq:e_b_0}.
We also report that in \cite{fiaschi2009vanishing} quasitatic evolutions were constructed as stochastic processes obtained as limits of vanishing viscosity solutions in the framework Young measures.

In \cite{Scilla_2018} G. Scilla and F. Solombrino perform a convergence analysis of a discrete-in-time minimization scheme approximating a finite dimensional singularly perturbed gradient flow. They allow for different scalings between the viscosity parameter $\varepsilon$ and the time scale $\tau$. When the ratio $\varepsilon/\tau$ diverges, the authors  prove the convergence of this scheme to balanced viscosity solutions of the quasistatic evolution problem obtained as a formal limit for $\varepsilon \to 0$ of the gradient flow. 
They also characterize the limit evolution corresponding to an asymptotically finite ratio between the scales, and they derived the expression of the transition action $c_t$ corresponding to the minimizing movement scheme.

Although transversality conditions for critical points as in \cite{Zanini2007} are known to be generically fulfilled \cite{ARS_2015}, they exclude interesting situations, often appearing in the application (for instance, stationary solutions to \eqref{eq:crit_point_evol} whose stability changes depending on the time $t$, usually giving rise to bifurcation of other branches of critical points). In \cite{Scilla_2018_2}, G. Scilla and F. Solombrino investigate the phenomenon of delayed loss of stability in singularly perturbed gradient flows. Their study focuses on the relaxation of transversality conditions imposed on critical points, exploring the consequences of their removal. 

The papers \cite{Agostiniani_2012,Scilla_2019} also study the vanishing inertia and viscosity limit of a second order system set, driven by a possibly nonconvex time-dependent energy satisfying very general assumptions. By means of a variational approach, they show that the solutions of the singularly perturbed problem converge to a curve of stationary points of the energy and characterize the behavior of the limit evolution at jump times. At those times, the left and right limits of the evolution are connected by a finite number of heteroclinic solutions to the unscaled equation. 

To position our contribution within this line of research, we emphasize that the aforementioned works \cite{Zanini2007, AR17, Scilla_2018, Scilla_2018_2} have exclusively studied the vanishing viscosity limit of singularly perturbed gradient flows in Euclidean spaces. In contrast, this paper extends the construction of balanced quasistatic evolutions of critical points to metric spaces, without imposing transversality conditions or assuming isolated critical points\footnote{From private communications by G.~Savaré,  we are aware that V.~Agostiniani, R.~Rossi, and G.~Savaré have been working for a while on generalizing the vanishing viscosity limit of solutions to singularly perturbed gradient flows in metric spaces. However, our approach differs significantly, as it relies on a separation of scales and limits of discrete-time evolutions.}.

It is now relevant to highlight another related line of research, initiated by the seminal works \cite{efendiev2006rate,MR2525194, MR2887927, MR3531671} on the construction of rate-independent BV-evolutions. The key distinction between this framework and both our setting and the previously mentioned works \cite{Zanini2007, AR17, Scilla_2018, Scilla_2018_2} lies in the presence of a 1-homogeneous term in the energy functional, which introduces a dissipation term into the dynamics. In their pioneering work \cite{MR2525194}, A. Mielke, R. Rossi, and G. Savaré construct evolutions as limits of viscous regularizations of solutions. Using similar techniques as in \cite{efendiev2006rate}, they introduce the concept of parametrized metric solutions for rate-independent systems, which are absolutely continuous mappings from a parameter interval into an extended state space. In this framework, jumps are interpreted as generalized gradient flows, during which time remains constant, ultimately leading to BV-solutions. Notably, their formulation is developed entirely within the setting of metric spaces. In the follow-up paper \cite{MR2887927} the same authors have revisited and aimed to clarify the key features of balanced viscosity solutions. 
The work is conducted in a finite-dimensional setting but employs two distinct convex functionals on the derivative of the solution: a 1-homogeneous functional, which is the standard choice for rate-independent systems {(cf.~\cite{Roubicek-Mielke} and references therein),} and a superlinear functional that introduces viscous regularization.  
This approach offers significantly greater generality compared to the metric framework, which imposes the use of the same norm for both the rate-independent term and the quadratic regularization. As a result, a more refined analysis is required to understand the behavior of solutions during jumps. In the concluding work \cite{MR3531671} the authors extend their results to the infinite-dimensional setting (Banach spaces). In this line of research as well, the approach involves constructing solutions as limits of a vanishing viscosity setting, where two distinct time scales coexist: one governing the evolution of critical points and the other regulating transitions during jumps. 
{We further refer to the recent contributions~\cite{Riva-Scilla-Solombrino-1, Riva-Scilla-Solombrino-2} for a similar research plan involving vanishing inertia and viscosity, where inertial effects appear in the energetic characterization of jump points.}

To clarify how the present work is collocated in relation to \cite{MR2525194, MR2887927, MR3531671} and the very vast related literature on rate-independent systems and doubly nonlinear equations (which we purposely do not report here), we reiterate the key differences:  
1. We do not consider energy functionals with 1-homogeneous terms, nor do we seek BV rate-independent solutions.  
2. Our construction explicitly separates/decouples the time scales of critical point evolution and jump transitions.  
3. Our formulation is applicable in metric spaces.

We conclude this review of related work by noting that the separation/decoupling of time scales through local-in-time transitions was first explored in \cite{ACFS17}. In that model, the energy was fixed, while the driving force of the dynamics was introduced via a time-dependent constraint. Additionally, the analysis was conducted in Euclidean spaces setting and only an energy inequality \eqref{eq:enineq} was obtained.
}

\subsection{Plan of the Paper} 
In \cref{sec:assumptions}, we collect the assumptions that are needed in the results proved in the paper. \\
In \cref{sec:quasist_constr}, we propose an axiomatic construction of discrete quasistatic evolutions. Namely, the main result of the section (\cref{thm:traj_convergence}) states that, if such piecewise constant trajectories are defined through a transition rule that complies with \cref{ax:decreasing,ax:gradient,ax:limits_traject}, then the discrete quasistatic evolutions converge pointwisely---up to the extraction of a subsequence---to a limiting quasistatic curve taking values in a quotient space of $[0,T]\times \Xx$.
What is missing at this level, though, is a characterization of the energy jumps in the discontinuity points of the trajectory in terms of a variational action.\\
In \cref{sec:actions_abstract}, we tackle this point by introducing the notion of \emph{transition rule compatible with an action} (see \cref{def:action}) and of \emph{action generated by curves} (see \cref{def:generated_by_curves}).
Namely, we first show that \cref{ax:action_triangular_time_stable} implies \cref{ax:limits_traject}, and that we can relate energy jumps to the values of the action (see \cref{lemma:action_implies_traj}).
Then, we prove that, if a transition rule is compatible with an action generated by curves, then it satisfies \cref{ax:action_triangular_time_stable} (see \cref{lemma:good_curves_fulfill_continuous_triangular_inequality}).
Finally, we are in a position to prove the main result of the paper (see \cref{thm:main_result_complete}).\\
In \cref{sec:examples_of_evolution_rules_elaborate} we show that transition rules obtained using gradient flow, minimizing movements and the BDF2 scheme are compatible with actions generated by curves (see \cref{subsec:gradient_flow,subsec:elaborate_mms,subsec:elaborate_BDF2}). \\
Finally, in \cref{sec:num} we present  simple numerical experiments to illustrate our theoretical findings, where we simulate the breaking of an elastic rod.

\section{General Notations and Assumptions} \label{sec:assumptions}

In this section, we collect the assumptions that are required throughout the paper. In each statement we shall explicitly list the needed assumptions. To ease the reader, we adopt similar notations as in \cite{AR17}.
We set our analysis on a metric space $(\Xx, d)$. We first state the basic structure that we need on $(\Xx, d)$.
\begin{assumption} \label{ass:metric_space}
    The space $(\Xx, d)$ is a locally compact metric space.
\end{assumption}

We make the following assumptions on the time-varying energy $E\colon [0,T]\times \Xx \to \R$ that drives the evolution.

\begin{assumption}\label{ass:reg_E}
    The function $E\colon [0,T]\times \Xx \to \R$ is continuous. Moreover, it is differentiable in the first variable, for every $t\in [0,T]$ and for every $x\in \Xx$, and the derivative $\partial_t E_\cdot \colon [0,T]\times \Xx\to \R$ is continuous.
    Finally, for every $t\in[0,T]$ and $x\in\Xx$ we consider the slope
    \begin{equation*}
        |\partial E_t|(x) \coloneqq 
        \limsup_{x'\to x} \frac{\big( E_t(x) - E_t(x') \big)^+}{d(x,x')},
    \end{equation*}
    and we assume that $|\partial E_\cdot|\colon [0,T]\times \Xx \to \R$ is a continuous function.
\end{assumption}
\begin{remark}\label{rmk:slope_strong_upper_gradient}
    Under the continuity assumption and using local compactness, the slope coincides with the local Lipschitz constant
    \begin{equation*}
        \limsup_{x'\to x} \frac{\big| E_t(x) - E_t(x') \big|}{d(x,x')}.
    \end{equation*}
    As the local Lipschitz constant is a strong upper gradient (as defined in~\cite[Def. 1.2.1]{greenbook}), under the preceding assumption, the slope is as well.
\end{remark}

\begin{assumption}\label{ass:der_time}
    There exist positive constants $C_1,C_2>0$  such that
    \begin{equation*}
        |\partial_t E_t(x)| \leq C_1 E_t(x) + C_2
    \end{equation*}
    for every $(t,x)\in [0,T]\times \Xx$.
\end{assumption}

From \cref{ass:der_time} it follows that $E_t(x)\geq -C_2/C_1$ for every $(t,x)\in [0,T]\times \Xx$, i.e., $E\colon[0,T] \times \Xx\to \R$ is bounded from below. Without loss of generality, we will assume throughout the paper that $E$ is non-negative.

\begin{assumption}\label{ass:unif_coerc}
    We set $G\colon \Xx\to \R$ as  $G(x)\coloneqq \sup_{t\in[0,T]} |E_t(x)|$ for every $x\in\Xx$, and we require that $G$ is coercive.  Namely, the sublevel sets $\{x \in \Xx \mid G(x)\leq C\}$ are compact in $\Xx$ for every $C\in \mathbb{R}$.
\end{assumption}

\begin{remark}\label{rmk:coerc_Et}
    From \cref{ass:reg_E,ass:unif_coerc,ass:der_time} it follows that $E_t:\Xx\to\R$ is uniformly coercive in the time variable. Indeed, if we consider the set $\{x\in\Xx \mid E_t(x)\leq C\}$ for $C\geq0$, then from \cref{ass:der_time} and the Gr\"onwall Lemma we get
    $E_s(x)\leq (C+C_2|s-t|)e^{C_1|s-t|}$ for every $s\in[0,T]$ and for every $x\in\Xx$ satisfying $E_t(x)\leq C$.
    Hence, for every $t\in[0,T]$ we deduce the inclusion
    \begin{equation*}
        \{x\in\Xx \mid  E_t(x)\leq C\} \subset
        \left\{
        x\in\Xx \mid G(x)\leq (C+C_2T)e^{C_1T}
        \right\},
    \end{equation*}
    where the larger set (which does not depend on $t$) is compact by \cref{ass:unif_coerc}.
\end{remark}

\begin{assumption} \label{ass:conn_comp}
    For every $t\in [0,T]$, the set $\Crit_t\coloneqq \{x\in \Xx \mid |\partial E_t|(x) =0\}$ can be expressed as the disjoint union of well-separated compacts, each of them being path-connected with rectifiable paths.
    Moreover, for every path-connected component $\Crit_t' \subset \Crit_t$, $E_t\colon \Xx\to\R$ is constant on $\Crit_t'$.
\end{assumption}

\begin{remark}
    We recall that in general the constancy of $E_t$ on the connected components of $\Crit_t$ may fail (see \cite{W35}).
    However, this property is implied, e.g., by the fact that any two points in a connected component can be joined by a path $\eta \in AC([0,1],\Xx)$ (see \cite[Section~1.1]{greenbook} for this notion) such that $|\partial E_t|\big( \eta(s) \big) =0$ for every $s \in [0,1]$.
    Moreover, also from \cref{ass:PL+time_der_slope} below, it follows that $E_t$ is constant on the path-connected components of $\Crit_t$ (see \cref{rmk:constant_components}).
\end{remark}

\begin{assumption} \label{ass:PL+time_der_slope}
    We require that, for every $t\in [0,T]$ and for every $x \in \Xx$ with $|\partial E_t|(x)=0$ there exists a neighborhood $U\ni x$ and a function $\varepsilon_x\colon[0,+\infty)\to [0,+\infty)$ such that
    \begin{equation} \label{eq:liminf_ineq_relaxed}
        E_t(x') - E_t(x) \geq - \varepsilon_x\big( d(x,x') \big) |\partial E_t|(x')
    \end{equation}
    for every $x'\in U$, where $\varepsilon_x(0) = \lim_{s\to 0^+} \varepsilon_x(s) = 0$.\\
    Moreover, we assume that for a fixed $u \in \Xx$, the function $t \mapsto |\partial E_t|(x)$ is Lipschitz continuous on $[0, T]$, locally uniformly w.r.t. $x$.
\end{assumption}

\begin{remark}
    We report that the inequality \eqref{eq:liminf_ineq_relaxed} in \cref{ass:PL+time_der_slope} is a reformulation of the more classical condition
    \begin{equation*}
        \liminf_{x'\to x} \frac{E_t(x') - E_t(x)}{|\partial E_t|(x')} \geq 0
    \end{equation*}
    required for those $x$ such that $|\partial E_t|(x)=0$ (see, e.g., \cite[Assumption~E4]{AR17} and\cite[Assumption~F4]{Scilla_2018}), under the hypothesis of \emph{isolated critical points}.
    In our setting, the ratio written above is not suitable, as it is not defined whenever $x$ lies in the interior of the set where $|\partial E_t|$ vanishes.
\end{remark}

\begin{remark} \label{rmk:constant_components}
    From \cref{ass:reg_E,ass:PL+time_der_slope} it follows that for every path-connected component $\Crit_t' \subset \Crit_t$, $E_t\colon \Xx\to\R$ is constant on $\Crit_t'$.
    To see this, we show that, for every $c\in \R$, we have that both $A_c\coloneqq \{x \in \Crit_t' \mid  E_{t}  (x) \geq c \}$ and $\bar A_c \coloneqq\{x \in \Crit_t' \mid  E_{t}  (x) < c \}$ are open in the topology induced by $(\Xx,d)$ on $\Crit_t'$.
    On the one hand, let $x \in A_c$, then by virtue of \cref{ass:PL+time_der_slope}, there exists an open ball $B_\rho(x)$ such that  $E_{t} (x')\geq E_{t} (x) \geq c$ for every $x' \in \Crit_t'\cap B_\rho(x)$ (here we combined \cref{eq:liminf_ineq_relaxed} with the fact that $x'$ is a critical point).
    On the other hand, if $x \in \bar A_c$, then by continuity of $E_t$ we deduce the existence of an open ball $B_\rho(x)$ such that $ E_{t}  (x')< c$ for every $x' \in \Crit_t'\cap B_\rho(x)$.
    Since $\Crit_t' = A_c \cup \bar A_c$ for every $c\in\R$, we deduce that either $\Crit_t' = A_c$ and $\bar A_c= \emptyset$, or $\Crit_t' = \bar A_c$ and $ A_c= \emptyset$.
    Therefore, we have that, for every $x\in \Crit_t'$, $E_t(x) = \sup\{c\in\R\mid A_c \neq \emptyset\}$.
\end{remark}

It turns out that the time derivative $\partial_t E_t$ is constant on the connected components of critical points, for every but at most countably many exceptional instants $t\in [0,T]$ (see \cref{lemma:time_derivative_factors}). However, to show that, we need to strengthen the condition formulated in \cref{eq:liminf_ineq_relaxed}.

{
\renewcommand{\theassumption}{6'}
\begin{assumption} \label{ass:strengthening_6}
        We require that, for every $t\in [0,T]$ and for every $x \in \Xx$ with $|\partial E_t|(x)=0$ there exists a neighborhood $U\ni x$ and a function $\varepsilon_x\colon[0,+\infty)\to [0,+\infty)$ such that
    \begin{equation} \label{eq:liminf_ineq_stronger}
        |E_t(x') - E_t(x)| \leq  \varepsilon_x\big( d(x,x') \big) |\partial E_t|(x')
    \end{equation}
    for every $x'\in U$, where $\varepsilon_x(0) = \lim_{s\to 0^+} \varepsilon_x(s) = 0$.\\
    Moreover, we assume that the function $t \mapsto |\partial E_t|(x)$ is Lipschitz continuous on $[0, T]$, locally uniformly with respect to $x \in \Xx$.
\end{assumption}
\addtocounter{assumption}{-1}
}

\begin{remark}
    We observe that the inequality \eqref{eq:liminf_ineq_stronger} and its one-sided version \eqref{eq:liminf_ineq_relaxed} hold if $E_t$ complies with the local Polyak-{\L}ojasiewicz condition and if $|\partial E_t|$ is continuous (as prescribed in \cref{ass:reg_E}).
\end{remark}

The next hypothesis will be used in the last part of the paper.

\begin{assumption} \label{ass:path+Lipsch}
    $(X,d)$ is a path space, i.e.,
    \begin{equation*}
        d(x,x') = \inf\left\{
        \int_0^1 |\dot \eta (s)| \diff s \, \mid \,
        \eta \in AC([0,1],\Xx), \, \eta(0)=x, \, \eta(1)=x'
        \right\}.
    \end{equation*}
    Moreover, we require that $|\partial E_t| \colon \Xx \to \R$ is Lipschitz continuous, uniformly as $t$ varies in $[0,T]$, i.e., there exists $L>0$ such that 
    \begin{equation*}
        \big| |\partial E_t|(x) - |\partial E_t|(x') \big| \leq L\, d(x,x')
    \end{equation*}
    for every $t\in[0,T]$  and every $x, x' \in \mathbb{X}$.
\end{assumption}
\begin{remark}\label{rmk:lipschitz_tau_relation}
    In \cref{subsec:elaborate_mms,subsec:elaborate_BDF2}, we describe discrete minimizing schemes with a step size $\tau$. Whenever we work with \cref{ass:path+Lipsch}, we will assume that $\tau \leq \frac{1}{L}$, where $L$ is the Lipschitz constant assumed in \cref{ass:path+Lipsch}.
\end{remark}

Finally, we need the following assumption on the time derivative of the driving energy for proving that the measure involved in the energy balance is purely atomic.
Before proceeding, we introduce the quotient space $\X$ as follows:
\begin{equation}\label{eq:def_quot_space}
    \X \coloneqq \big([0,T]\times \Xx \big)/\sim,
\end{equation}
where $\sim$ is the equivalence relation on $[0,T]\times\Xx$ given by
\begin{equation}\label{eq:def_equiv_rel}
    \begin{split}
        (t_1,x_1)\sim (t_2,x_2) \iff &
        t_1=t_2 \mbox{ and $x_1,x_2$ belong to the same} \\ & \mbox{path-connected component of }\{x \in \Xx \mid |\partial E_{t_1}|(x) = 0\}.
    \end{split}
\end{equation}
We equip the space $\X$ with the quotient topology, and we denote with $q\colon [0,T] \times \Xx \to\X$ the quotient map. 
We recall that the quotient topology on $\X$ is the strongest that makes the mapping $q$ continuous.
Finally, we use the notation $[(t,x)]$ to describe the elements of $\X$, i.e., the equivalence classes induced by the relation \eqref{eq:def_equiv_rel}. 

\begin{assumption} \label{ass:t_der_quotient}
    For every $[(s,x)] \in \X$ we have that $\partial_t E_s(x_1) = \partial_t E_s(x_2)$, for every $(s,x_1),(s,x_2)\in \X$. Hence, we can define $\partial_t \hat E\colon \X\to \R$ as $\partial_t \hat E\big( [(s,x)] \big)\coloneqq \partial_t E_s(x)$ for every $[(s,x)] \in \X$. Furthermore, we require that $\partial_t \hat E$ is continuous.
\end{assumption}

\begin{remark}
    In the case the set $\{x\in \Xx \mid |\partial E_t|(x) =0\}$ consists of isolated points for every $t\in [0,T]$, then \cref{ass:t_der_quotient} is implied by \cref{ass:reg_E}, as the quotient space $\X$ is homeomorphic to $[0,T]\times \Xx$.
\end{remark}

\section{Discrete quasistatic evolutions: axiomatic construction} \label{sec:quasist_constr}

In this section, we detail the properties that a family of piecewise constant curves should satisfy to be employed to retrieve (through a limiting argument) the solutions of \cref{eq:crit_point_evol}.
We begin by introducing the curves that we use in our construction.

\begin{definition}[Discrete quasi\-static evolution]\label{def:discr_evol}
    Let us assume \cref{ass:metric_space,ass:reg_E}.
    Here, we fix a family of mappings $\bar{\omega}_t\colon \Xx \to \Xx$ indexed by $t \in [0, T]$, which we call the \emph{transition rule}.
    Given $\delta>0$ and an initial state $x_0\in \Xx$ such that $|\partial E_0|(x_0)=0$, we construct a discrete quasistatic evolution $\eta^\delta\colon [0, T] \to \Xx$ as follows:
    \begin{itemize}
        \item We define $M\coloneqq \lfloor T/\delta \rfloor$ (or $M\coloneqq T/\delta - 1$, if $T/\delta$ is integer).
        \item We set $w_0 = x_0$.
        \item For $i=1,\ldots,M$, we assign $w_{i} = \bar \omega_{i\delta}(w_{i-1})$.
        \item For $i=0,\ldots,M$, for $t\in [i\delta, (i+1)\delta)$ (or $t\in [i\delta,T]$ in case $i=M$), we set $\eta^\delta(t)=w_i$.
    \end{itemize}
    We also introduce
    \begin{equation} \label{eq:def_J_delta}
        J^\delta\coloneqq \{ i\delta \mid i=1,\ldots,M \}.
    \end{equation}
\end{definition}
We observe that in \cref{def:discr_evol} we have not specified any rule for deriving  $\bar \omega_{t}$.
One of the main contributions of the present paper is to provide a list of axioms that such a scheme should satisfy to be used for the construction---through a limiting argument---of a balanced viscosity quasistatic solution (see \cite{AR17} for this notion).

In our arguments, the identities on the energy balance shall play a pivotal role.
Given $\delta>0$ and a discrete quasistatic evolution $\eta^\delta \colon [0,T]\to \Xx$ constructed according to \cref{def:discr_evol}, we define the function $\E_\cdot(\eta^\delta) \colon [0,T]\to\R$ as
\begin{equation}\label{eq:def_evolving_energ}
    t \mapsto \E_t(\eta^\delta)\coloneqq E_t(\eta^\delta(t))
\end{equation}
for every $t\in [0,T]$.
We observe that, owing to \cref{ass:reg_E}, the following limits exist:
\begin{align*}
    \E_t^+(\eta^\delta)\coloneqq \lim_{\tau\to t^+}  \E_\tau(\eta^\delta) = \lim_{\tau\to t^+}  E_\tau(\eta^\delta(\tau)), \\
    \E_t^-(\eta^\delta)\coloneqq \lim_{\tau\to t^-}  \E_\tau(\eta^\delta) = \lim_{\tau\to t^-}  E_\tau(\eta^\delta(\tau))
\end{align*}
for every $t\in (0,T)$.
Moreover, we define a priori $\E_0^-(\eta^\delta)\coloneqq \E_0(\eta^\delta)$ and $\E_T^+(\eta^\delta)\coloneqq \E_T(\eta^\delta)$.
We observe that by construction $\E_\cdot (\eta^\delta)$ is \emph{continuous from the right}.

\subsection{Energy balance}
In this part, we state the first axiom and we establish an energy balance for the discrete quasi\-static evolutions that fulfill it.
Moreover, we show at which extent the energy balance is preserved when we let $\delta$ tend to $0$.

\begin{axiom}\label{ax:decreasing}
    Let the family of mappings $\bar{\omega}_t\colon \Xx \to \Xx$ indexed by $t \in [0, T]$ be the transition rule as in \cref{def:discr_evol}.
    Then, we require that for every $x \in \Xx$ and for every $t \in [0, T]$, we have that $E_t\big(\bar \omega_t(x)\big) \leq E_t(x)$.
\end{axiom}
We define the non-negative Radon measure $\mu^\delta \in \mathcal{M}([0,T])$ as follows:
\begin{equation}\label{eq:def_mu_delta}
    \mu^\delta(B) \coloneqq \sum_{t\in J^\delta \cap B}
    \Big[  \E_{t}^-(\eta^\delta) -
        \E_{t}^+(\eta^\delta) \Big]
\end{equation}
for every Borel set $B\subset[0,T]$.
On one hand, owing to \cref{ax:decreasing}, we observe that
\begin{equation} \label{eq:meas_point}
    \mu^\delta(\{ t \}) = \E_{t}^-(\eta^\delta) -
    \E_{t}^+(\eta^\delta) \geq 0
\end{equation}
for every $t\in J^\delta$ (i.e., $t=i\delta$ with $i=1,\ldots,M$). On the other hand, $\mu(\{ t \})=0$ if $t\not \in J^\delta$.

We are now in a position to establish the energy balance identity for discrete quasistatic evolutions.

\begin{lemma}\label{prop:energy_balance}
    Let us assume \cref{ass:metric_space,ass:reg_E}. 
    Given $\delta>0$, let $\eta^\delta \colon [0,T]\to \Xx$ be a discrete quasistatic evolution constructed according to \cref{def:discr_evol}.
    Then, if \cref{ax:decreasing} is satisfied, for every $0\leq t_1 \leq t_2 \leq T$ the following identity holds:
    \begin{equation}\label{eq:ener_bal}
        \E_{t_2}^+( \eta^\delta) -
        \E_{t_1}^-( \eta^\delta) =
        \int_{t_1}^{t_2} \partial_t E_\tau \big(\eta^\delta(\tau)\big) \diff \tau - \mu^\delta([t_1,t_2]),
    \end{equation}
    where the non-negative measure $\mu^\delta$ is defined as in \cref{eq:def_mu_delta}.
\end{lemma}
\begin{proof}
    Owing to \cref{ass:reg_E} and recalling that $\eta^\delta$ is by construction piecewise constant, it follows that the function $t\mapsto \E_t(\eta^\delta)$ is of class $C^1$ on the interval $\big((i-1)\delta,i\delta\big)$ for every $i=1,\ldots,M$ (as well as on the very last piece $\big(M\delta,T\big)$).
    With an algebraic manipulation, we write
    \begin{equation}\label{eq:decomp_sum}
        \begin{split}
            \E_{t_2}^+(\eta^\delta) - \E_{t_1}^-(\eta^\delta) & =
            \Big( \E_{t_1}^+(\eta^\delta) - \E_{t_1}^- (\eta^\delta) \Big)
            +
            \Big( \E_{i_1\delta}^-(\eta^\delta) - \E_{t_1}^+ (\eta^\delta) \Big)                         \\& \quad 
            + \sum_{j=1}^{\ell} \Big( \E_{i_j\delta}^+(\eta^\delta) - \E_{i_j\delta}^- (\eta^\delta) \Big) 
            + \sum_{j=1}^{\ell-1}
            \Big( \E_{i_{j+1}\delta}^-(\eta^\delta) - \E_{i_j\delta}^+ (\eta^\delta) \Big)
            \\& \quad 
            + \Big( \E_{t_2}^-(\eta^\delta) - \E_{i_\ell \delta}^+ (\eta^\delta) \Big)
            +
            \Big( \E_{t_2}^+(\eta^\delta) - \E_{t_2}^- (\eta^\delta) \Big),
        \end{split}
    \end{equation}
    where $\{i_1<\ldots< i_\ell\}=J^\delta \cap (t_1,t_2)$. Using the Fundamental Theorem of Calculus and the measure $\mu^\delta$ defined in \cref{eq:def_mu_delta}, we rephrase \cref{eq:decomp_sum} as
    \begin{equation*}
        \begin{split}
            \E_{t_2}^+(\eta^\delta) - \E_{t_1}^-(\eta^\delta) & = -\mu^\delta(\{ t_1 \}) + \int_{t_1}^{i_1\delta} \partial_t E_\tau \big(\eta^\delta(\tau) \big) \diff \tau +
            \sum_{j=1}^{\ell}  -\mu^\delta(\{ i_j\delta \})
            \\& \qquad 
             +
            \sum_{j=1}^{\ell-1} \int_{i_j\delta}^{i_{i+1}\delta} \partial_t E_\tau \big(\eta^\delta(\tau)\big) \diff \tau 
            +
            \int_{i_\ell\delta}^{t_2} \partial_t E_\tau \big(\eta^\delta(\tau)\big) \diff \tau
            -\mu^\delta(\{ t_2 \}),
        \end{split}
    \end{equation*}
    which yields \cref{eq:ener_bal}.
\end{proof}

\begin{remark} \label{rmk:ener_bal_variants}
    In some circumstances, it can be useful to employ slight variations of \cref{eq:ener_bal}. For example, when dealing with the difference of the right limits $\E^+_{t_2}(\eta^\delta) - \E^+_{t_1}(\eta^\delta)$, we observe that
    \begin{equation*}
        \E^+_{t_2}(\eta^\delta) - \E^+_{t_1}(\eta^\delta) = \E^+_{t_2}(\eta^\delta) - \E^-_{t_1}(\eta^\delta) + \mu^\delta(\{t_1\}),
    \end{equation*}
    yielding
    \begin{equation*}
        \E_{t_2}^+( \eta^\delta) -
        \E_{t_1}^+( \eta^\delta) =
        \int_{t_1}^{t_2} \partial_t E_\tau \big(\eta^\delta(\tau)\big) \diff \tau - \mu^\delta((t_1,t_2])
    \end{equation*}
    for every $0 \leq t_1 \leq t_2 \leq T$. Similarly, we also have
    \begin{equation*}
        \E_{t_2}^-( \eta^\delta) -
        \E_{t_1}^-( \eta^\delta) =
        \int_{t_1}^{t_2} \partial_t E_\tau \big(\eta^\delta(\tau)\big) \diff \tau - \mu^\delta([t_1,t_2)),
    \end{equation*}
    and
    \begin{equation*}
        \E_{t_2}^-( \eta^\delta) -
        \E_{t_1}^+( \eta^\delta) =
        \int_{t_1}^{t_2} \partial_t E_\tau \big(\eta^\delta(\tau) \big) \diff \tau - \mu^\delta((t_1,t_2)).
    \end{equation*}
\end{remark}

We are now interested in studying the family $(\E_\cdot(\eta^{\delta_n}))_{n\geq 1}$ when $\delta_n\to 0$ as $n\to\infty$, under the hypothesis that the discrete quasistatic evolutions $(\eta^{\delta_n})_n$ are constructed using the same initial state $x_0\in \Xx$.
We begin by proving a boundedness result.

\begin{lemma}\label{lem:bound_energ_evol}
    Let us assume \cref{ass:metric_space,ass:reg_E,ass:der_time}.
    For every $\delta>0$, let $\eta^\delta \colon [0,T]\to \Xx$ be a discrete quasistatic evolution constructed according to \cref{def:discr_evol} and starting from $x_0\in \Xx$.
    Then, if \cref{ax:decreasing} is satisfied, we have
    \begin{equation} \label{eq:bound_energ_evol}
        \sup_{t\in[0,T]} \E_t(\eta^\delta) \leq E_0(x_0)e^{C_1(T+1)} +C_2 \frac{e^{C_1(T+1)}-1}{C_1},
    \end{equation}
    where $C_1,C_2>0$ are the constants that appear in \cref{ass:der_time}. In particular, the estimate is independent of $\delta$.
\end{lemma}
\begin{proof}
    Let us consider $t\in[0,\delta)$. Since $\eta^\delta$ is constant on this interval, we have that  $\E_\cdot (\eta^{\delta})$ is $C^1$ in $[0, \delta)$.
    Then, by virtue of \cref{ass:der_time}, we compute
    \begin{equation} \label{eq:der_energ_evolv}
        \left| \frac{\mathrm{d}}{\mathrm{d}t} \E_t(\eta^\delta) \right| =  \left| \partial_t E_t(x_0) \right| \leq C_1 E_t(x_0) + C_2 =
        C_1\E_t(\eta^\delta) + C_2,
    \end{equation}
    which yields
    \begin{equation}\label{eq:est_2_energ_evolv}
        \sup_{t\in[0,\delta)} \E_t(\eta^\delta) \leq \left( \E_0(\eta^\delta) + \delta C_2 \right) e^{\delta C_1},
    \end{equation}
    where we applied the Gr\"onwall inequality. Moreover, repeating the same argument on the interval $[\delta,2\delta)$, we deduce that
    \begin{equation}\label{eq:est_1_energ_evolv}
        \sup_{t\in[\delta,2\delta)} \E_t(\eta^\delta) \leq \left( \E_\delta(\eta^\delta) + \delta C_2 \right) e^{\delta C_1}.
    \end{equation}
    However, recalling \cref{def:discr_evol} and \cref{eq:def_evolving_energ}, by virtue of \cref{ax:decreasing}, we have that
    \begin{equation*}
        \sup_{t\in[0,\delta)} \E_t(\eta^\delta) \geq \E^-_\delta(\eta^\delta) \geq \E^+_\delta(\eta^\delta) = \E_\delta(\eta^\delta),
    \end{equation*}
    which, together with \cref{eq:est_1_energ_evolv,eq:est_2_energ_evolv}, implies
    \begin{equation*}
        \begin{split}
            \sup_{t\in[0,2\delta)} \E_t(\eta^\delta) & \leq  \left( \sup_{t\in[0,\delta)} \E_t(\eta^\delta) + C_2\delta \right) e^{C_1\delta} 
            \leq \E_0(\eta^\delta)e^{2\delta C_1} + \delta C_2 \sum_{j=1}^2 e^{j\delta C_1}.
        \end{split}
    \end{equation*}
    With an inductive argument, it follows that
    \begin{equation*}
        \sup_{t\in[0,T]} \E_t(\eta^\delta) \leq \E_0(\eta^\delta)e^{(M +1) \delta C_1 } + \delta C_2 \sum_{j=1}^{M+1} e^{j\delta C_1},
    \end{equation*}
    where $M$ is the integer defined in \cref{def:discr_evol}.  
    By simple algebraic manipulations and basic properties of the exponential functions, we conclude for~\eqref{eq:bound_energ_evol}.
\end{proof}

The boundedness of the energy along discrete quasistatic evolutions implies that the trajectories themselves are uniformly bounded, as we shall see below.

\begin{lemma}\label{lem:bound_traj}
    Let us assume \cref{ass:metric_space,ass:reg_E,ass:der_time,ass:unif_coerc}.
    For every $\delta>0$, let $\eta^\delta\colon [0,T]\to \Xx$ be a discrete quasistatic evolution constructed according to \cref{def:discr_evol} and starting from $x_0\in \Xx$.  
    If \cref{ax:decreasing} is satisfied, there exists a compact set $K \subset \Xx$ such that $\eta^\delta(t)\in K$ for every $t\in [0,T]$ and every $\delta >0$.
\end{lemma}
\begin{proof}
    From \cref{lem:bound_energ_evol}, it follows that there exists $K_1> 0$ independent of $\delta$ such that
    \begin{equation*}  
        \sup_{t\in [0,T]} \E_t(\eta^\delta) =
        \sup_{t\in [0,T]} E_t(\eta^\delta(t)) \leq K_1,
    \end{equation*}
    yielding $E_t(\eta^\delta(t)) \leq K_1$ for every $t\in [0,T]$.
    Owing to \cref{rmk:coerc_Et}  and of \cref{ass:unif_coerc}, we deduce the thesis.
\end{proof}

In the next result we study the variation of the functions $t\mapsto \E_t(\eta^\delta)$ as we tune the parameter $\delta$.

\begin{lemma}\label{lem:bound_var}
    Let us assume \cref{ass:metric_space,ass:reg_E,ass:der_time,ass:unif_coerc}.
    For every $\delta>0$, let $\eta^\delta \colon [0,T]\to \Xx$ be a discrete quasistatic evolution constructed according to \cref{def:discr_evol} and starting from $u_0\in \Xx$. 
    Let us consider the function $\E_\cdot(\eta^\delta) \colon [0,T]\to\R$ defined in \cref{eq:def_evolving_energ}. Then, if \cref{ax:decreasing} is satisfied, there exists $K>0$ independent of $\delta$ such that
    \begin{equation}\label{eq:bound_variation}
         \sup_{\mathcal{P}} \, \sum_{j=1}^m|\E_{s_{j+1}}(\eta^\delta)-\E_{s_{j}}(\eta^\delta)| \leq  K,
    \end{equation}
     where the supremum is taken over the family of finite partitions $\mathcal{P}=\{0=s_1<s_2<\ldots<s_m=T \}$ of $[0,T]$.
\end{lemma}
\begin{proof}
Let $m \in \mathbb{N}$ and let $\mathcal{P} = \{ 0 = s_{1} < s_{2}  < \ldots < s_{m} = T\}$ be a partition of~$[0, T]$.
Recalling that $\E_\cdot(\eta^\delta)$ is continuous from the right,  for every $j = 1, \ldots, m-1$ we have that
    \begin{equation} \label{eq:ener_bal_op_cl}
        \begin{split}
            \E_{s_{j+1}}(u^\delta)-\E_{s_{j}}(u^\delta) 
            & =  \E_{s_{j+1}}^+(\eta^\delta)-\E_{s_{j}}^-(\eta^\delta) + \E_{s_{j}}^-(\eta^\delta) - \E_{s_{j}}^+(\eta^\delta) \\
            & = \int_{s_j}^{s_{j+1}} \partial_t E_\tau \big( \eta^\delta(\tau) \big) \diff \tau - \mu^\delta([s_j,s_{j+1}]) + \mu^\delta(\{ s_j \}) \\
            & = \int_{s_j}^{s_{j+1}} \partial_t E_\tau \big( \eta^\delta(\tau) \big) \diff \tau - \mu^\delta((s_j,s_{j+1}]),
        \end{split}
    \end{equation}
    for every $s_j<s_{j+1}$.
    Since $\mu^\delta$ is a non-negative measure, we deduce that
    \begin{equation} \label{eq:bound_var_1}
        \sum_{j=1}^m|\E_{s_{j+1}}(\eta^\delta)-\E_{s_{j}}(\eta^\delta)| \leq
        \int_0^T \left| \partial_t E_\tau\big( \eta^\delta(\tau) \big) \right| \diff \tau + \mu^\delta([0,T]).
    \end{equation}
    We further observe that, by virtue of \cref{ass:der_time} and \cref{lem:bound_energ_evol}, we have
    \begin{equation} \label{eq:bound_var_2}
        \left| \partial_t E_\tau\big( \eta^\delta(\tau) \big) \right| \leq  C_1 \E_\tau(\eta^\delta) + C_2  \leq K_1
    \end{equation}
    for every $\tau\in [0,T]$, where $K_1$ is a constant that does not depend on $\delta$.
    Moreover, from \cref{prop:energy_balance} it descends that
    \begin{equation*}
        \mu^\delta([0,T]) = \E_0(\eta^\delta) - \E_T(\eta^\delta) + \int_0^T\partial_t E_\tau\big(\eta^\delta(\tau)\big) \diff \tau,
    \end{equation*}
    and, combining the last identity with \cref{eq:bound_var_2} and with \cref{prop:energy_balance}, we obtain
    \begin{equation} \label{eq:bound_var_3}
        \mu^\delta([0,T]) \leq K_2,
    \end{equation}
    where, once again, $K_2>0$ does not depend on $\delta$. Finally, by combining \cref{eq:bound_var_1,eq:bound_var_2,eq:bound_var_3}, we deduce the bound in \cref{eq:bound_variation}.
\end{proof}

We are now in a position to establish a result analogue to \cite[Proposition~5.2]{AR17}.

\begin{proposition}\label{prop:limit_energies}
    Let us assume \cref{ass:metric_space,ass:reg_E,ass:der_time,ass:unif_coerc}.
    Given a non-negative decreasing sequence $(\delta_n)_n$ such that $\delta_n\to 0$ as $n\to\infty$, let $\eta^{\delta_n} \colon [0,T]\to\Xx$ be discrete quasistatic evolutions constructed according to \cref{def:discr_evol} and starting from $x_0\in \Xx$. Then, if \cref{ax:decreasing} is satisfied, there exist a positive Radon measure $\bar \mu \in \mathcal{M}([0,T])$, and functions $\bar \E \in BV([0,T],\R)$ and $\mathcal{D}\in L^\infty([0,T],\R)$ such that, up to a subsequence, for $n\to \infty$ we have
    \begin{equation}\label{eq:convergences}
        \begin{split}
             & \mu^{\delta_n}\weak^* \bar \mu \quad \mbox{in }\mathcal M([0,T]),                  \\
             & \lim_{n\to\infty} \E_t(\eta^{\delta_n}) = \bar\E_t \quad \mbox{for every } t\in[0,T], \\
             & \mathcal{D}^{\delta_n} \weak^* \mathcal{D} \quad \mbox{in } L^\infty([0,T],\R),
        \end{split}
    \end{equation}
    where we introduced the notation $t\mapsto \mathcal{D}^{\delta_n}(t)\coloneqq \partial_t E_t\big( \eta^{\delta_n}(t) \big)$.
    Moreover, if we use $\bar \E^+_t, \bar \E^-_t$ to denote, respectively, the right and the left limits of $\bar \E$ at the instant $t\in[0,T]$ (here we set $\bar \E^-_0 \coloneqq \bar \E_0$ and $\bar \E^+_T \coloneqq \bar \E_T$), then we have that
    \begin{equation} \label{eq:ener_bal_limit}
        \bar \E^+_{t_2} - \bar \E^-_{t_1} = \int_{t_1}^{t_2} \mathcal{D}(\tau) \diff \tau - \bar \mu([t_1,t_2])
    \end{equation}
    for every $0\leq t_1\leq t_2 \leq T$.
\end{proposition}
\begin{proof}
    For every $n\geq1$, we define $\F^{\delta_n}\colon [0,T]\to\R$ as
    \begin{equation} \label{eq:def_F_ancill}
        \F^{\delta_n}_t \coloneqq \E_t(\eta^{\delta_n})- \int_0^t \mathcal{D}^{\delta_n}(\tau) \diff \tau
    \end{equation}
    for every $t\in [0,T]$.
    We observe that the functions $\F^{\delta_n}$ are non-increasing. Indeed, given $0\leq t_1 \leq t_2 \leq T$ we have that
    \begin{equation*}
        \begin{split}
            \F_{t_2}^{\delta_n} - \F_{t_1}^{\delta_n} & =
            \E_{t_2}^+(\eta^{\delta_n}) - \E_{t_1}^+(\eta^{\delta_n})
            - \int_{t_1}^{t_2} \mathcal{D}^{\delta_n}(\tau)\diff \tau 
            = -\mu^{\delta_n}((t_1,t_2]) \leq 0,
        \end{split}
    \end{equation*}
    where we used \cref{rmk:ener_bal_variants}.
    Moreover, the sequence $(\F^{\delta_n})_n$ is uniformly bounded, owing to \cref{eq:bound_var_2} and \cref{lem:bound_energ_evol}. By Helly's Selection Theorem, it follows that up to a subsequence, $(\F^{\delta_n})_n$ is pointwise convergent to a non-increasing function $\bar \F$ at every point $t\in [0,T]$. In addition, using again the estimate in \cref{eq:bound_var_2}, it follows that the sequence $(\mathcal{D}^{\delta_n})_n$ is pre-compact in the weak-$*$ topology of $L^\infty$, while the measures $(\mu^{\delta_n})_n$ are pre-compact in the weak-$*$ topology of $\mathcal{M}([0,T])$.
    These  observations establish the first and the third convergences in \eqref{eq:convergences}, along a proper and not relabelled subsequence. Finally, by passing to the limit in \cref{eq:def_F_ancill}, we get the pointwise convergence of the energy stated in \eqref{eq:convergences}.
    To prove the energy balance for $\bar \E$, we argue as in \cite{AR17}. Namely, given $0< t_1\leq t_2 < T$, we compute
    \begin{equation}\label{eq:portmant_inf}
        \begin{split}
            \bar \mu ([t_1,t_2]) & =
            \lim_{k\to\infty}
            \bar \mu((t_1-1/k, t_2+1/k))  \leq
            \lim_{k\to\infty}
            \liminf_{n\to\infty} \mu^{\delta_n}((t_1-1/k, t_2+1/k))                                            \\
                                 & \leq
            \lim_{k\to\infty}
            \liminf_{n\to\infty} \mu^{\delta_n}((t_1-1/k, t_2+1/k])                                            \\
                                 & = \lim_{k\to\infty}
            \left( \bar\E_{t_1-1/k} -
            \bar\E_{t_2+1/k} + \int_{t_1-1/k}^{t_2+1/k}\mathcal{D}^{\delta_n}(\tau)\diff \tau \right)             \\
                                 & =\bar\E_{t_1}^- -\bar\E_{t_2}^+ + \int_{t_1}^{t_2}\mathcal{D}(\tau)\diff \tau.
        \end{split}
    \end{equation}
    Moreover, we have
    \begin{equation}
        \label{eq:portmant_sup}
        \begin{split}
            \bar \mu ([t_1,t_2]) & =
            \lim_{k\to\infty}
            \bar \mu([t_1-1/k, t_2+1/k])  \geq
            \lim_{k\to\infty}
            \limsup_{n\to\infty} \mu^{\delta_n}([t_1-1/k, t_2+1/k])                                            \\
                                 & \geq
            \lim_{k\to\infty}
            \limsup_{n\to\infty} \mu^{\delta_n}((t_1-1/k, t_2+1/k])                                            \\
                                 & = \lim_{k\to\infty}
            \left( \bar\E_{t_1-1/k} -
            \bar\E_{t_2+1/k} + \int_{t_1-1/k}^{t_2+1/k}\mathcal{D}^{\delta_n}(\tau)\diff \tau \right)             \\
                                 & =\bar\E_{t_1}^- -\bar\E_{t_2}^+ + \int_{t_1}^{t_2}\mathcal{D}(\tau)\diff \tau,
        \end{split} 
    \end{equation}
    and this concludes the proof.
\end{proof}



In the next corollary we report some further properties of the limiting function $\bar \E$ constructed above.

\begin{corollary} \label{cor:reg_limit_ener}
    Under the same assumptions and notations as in \cref{prop:limit_energies}, we deduce that the function $\bar \E \colon [0,T]\to \R$ is of bounded variation, that its distributional derivative satisfies $d\bar\E = \mathcal{D L}^1 -\bar \mu$, and that
    \begin{equation*}
        \bar \E_t^+ - \bar \E_t^- = - \bar \mu(\{t\})  \qquad \text{for every $t \in [0, T]$}. 
    \end{equation*}
    Finally, the discontinuity points of $\bar \E$ (i.e., the atoms of $\bar \mu$) are at most countably many.
\end{corollary}
\begin{proof}
    The proof directly descends from \eqref{eq:portmant_inf} and \eqref{eq:portmant_sup}.
\end{proof}

\subsection{Limiting construction of quasi\-static evolutions}
In this part, we show how we can obtain quasi\-static evolutions using a family of curves $(\eta^\delta)_{\delta}$ constructed according to \cref{def:discr_evol}.
We introduce below the second axiom.

\begin{axiom}\label{ax:gradient}
    Let the family of mappings $\bar{\omega}_t\colon \Xx \to \Xx$ indexed by $t \in [0, T]$ be the transition rule as in \cref{def:discr_evol}.
    Then, we require that for every $x \in \Xx$ and for every $t \in [0, T]$, we have that $|\partial E_t|\big(\bar \omega_t(x)\big) = 0$.
\end{axiom}


In the next result, we show that the curves $\eta^{\delta}$ are somehow close to be made of critical points, even for $t\not\in J^\delta$.
More precisely, we provide an estimate uniform in $\delta$ for the magnitude of $|\partial E_t|$ along $\eta^\delta$.

\begin{lemma}\label{lem:est_grad_ener}
    Let us assume \cref{ass:metric_space,ass:reg_E,ass:der_time,ass:unif_coerc}.
    For every $\delta>0$, let $\eta^\delta \colon [0,T]\to \R^d$ be a discrete quasistatic evolution constructed according to \cref{def:discr_evol} and starting from $x_0\in \Xx$.
    Then, if \cref{ax:decreasing,ax:gradient} are satisfied, for every $\e>0$ there exists  $\bar \delta >0$ such that for every $\delta \in (0,  \bar \delta]$
    \begin{equation*}
        \sup_{t\in [0,T]} |\partial E_t| \big( \eta^\delta(t) \big) \leq \e.
    \end{equation*}
\end{lemma}
\begin{proof}
    Recalling the definition of $J^\delta$ in \cref{eq:def_J_delta}, we observe that the set $J^\delta \cup \{0\}=\{i\delta\mid i=0,\ldots,M\}$ is  a $\delta$-net for the interval $[0,T]$. By virtue of \cref{lem:bound_traj}  and of \cref{ass:unif_coerc}, there exists a compact set $K\subset \Xx$ such that $\eta^\delta(t)\in K$ for every $t\in[0,T]$ and for every $\delta>0$.
    Moreover, owing to \cref{ass:reg_E}, it descends that the space-gradient $|\partial E_\cdot|\colon  [0,T]\times \Xx \to \R$ is uniformly continuous when restricted to $[0,T]\times K$, and we denote with $\xi \colon [0,T]\times [0,\mathrm{diam}(K)]\to \R_+$ a modulus of continuity. We recall that $\xi$ is a function non-decreasing in each argument and that satisfies $\xi(0,0)=\lim_{(s,r)\to(0^+,0^+)}\xi(s,r) =0$.
    Let us fix $t\in [0,T]$. Then, there exists $\hat i\in \{0,\ldots,M\}$ such that $\hat i \delta\leq t<(\hat i + 1)\delta$. By \cref{ax:gradient} and recalling that $\eta^\delta$ is piecewise constant, we observe that $|\partial E_{\hat i \delta}| \big(\eta^\delta(t)\big)=0 $. Therefore, we have that
    \begin{equation*}
        |\partial E_{t}| \big(\eta^\delta(t)\big) \leq |\partial E_{\hat i \delta}| \big(\eta^\delta(t)\big) + \left|
        |\partial E_{t}| \big(\eta^\delta(t)\big) - |\partial E_{\hat i \delta}| \big(\eta^\delta(t)\big) \right|
        \leq \xi(\delta,0),
    \end{equation*}
    and this concludes the proof.
\end{proof}

\begin{remark}\label{rmk:almost_critic}
    A similar result is reported in \cite[Notation~5.3]{AR17} for the class of curves obtained by solving a properly rescaled gradient flow.
    However, in the construction in \cite{AR17}, it is possible to prove that the set of instants where the gradients converge to $0$ has \emph{full Lebesgue measure}, but not that it is the whole interval $[0,T]$.
\end{remark}

Before proceeding, we recall the definition of the quotient space $\X$ given in \cref{eq:def_quot_space}:
\begin{equation*}
    \X \coloneqq \big([0,T]\times \Xx \big)/\sim,
\end{equation*}
where $\sim$ is the equivalence relation on $[0,T]\times\Xx$ given by (cf.~\cref{eq:def_equiv_rel})
\begin{equation*}
    \begin{split}
        (t_1,x_1)\sim (t_2,x_2) \iff &
        t_1=t_2 \mbox{ and $x_1,x_2$ belong to the same} \\ & \mbox{path-connected component of }\{x \in \Xx \mid |\partial E_{t_1}|(x) = 0\}.
    \end{split}
\end{equation*}
The space $\X$ is equipped with the quotient topology.  We denote with $q\colon [0,T] \times \Xx \to\X$ the quotient map, and we use the notation $[(t,x)]$ to describe the elements of $\X$, i.e., the equivalence classes induced by \eqref{eq:def_equiv_rel}. We will often denote by $\hat{x}$ an element $[(t, x)]$ of~$\X$.

\begin{lemma}\label{lem:q_space_haus}
Let us assume \cref{ass:metric_space,ass:conn_comp}. Then, the quotient space $\X$ defined in \eqref{eq:def_quot_space} with the relation \eqref{eq:def_equiv_rel} is a Hausdorff space, i.e., for every $\hat{x}_{1}, \hat{x}_{2}\in \X$ with $\hat{x}_{1} \neq \hat{x}_{2}$ there exist two disjoint open sets $U_1,U_2 \subset \X$ such that  $\hat{x}_{1}\in U_1$ and $\hat{x}_{2}\in U_2$.
\end{lemma}

\begin{proof}
    For every $ \hat{x} = [(t,x)] \in\X$, let us introduce the set $A\coloneqq \{w\in\Xx \mid (t,w)\sim (t,x)\}$, which satisfies either $A=\{x\}$ in the case $|\partial E_t|(x)\neq 0$, or, if $|\partial E_t|(x)=0$, $A$ does coincide with the path-connected component of critical points of $E_t$ that contains $x$.
    Moreover, from \cref{ass:conn_comp} it descends that $A$ is compact.
    Then, observing that  $q^{-1} ( \hat{x} )  = \{t\} \times A$, we deduce the thesis by applying \cite[Theorem~8.11 and Exercise~8.13.(l)]{Kos}, and recalling that $[0,T]\times\Xx$ is a Hausdorff space.
\end{proof}

We observe that \cref{lem:q_space_haus} implies the uniqueness of the limit for every converging sequence in $\X$.
Furthermore, we notice that the energy $E\colon [0,T]\times\Xx \to\R$ can be defined also on the quotient space $\X$. Namely, given $\hat x = [(t,x)]\in\X$, we set
\begin{equation} \label{eq:energy_hat}
    \hat E(\hat x) = \hat E\big( [(t,x)] \big)\coloneqq  E_{t} (x).
\end{equation}
Let us show that this definition does not depend on the representative of the class. Indeed, if $|\partial E_t|(x)\neq 0$, we have that $[(t,x)] = \{(t,x)\}$. Otherwise, if  $|\partial E_t|(x) = 0$, let us consider another element of the same class $(t,x')\in[(t,x)]$ (i.e., in the same path-connected component of the set of critical points), and we observe that $ E_{t} ( x) = E_{t} (x')$ by virtue of \cref{ass:conn_comp}. 
This shows that $\hat E\colon \X\to\R$ is well-defined, and we have that $E=\hat E \circ q$.
The last identity (together with the continuity of $E$) implies that $\hat E \colon \X\to\R$ is continuous as well (see the \textit{universal property of quotients} \cite[Theorem~5.2]{Kos}).

In the next fundamental axiom, we state a property for the limits of the sequences $\left(\eta^{\delta_n}(t_1^n)\right)_n,\left(\eta^{\delta_n}(t_2^n)\right)_n$ when $t_1^n, t_2^n\to t$ as $n\to\infty$.
In the limiting construction, this axiom plays the same role as \cite[Lemma~5.1]{AR17}.

\begin{axiom}\label{ax:limits_traject}
    Given a non-negative decreasing sequence $(\delta_n)_n$ such that $\delta_n\to 0$ as $n\to\infty$, let $\eta^{\delta_n}\colon [0,T]\to\Xx$ be discrete quasistatic evolutions constructed according to \cref{def:discr_evol} and starting from $x_0\in \Xx$. Let us further assume that along $\left(\eta^{\delta_n}\right)_n$ the convergences reported in~\eqref{eq:convergences} hold.
    For every $t\in [0,T]$, let us consider sequences $(t_1^n)_n,(t_2^n)_n\subset[0,T]$ and $x_{1}, x_{2} \in \Xx$ such that $t_1^n,t_2^n\to t$ as $n\to\infty$ with $t_1^n\leq t_2^n$ for every $n$, and such that $\eta^{\delta_n}(t_1^n)\to x_1$ and $\eta^{\delta_n}(t_2^n)\to x_2$ as $n\to\infty$.
    If $x_1$ and $x_2$ belong to \emph{different path-connected components} of the set $\{x\in \Xx \mid |\partial E_t|(x) = 0\}$, then there exists $c>0$ such that
    \begin{equation*}
        \bar \mu(\{t\}) \geq c, \qquad E_t( x_1 ) - E_t( x_2 )\geq c.
    \end{equation*}
\end{axiom}

\begin{remark}
    If compared to \cref{ax:decreasing,ax:gradient}, we notice that \cref{ax:limits_traject} sounds intrinsically different and less elegant, as it does not directly involve the transition rules $( \bar{\omega}_t )_{t\in [0,T]}$, but it is rather formulated in terms of the subsequence constructed in \cref{prop:limit_energies}.
    We shall devote \cref{sec:actions_abstract} to amend this point: We formulate \cref{ax:action_triangular_time_stable}, we show that \cref{ax:decreasing,ax:gradient,ax:action_triangular_time_stable} imply \cref{ax:limits_traject} (see \cref{lemma:action_implies_traj}), and finally in \cref{def:generated_by_curves} we describe a class of transition rules that comply with \cref{ax:action_triangular_time_stable} (see \cref{lemma:good_curves_fulfill_continuous_triangular_inequality}). Examples of such transition rules are provided in \cref{sec:examples_of_evolution_rules}. 
\end{remark}

\begin{remark}\label{rmk:crit_limits}
    Using the same notations as in \cref{ax:limits_traject}, the points $x_1,x_2$ are automatically critical points.
    Indeed, from the construction of $\eta^{\delta_n}$ and from \cref{lem:est_grad_ener}, it follows that
    \begin{equation*}
        \lim_{n\to\infty} \sup_{s\in[0,T]}|\partial  E_s| \big(\eta^{\delta_n}(s) \big)=0.
    \end{equation*}
    Hence, from the continuity of $|\partial E_\cdot|\colon [0,T]\times\Xx\to \R$, we conclude that $|\partial E_t|(x_1)=0$. The argument for $x_2$ is exactly the same.
\end{remark}

\begin{remark}\label{rmk:discont_points}
    The previous axiom is particularly meaningful for the instants $t\in[0,T]$ such that $\bar \mu(\{t\}) =0$. Indeed, in this case we conclude that the points $x_1,x_2\in \Xx$ obtained as above must belong to the same path-connected component of $\{x\in \Xx \mid |\partial E_t| (x) = 0\}$. In particular, in  case that the critical points of $E_t$ are isolated (as assumed in \cite{AR17}), it descends that $x_1=x_2$.
\end{remark}

\begin{remark}\label{rmk:atoms_discontinuity}
    Under \cref{ax:decreasing}, it is possible to prove a sort of converse of \cref{ax:limits_traject}. Namely, using the same notations, let us assume that along $\big(\eta^{\delta_n}\big)_n$ the convergences reported in \eqref{eq:convergences} hold, and that there exists $t\in(0,T)$ such that $\bar \mu(\{t\}) =c>0$.
    Then, following the arguments of \cite[Proposition~4.1]{AR17} it is possible to construct a subsequence $\big(\eta^{\delta_{n_k}}\big)_k$ and sequences $(t_1^k)_k, (t_2^k)_k$ satisfying $t_1^k\leq t_2^k$ and $t_1^k,t_2^k\to t$ as $k\to\infty$, such that
    $u^{\delta_{n_k}}(t_1^k)\to x_1$ and $u^{\delta_{n_k}}(t_2^k)\to x_2$, where $x_1,x_2$ belongs to different connected components of the set of critical points of $E_t$.
\end{remark}

We are finally in a position to construct a solution $\hat \eta\colon [0,T]\to\X$ obtained as the pointwise limit of (the graphs of) discrete quasistatic evolutions.
Our construction of $\hat \eta$ follows steps similar to the ones detailed in \cite{AR17} for their framework.

\begin{theorem} \label{thm:traj_convergence}
    Let us assume  \cref{ass:metric_space,ass:reg_E,ass:der_time,ass:unif_coerc,ass:conn_comp}.
    Given a non-negative decreasing sequence $(\delta_n)_n$ such that $\delta_n\to 0$ as $n\to\infty$, let $\eta^{\delta_n} \colon [0,T]\to\Xx$ be discrete quasistatic evolutions constructed according to \cref{def:discr_evol} and starting from $x_0\in \Xx$. Let us suppose that the convergences \eqref{eq:convergences} hold  and that \cref{ax:decreasing,ax:gradient,ax:limits_traject} are satisfied. Then, there exists a curve $\hat \eta \colon [0,T]\to \X$ such that
    \begin{itemize}
        \item The right limit $\hat \eta^+(t)$ exists for every $t\in(0,T]$ and left limit $\hat \eta^-(t)$ exists for every $t\in(0,T]$.
        \item For every $t\in [0,T]$, $\hat \eta(t) \subset \{ (t,x) \in [0,T]\times\Xx \mid | \partial  E_t|(x) = 0\}$.
        \item For every $0\leq s\leq t\leq T$, the following energy balance holds:
              \begin{equation} \label{eq:ener_balanc_limit_traj}
                  \hat E\big(\hat \eta^+(t)\big) -  \hat E\big(\hat \eta^-(s)\big) =
                  \int_s^t \mathcal{D}(\tau) \diff \tau - \bar\mu([s,t]),
              \end{equation}
            where $\mathcal{D} \in L^{\infty} ([0, T, \mathbb{R})$ and $\bar \mu \in \mathcal{M}^+([0,T])$ are defined in \cref{prop:limit_energies}. In particular, the set $J$ of atoms of~$\bar{\mu}$ is at most countable.
        \item $\hat \eta$ is continuous in $[0,T]\setminus J$, $J$ coincides with the jump set of $\hat \eta$, and for every $t\in [0,T]$ we have $\hat E \big(\hat \eta^-(t)\big) -  \hat E \big(\hat \eta^+(t) \big) = \bar\mu(\{t\})$.
        \item There exists a subsequence $(\delta_{n_k})_k$ such that $\big(t, \eta^{\delta_{n_k}}(t)\big) \rightarrow_\X \hat\eta(t)$ for all $t \in [0, T]$.
    \end{itemize}
\end{theorem}
\begin{proof}
\textbf{Step 1: Definition of the limiting trajectory on a countable dense set.} Let us consider the set $I\coloneqq J \cup A$, where $J\coloneqq\{t\in [0,T] \mid  \bar\mu(\{t\})>0\}$ is the (at most) countable set of discontinuity points of $\bar \E_\cdot$, and $A$ is a countable set dense in $[0,T]$. Then, since by \cref{lem:bound_traj} there exists $K \subset \Xx$ compact such that $\eta^{\delta_n}(t)\in K$ for every $t\in [0,T]$ and every $n\geq1$, we may construct a subsequence $(\eta^{\delta_n})_n$ that is convergent at every $t\in I$.
Therefore, we define $\eta_{\mathrm{temp}} \colon I\to\Xx$ as $\eta_{\mathrm{temp}} (t) \coloneqq \lim_{n\to\infty}\eta^{\delta_n}(t)$.
In order to manage the connected components of the critical points of the driving energy, it is convenient to introduce the curve function $\hat \eta\colon I\to\X$, defined as the composition $  \hat \eta (t)  \coloneqq  q\big( (t,  \eta_{\mathrm{temp}} (t)) \big) = [(t,  \eta_{\mathrm{temp}} (t)) ]$.\\
\textbf{Step 2: Extension of $\hat \eta$ to the whole evolution interval.}
We show that $\hat \eta$ admits a unique extension at any point $t\in[0,T]\setminus I$.
First of all, let us consider a sequence $(t^k)\subset I$ such that $t_k\to t\not\in I$ as $k\to\infty$. We want to show that the sequence $\big( \hat \eta(t^k) \big)_{k}=\big( [(t^k,  \eta_{\mathrm{temp}} (t^k))] \big)_{k}\subset\X$ admits a converging sub\-sequence in $\X$. 
Since  $ \eta_{\mathrm{temp}}$ takes value in  the compact set $K\subset \Xx$, we can extract a (not relabelled) sub\-sequence such that $\big((t^k,  \eta_{\mathrm{temp}}  (t^k))\big)_{k}$ converges to $(t,x)$ in $[0,T] \times \Xx$. 
Therefore, by the continuity of $q\colon [0,T]\times\Xx \to \X$, along such a sub\-sequence we have that $(\hat \eta(t^k) )_k= \left( q \big( t^k,  \eta_{\mathrm{temp}} (t^k) \big) \right)_k$ converges in $\X$ to $[(t,x)]=q((t,x))$. 
Hence, we can set $\hat \eta(t)\coloneqq [(t,x)]$ for $t\not\in I$. We have now to show that such an extension is uniquely defined.\\
To see this, let us assume that there exist $(t_1^k)_k,(t_2^k)_k\subset I$ such that $\lim_{k\to\infty}t_1^k= \lim_{k\to\infty}t_2^k= t\not\in I$, and let us assume that $\big(\hat \eta(t_1^k) \big)_k, \big(\hat \eta(t_2^k)) \big)_k\subset \X$ have limits in $\X$. We shall prove that the two limits coincide by considering for every $k$ the elements $\big( t^k_1,  \eta_{\mathrm{temp}}  (t^k_1)\big) \in \hat \eta(t_1^k)$ and $\big( t^k_2,  \eta_{\mathrm{temp}} (t^k_2)\big) \in\hat \eta(t_2^k)$. Arguing as above, up to a not relabelled sub\-sequence, we may assume that $\left( \big(t^k_1,  \eta_{\mathrm{temp}}  (t_1^k) \big)\right)_k$ and $\left( \big(t^k_1,  \eta_{\mathrm{temp}}  (t_1^k) \big)\right)_k$ converge in $[0,T]\times \Xx$ to $(t,x_1),(t,x_2)$, respectively. The extension is unique if we show that $(t,x_1)\sim(t,x_2)$.\\
Since $ \eta_{\mathrm{temp}} $ is defined as the pointwise limit of $(\eta^{\delta_n})_n$ on $I$, with a diagonal procedure we can extract a sub\-sequence $n_k$ such that $x_1=\lim_{k\to\infty} \eta^{\delta_{n_k}}(t_1^k)$ and $x_2=\lim_{k\to\infty} \eta^{\delta_{n_k}}(t_2^k)$. We may further assume that $t_1^k\leq t_2^k$ or $t_1^k\geq t_2^k$ for every $k\geq1$.
Since $t\not \in I$ and $J \subset I$, we have that $\bar \mu(\{t\})=0$. Therefore, by virtue of \cref{ax:limits_traject}, we deduce that $x_1$ and $x_2$ belong to the same connected component of the critical points of $E_t$, i.e., $(t,x_1)\sim(t,x_2)$ according to \cref{eq:def_equiv_rel}.\\
Hence, we can uniquely extend $\hat \eta$ to $[0,T]$.\\
\textbf{Step 3: Pointwise convergence on the whole evolution interval.} Let $(\eta^{\delta_n})_n$ be the sequence of discrete quasistatic solutions that converge pointwisely to $\hat \eta_{\mathrm{temp}}$ on $I$. For $t\not\in I$, we consider the sequence of points $\big(\eta^{\delta_n}(t) \big)_n \subset K$. We aim to show that, if $x'\in \Xx$ is a limiting point of $\big( \eta^{\delta_n}(t) \big)_n$, then $[(t,x')]=\hat \eta(t)$, so that we can conclude that $[ \big(t,\eta^{\delta_n}(t) \big)] \to_\X \hat \eta(t)$ in $\X$ as $n\to\infty$. 
To see that, let us restrict to a (not relabelled) sub\-sequence such that $x'=\lim_{n\to\infty}\eta^{\delta_n}(t)$.
Moreover, let us take a sequence of instants $I\ni t^k\nearrow t$ as $k\to\infty$ and such that $x=\lim_{k\to\infty} \eta_{\mathrm{temp}} (t^k)$, and, by virtue of Step~2, we have that $\hat \eta(t) = [(t,x)]$ and
$x = \lim_{k\to\infty} \eta^{\delta_{n_k}}(t^k)$, where the last limit is computed along a suitable sub\-sequence of $(\eta^{\delta_n})_n$ obtained with a diagonal procedure.
If we use \cref{ax:limits_traject} on the sequences $\big(\eta^{\delta_{n_k}}(t) \big)_k$ and $\big( \eta^{\delta_{n_k}}(t^k) \big)_k$, from the fact that $\bar \mu(\{t\})=0$ (we recall that  $t\not \in I$), we deduce that $x$ and $x'$ must lie in the same connected component of $\Crit_t$, i.e., $[(t,x')]= [(t,x)] =\hat \eta(t)$.\\
Hence, we deduce that the sequence $(\eta^{\delta_n})_n$, which converges to $\eta_{\mathrm{temp}}$ pointwise on $I$, satisfies  as well $[ \big(t,\eta^{\delta_n}(t) \big) ] \to_\X \hat \eta(t)$ in $\X$ as $n\to\infty$ for every $t\in[0,T]$.\\
\textbf{Step 4: Driving energy along $\hat \eta$.}
Owing to the convergences \eqref{eq:convergences} in \cref{prop:limit_energies}, we have that
\begin{equation} \label{eq:identity_energy_traj}
    \bar\E_t = \lim_{n\to\infty} \E_t(\eta^{\delta_n}) = \lim_{n\to\infty} E_t \big(\eta^{\delta_n}(t) \big) = \hat E \big(\hat \eta (t) \big),
\end{equation}
where we used Step~3 in the last identity.\\
\textbf{Step~5: The limiting trajectory $\hat \eta$ admits left and right limits.}
Assume the sequences $(t_1^k)_k,(t_2^k)_k$ are such that $t_1^k\searrow t$, $t_2^k\searrow t$ as $k\to\infty$, and $\hat x_1\coloneqq\lim_{k\to\infty} \hat \eta(t_1^k)$ and $\hat x_2\coloneqq\lim_{k\to\infty}\hat \eta(t_2^k)$.
Without loss of generality, we may assume that for every $k\geq 1$ we have $t_1^k\leq t_2^k$ or $t_1^k\geq t_2^k$.
If $(\eta^{\delta_n})_n$ is the sequence of discrete quasistatic evolutions constructed in Step~1, we recall that, by \cref{lem:bound_traj}, for every $t\in[0,T]$ the sequence $\big( \eta^{\delta_n}(t)\big)_n$ is contained in the compact subset $K \subset \Xx$. 
Moreover, if $y\in\Xx$ is a limiting point of $\big( \eta^{\delta_n}(t) \big)_n$, from Step~3 it follows that $(t,y)\in \hat \eta(t)$.
Let us consider the sets $\big( \eta^{\delta_n}(t_1^k) \big)_{n,k}, \big( \eta^{\delta_n}(t_2^k) \big)_{n,k}$ indexed by $n,k\in\NN$, and, with a diagonal argument on the index $n$, we extract a (not relabelled) sub\-sequence in $n$ such that $\eta^{\delta_n}(t_1^k)\to y_1^k$ and $\eta^{\delta_n}(t_2^k)\to y_2^k$ in $\Xx$ as $n\to\infty$.
As recalled above, we have that $(t_1^k,y_1^k)\in\hat \eta(t_1^k)$ and $(t_2^k,y_2^k)\in\hat \eta(t_2^k)$ for every $k\in\NN$. Moreover, up to the extraction of a sub\-sequence in $k$, we may assume that $y_1^k\to y_1$ and $y_2^k\to y_2$ in $\Xx$ as $k\to\infty$, and we get that $(t,y_1)\in\hat x_1$ and $(t,y_2) \in \hat x_2$.
Therefore, to show that $\hat x_1 = \hat x_2$ (i.e., that the limits of $\hat \eta$ computed along $(t_1^k)_k,(t_2^k)_k$ coincide) it suffices to prove that $(t,y_1)\sim (t,y_2)$.\\
To see that, we first establish the following identities on the energy:
\begin{equation} \label{eq:energ_ident_limit}
    \begin{split}
         \bar \E^+_t= \lim_{\tau\to t^+} \bar \E_\tau 
        & = \lim_{k\to\infty} \bar \E_{t_1^k} = \lim_{k\to\infty} \hat E \big(\hat \eta(t_1^k) \big) = \hat E(\hat x_1) = E_t(y_1),  \\
         \bar \E^+_t= \lim_{\tau\to t^+} \bar \E_\tau   
        & = \lim_{k\to\infty} \bar \E_{t_2^k} = \lim_{k\to\infty} \hat E\big( \hat \eta(t_2^k) \big) = \hat E(\hat x_2) = E_t(y_2).
    \end{split}
\end{equation}
Now, for every $k\geq 1$, we construct $n_k>n_{k-1}$ such that
\begin{equation*}
    \max\left\{
    |\eta^{\delta_{n_k}}(t_1^k)-y_1^k|,
    |\eta^{\delta_{n_k}}(t_2^k)-y_2^k|
    \right\} \leq \frac1k,
\end{equation*}
so that we obtain $\eta^{\delta_{n_k}}(t_1^k) \to y_1$ and $\eta^{\delta_{n_k}}(t_2^k)\to y_2$ in $\Xx$ as $k\to\infty$.
Invoking \cref{ax:limits_traject}, from \cref{eq:energ_ident_limit} we deduce that $y_1$ and $y_2$ belong to the same path-connected component of $\Crit_t$, i.e., $(t,y_1) \sim (t,y_2)$.
Moreover, since the curve $\tau\mapsto \hat \eta (\tau)$ in $\X$ admits limit from the right, we further deduce that
\begin{equation} \label{eq:rl_limits_traj_ener}
    \bar\E_t^+ = \lim_{\tau\to t^+}\hat E \big(\hat \eta(\tau)\big)
\end{equation}
for every $t\in [0,T)$, and we set $\hat \eta^+(t)\coloneqq \lim_{\tau\to t^+}\hat \eta (\tau)$. The same arguments and conclusions hold for the left limits.
Finally, if $\bar \mu (\{ t\})=0$ for $t\in(0,t)$, we obtain that $\hat \eta$ is continuous at $t$.\\
\textbf{Step~6: Energy balance with the limiting trajectory.}
From \cref{eq:ener_bal_limit,eq:rl_limits_traj_ener}, recalling that $\hat E \colon \X\to\R$ is continuous, we deduce \cref{eq:ener_balanc_limit_traj}.
In particular, we obtain that 
$
    \hat E \big(\hat \eta^-(t) \textbf{}) -  \hat E \big(\hat \eta^+(t) \big) = \bar\mu(\{t\})
$
for every $t\in[0,T]$.
\end{proof}

\begin{remark} \label{rmk:compactness_limit_traj}
    As the curve $\hat \eta \colon [0,T]\to \X$ is obtained as point-wise limit of discrete quasi\-static curves $\eta^{\delta_n}$ such that $\eta^{\delta_n}(t)\in K$ for every $t\in [0,T]$ and every $n\geq1$,  for some compact subset $K$ of~$\Xx$, it turns out that $q^{-1}(\hat \eta) \subset [0,T]\times  K  $. Therefore, if we consider any sequence of representatives $(t_i,x_i)\in \hat \eta(t_i)$ for $i\in\NN$, it turns out that the sequence $(x_i)_{i\in\NN}\subset \Xx$ is pre-compact, owing to \cref{ass:metric_space}.
\end{remark}

Let us stress that at this point, we have achieved the existence of a limit curve $\hat \eta$ with an energy balance (see
\cref{eq:ener_balanc_limit_traj}) and
where $\hat E \big(\hat \eta^-(t)\big) -  \hat E \big(\hat \eta^+(t) \big) = \bar\mu(\{t\})$ on the jump set $J$.
Comparing this to the energy balance obtained in~\cite[Eq. 1.9a and 1.9b]{AR17} within a more restrictive set of assumptions,
there are still two refinement steps to be taken:
\begin{enumerate}
    \item\label{enum:D} The function $\mathcal{D}$ which appears in the identity in \eqref{eq:ener_balanc_limit_traj} is defined in \cref{prop:limit_energies} as the weak-$*$ limit $\partial_t E_\cdot(\eta^{\delta_n}(\cdot)) \weak_{L^\infty}^* \mathcal{D}$ as $n\to\infty$, but we have not yet related $\mathcal{D}$ directly to the limit curve $\hat \eta$. In contrast, in~\cite[Eq. 1.9a]{AR17}, inside the integral term of the energy balance we read the evaluation of $\partial_t E_\cdot$ on the limit curve.
    To close this gap, we would like to relate $\mathcal{D}(t)$ to $\partial_t E_t$ and to $\hat \eta(t)$.
    \item\label{enum:energy_barriers} Our energy balance does not yet ensure that $\hat \eta$ does not cross energy barriers. In contrast, in~\cite[Eq. 1.9b]{AR17}, the authors show that
    \begin{equation}\label{eq:energy_barrier_AR17}
        \hat E \big(\hat \eta^-(t)\big) -  \hat E \big(\hat \eta^+(t) \big) = \bar\mu(\{t\}) = c_t(\hat \eta^-(t), \hat \eta^+(t)),
    \end{equation}
    where $c_t\colon \R^d\times \R^d\to \R$ is a cost function defined by minimizing the energy-dissipation integrals; \cref{eq:energy_barrier_AR17} prevents the limit trajectory form traversing energy barriers. Likewise, we would like to establish a similar relation for our energy balance.
\end{enumerate}
In the rest of this section and the next two, we address the points raised above: In \cref{sec:relating_D_E}, we work towards \cref{cor:relating_D_and_time_der}, which shows that $\mathcal{D}(t) \equiv \partial_t E_t(\hat \eta(t))$---in a suitable sense, as $\hat \eta$ takes values in $\X$---if $E$ fulfills \cref{ass:PL+time_der_slope}. 
In \cref{sec:actions_abstract} and \cref{sec:examples_of_evolution_rules_elaborate}, we investigate for which transition rules $\bar \omega_t$ we can find a characterization of the jumps of $\hat \eta$ analogous to~\eqref{eq:energy_barrier_AR17}, and thus rule out that $\hat \eta$ crosses energy barriers.

\subsubsection{Relating $\mathcal{D}(t)$ to $\partial_t E_\cdot$} \label{sec:relating_D_E}
When trying to relate $\mathcal{D}(t)$ to $\partial_t E_t$ and to $\hat \eta(t)$, we run into the limitation that $\hat{\eta}$ takes values in the quotient space $\X$, while the $\partial_t E_\cdot$ a priori does not factor through $\X$---and thus, the quantity $\partial_t E_t(\hat \eta(t))$ is not even well-defined. We will see, however, that $\partial_t E_\cdot$ factors through $\X$ \emph{for almost every} $t\in[0,T]$ if $E$ fulfills \cref{ass:PL+time_der_slope}. To start, let us first provide upper and lower bounds for $\mathcal{D}$---we will see later that those lower and upper bounds coincide for almost all $t \in [0, T]$.
\begin{lemma}\label{lemma:bound_der_lim}
    Let us assume that the requirements of \cref{thm:traj_convergence} are fulfilled and
    let $\hat{\eta}$ be a limit curve constructed according to \cref{thm:traj_convergence}. Then
    \begin{equation} \label{eq:low_bound_der}
                    \sup_{(t,x)\in \hat \eta(t)} \partial_t E_t (x) \geq \mathcal{D}(t) \geq \inf_{(t,x)\in \hat \eta(t)} \partial_t E_t (x)
    \end{equation}
    for a.e. $t\in[0,T]$.
\end{lemma}
\begin{proof}
    Recalling the pointwise convergence established in Step~3  in the proof of \cref{thm:traj_convergence}, for every $t$ we have that, given a subsequence $\big( \eta^{\delta_{n_k}}(t) \big)_k$ such that $\eta^{\delta_{n_k}}(t) \to x_1$ as $k\to\infty$, then $(t,x_1)\in \hat \eta(t)$ and
    \begin{equation} \label{eq:bound_der_subseq}
                    \lim_{k\to\infty} \mathcal{D}^{\delta_{n_k}}(t) =
                    \lim_{k\to\infty} \partial_t E_t \big( \eta^{\delta_{n_k}}(t) \big) =
                    \partial_t E_t(x_1) \geq
                    \inf_{(t,x)\in\hat \eta(t)} \partial_t E(x).
    \end{equation}
    Likewise, we get that $\lim_{k\to\infty} \mathcal{D}^{\delta_{n_k}}(t) \leq \sup_{(t,x)\in\hat \eta(t)} \partial_t E(x)$. Furthermore, since by \cref{lem:bound_traj} there exists $\rho>0$ such that $\eta^{\delta_n}(t)\in B_\rho(x_0)$, we have that for every $t\in [0,T]$, $\partial_t E_t \big( \eta^{\delta_n}(t) \big)$ is uniformly bounded. Thus, from Fatou's Lemma we deduce that
    \begin{equation*}
                    \limsup_{n\to\infty} \mathcal{D}^{\delta_{n}}(t) \geq \mathcal{D}(t) \geq \liminf_{n\to\infty} \mathcal{D}^{\delta_{n}}(t)
    \end{equation*}
    for a.e. $t\in[0,T]$. Since we can approximate the $\limsup$ with a subsequence, by combining the last inequality with \cref{eq:bound_der_subseq} we conclude the proof.
\end{proof}

In the next result, we show that, if we strengthen the inequality \eqref{eq:liminf_ineq_relaxed} in \cref{ass:PL+time_der_slope} by making it symmetric (see \cref{ass:strengthening_6}), the infimum and the supremum in \cref{lemma:bound_der_lim} coincide almost everywhere.

\begin{lemma} \label{lemma:time_derivative_factors}
    Let us assume \cref{ass:metric_space,ass:der_time,ass:reg_E,ass:unif_coerc,ass:conn_comp,ass:strengthening_6}.
    Let $K \subseteq \Xx$ be compact.
    Then, for all but countably many $t \in [0, T]$, and for all connected components $\{\Crit_i^t\}_{i \in I}$ of $\{x \in \Xx : |\partial E_t(x)| = 0\} \cap K$, there exists some $\kappa_i\in \R$ such that $\partial_t E_t(x) = \kappa_i$ for every $x \in \Crit_i^t$.
\end{lemma}
Before proving \cref{lemma:time_derivative_factors}, let us note that \cref{lemma:time_derivative_factors} together with \cref{lemma:bound_der_lim} fully characterize $\mathcal{D}$ almost everywhere. To this end, \cref{lem:bound_energ_evol} and \cref{ass:unif_coerc} ensure that we can
find a compact $K$ as required in \cref{lemma:time_derivative_factors}.

\begin{corollary} \label{cor:relating_D_and_time_der}
    Let us assume \cref{ass:metric_space,ass:der_time,ass:reg_E,ass:unif_coerc,ass:conn_comp,ass:strengthening_6}, and furthermore that the requirements of \cref{thm:traj_convergence} are fulfilled.
    Let $\hat{\eta}\colon [0,T]\to \X$ be a curve constructed according to \cref{thm:traj_convergence}. Then, we have
    \begin{equation}\label{eq:characterize_der}
        \mathcal{D}(t) =  \inf_{(t, x) \in \hat \eta(t)} \partial_t E_t(x) = \sup_{(t, x) \in \hat \eta(t)} \partial_t E_t(x) 
    \end{equation}
    for a.e.~$t \in [0, T]$.
    In particular, we can rewrite the energy balance \cref{eq:ener_balanc_limit_traj} as
    \begin{equation*}
        \hat E\big(\hat \eta^+(t_2)\big) -  \hat E\big(\hat \eta^-(t_1)\big) + \bar \mu([t_1,t_2]) = \int_{t_1}^{t_2} \mathfrak{D}^+_t E_\tau\big(\hat \eta(\tau)\big) \diff \tau  = \int_{t_1}^{t_2} \mathfrak{D}^-_t E_\tau\big(\hat \eta(\tau)\big) \diff \tau ,
    \end{equation*}
    where $\mathfrak{D}^+_t E_\tau(\hat \eta(\tau)) \coloneqq \sup_{ (\tau, x)  \in \hat\eta(\tau)}\partial_t E_{ \tau } (x)$ and $\mathfrak{D}^-_t E_\tau(\hat \eta(\tau)) \coloneqq \inf_{ (\tau, x )  \in \hat\eta(\tau)}\partial_t E_{ \tau } (x)$.
\end{corollary}

\begin{proof}[Proof of \cref{lemma:time_derivative_factors}]
    We set
    $\mathfrak{C}_t$ to be the set of connected components of $\{x \in \Xx : |\partial E_t|(x) = 0\} \cap K$.
    Defining
    \begin{equation*}
        H(t) = \sup_{\Crit \in \mathfrak{C}_t} \sup_{x, y \in \Crit} \big( \partial_t E_t(x) - \partial_t E_t(y) \big),
    \end{equation*}
    we thus need to show that for all but countably many $t$, $H(t) = 0$.
    Since we have $\{t \in [0,  T  ] : H(t) > 0\} = \bigcup_{n \in \mathbb{N}} \{t \in [0,  T  ] : H(t) > 1/n\}$,
    it suffices to show that for each $\varepsilon > 0$, we have that $H(t) \leq \varepsilon$ for all but countably many $t$.
    We will proceed by contradiction and assume that there is some $\varepsilon > 0$ which  does not satisfy such condition, i.e., $B_\varepsilon \coloneqq \{t \in [0,  T  ] : H(t) > \varepsilon\}$ is \emph{not}  at most  countable. Within $B_\varepsilon$, we can choose an increasing sequence which converges to $ \bar{t} \coloneqq \inf\{t \in [0, T] \mid B_\varepsilon \cap [t, T] \text{ is countable}\}$. Thus, we can find a sequence of  triples $(t_n, x_n, y_n) \in [0, T] \times \Xx \times \Xx$ such that
    \begin{enumerate}
        \item $t_n \uparrow  \bar{t} $
        \item $t_n \neq  \bar{t} $
        \item $|\partial E_{t_n}|(x_n) = | \partial E_{t_n}|(y_n) = 0$
        \item $E_{t_n}(x_n) = E_{t_n}(y_n)$  (by \cref{ass:conn_comp})\label{item:x_y_same_energy}
        \item $\partial_t E_{t_n}(x_n) > \partial_t E_{t_n}(y_n) + \varepsilon$\label{item:diff_in_time_derivative}
        \item $x_n \in K, y_n \in K$.
    \end{enumerate}
    Because of the last condition, we can---without relabeling---extract a subsequence such that, in addition to the conditions above, it holds that
    $(x_n, y_n) \to ( \bar{x}, \bar{y} )$ for some $ \bar{x}, \bar{y}  \in \Xx$. Passing to the limit in condition~\eqref{item:diff_in_time_derivative} above
    and setting $P \coloneqq \frac{\partial_t E_{ \bar{t} }( \bar{x}  ) + \partial_t E_{ \bar{t} }( \bar{y} )}{2} >0 $,
    we get that 
    \begin{equation}\label{eq:time_partial_diff_difference}
        \partial_t E_{ \bar{t} }( \bar{x} ) \geq P + \frac{\varepsilon}{2} \qquad \text{and} \qquad \partial_t E_{ \bar{t} }( \bar{y} ) \leq P - \frac{\varepsilon}{2}.
    \end{equation}
    We claim that there exists $N \in \mathbb{N}$ such that for every $n > N$ 
    \begin{equation}
         E_{t_n}(x_n) < E_{\bar{t}}(\bar{x}) + P (t_n -  \bar{t} ) \qquad \text{and} \qquad E_{t_n}(y_n) > E_{ \bar{t} }(\bar{y}) + P (t_n - \bar{t} ).  \label{eq:energy_difference}
    \end{equation}
    By symmetry, it is enough to prove the first inequality in~\eqref{eq:energy_difference}. 
    We then have
    \begin{equation*}
        E_{t_n}(y_n) > E_{ \bar{t} }( \bar{y} ) + P (t_n -  \bar{t} ) = E_{ \bar{t} }( \bar{x}  ) + P (t_n -  \bar{t} ) > E_{t_n}(x_n),
    \end{equation*}
    which contradicts condition~\eqref{item:x_y_same_energy} and thus finishes the proof. \\
    To prove \cref{eq:energy_difference}, we choose some $\Delta_{\bar x}, \Delta_{ \bar{t} }, \varepsilon', D > 0$ such that for all
    $t \in [ \bar{t}  - \Delta_{ \bar{t} },  \bar{t} ]$, $x \in  B_{\Delta_{\bar x}}(\bar{x}) $ we have that
    \begin{align}
        \partial_t E_{t}(x)                           & > P + \varepsilon', \label{eq:energy_negative}                                         \\
        \big||\partial E_t|(x) - |\partial E_{ \bar{t} }|(x)\big|                         & \leq D  |t -  \bar{t} | , \label{eq:hessian_small}                                                    \\
        E_{ \bar{t} }(x) - E_{ \bar{t} }( \bar x ) 
        &\leq  \varepsilon_{\bar x}\big( d(x, \bar x)  \big) |\partial E_{ \bar{t} }|(x).
        \label{eq:lojasiewicz}
    \end{align}
    Here we can demand~\eqref{eq:energy_negative} from continuity of $\partial_t E_t$ (see \cref{ass:reg_E}) together with \cref{eq:time_partial_diff_difference}; we can demand~\eqref{eq:hessian_small} from the Lipschitz continuity of $|\partial E_t|$ in \cref{ass:strengthening_6}; finally, we can demand~\eqref{eq:lojasiewicz} from the condition \eqref{eq:liminf_ineq_stronger} in \cref{ass:strengthening_6} applied at the instant $ \bar{t} $ and at the point $ \bar x $. 
    For large enough $n$, we have that $t_n \in [ \bar{t}  -\Delta_{ \bar{t} },  \bar{t} ]$ and $x_n \in B_{\Delta_{\bar x}}( \bar{x} )$.
    We will now show that $E_{ \bar{t} }(\hat{x}) + P (t_n -  \bar{t} ) - E_{t_n}(x_n)$ is positive for large $n$. Indeed, using conditions (1)--(5) we have that
    \begin{align*}
        E_{ \bar{t} }( \bar{x} ) + P (t_n -  \bar{t} ) - E_{t_n}(x_n) & 
        = \big(E_{ \bar{t} }(x_n) - E_{t_n}(x_n)\big) - P ( \bar{t}  - t_n) - \big( E_{ \bar{t} }(x_n) - E_{ \bar{t} }( \bar{x} ) \big)                                                                    \\
         & \geq (P + \varepsilon') ( \bar{t}  - t_n) - P ( \bar{t}  - t_n) - \varepsilon_{\bar x}\big( d(x_n, \bar x ) \big) |\partial E_{ \bar{t} }|(x_n)                                                 \\
        & \geq \varepsilon' ( \bar{t}  - t_n) -  \varepsilon_{\bar x}\big( d(x_n,\bar x)  \big) D ( \bar{t}  - t_n).
    \end{align*}
    For large $n$,  $\varepsilon_{\bar x}\big( d(x_n,\bar x)  \big) D$ becomes arbitrarily small, so that the first term dominates the second. Since $\varepsilon' > 0$, this proves inequality~\eqref{eq:energy_difference},which results in the contradiction and finishes the proof.
\end{proof}

\begin{remark}
    We observe that, in the proof of \cref{lemma:time_derivative_factors}, we can deduce the existence of $N$ such that $\forall n > N: E_{t_n}(y_n) > E_{ \bar{t} }( \bar{y} ) + P (t_n -  \bar{t} )$ by using the condition \eqref{eq:liminf_ineq_relaxed} in \cref{ass:PL+time_der_slope}.
    However, the latter does not suffice for establishing the relation in \cref{eq:energy_difference}, and we need the strengthen version reported in \cref{ass:strengthening_6}.
\end{remark}

\begin{remark}
    It is interesting to observe that in \cite{Fornasier-Tradeoff} the authors employed an argument similar to the one used in the proof of \cref{lemma:time_derivative_factors} to establish a property of critical points of regularized functionals of the form $H_\alpha = F + \alpha G$, with $F,G$ defined over a Banach space $U$, and with $\alpha\in [0,+\infty)$ (for more details, see \cite[Theorem~5.1]{Fornasier-Tradeoff}).
\end{remark}

\section{Characterizing energy jumps with actions}\label{sec:actions_abstract}
\newcommand{\Pcheap}{P_{\mathrm{cheap}}}
\newcommand{\Pexp}{P_{\mathrm{exp}}}
As the last step in describing the limit curve $\hat \eta$ and its energy balance, we would like to characterize the energy jumps $E_t(\hat \eta^+(t)) - E_t(\hat \eta^-(t))$ and thus ensure that
$\hat \eta$ does not cross energy barriers. The authors of~\cite{AR17} achieved this goal in their setting by describing the energy jumps through a cost function $c_t$, which is defined by minimizing the energy-dissipation integrals.
In this section, we prove an analogous result for certain transition rules $\bar{\omega}$ by characterizing the energy jumps through \emph{actions}, which play a similar role as the cost function $c$  in~\cite[Equation (2.4)]{AR17}, but whose value depends on the specific transition rule $\bar{\omega}$ used in the construction.
\subsection{Generalizing from the gradient flow}
\newcommand{\GF}{^{^{_F}}}
\newcommand{\cGF}{c\GF}
To investigate transition rules which admit a characterization of the energy jumps, we start by considering transition rules which bring our framework as close as possible to the one in \cite{AR17}, i.e., transition rules arising from the gradient flow, which fulfill the following condition:
\begin{equation}\label{eq:evolution_rule_of_gradient_flow}
    \bar{\omega}_t(x) \in \omega\mathrm{-}\lim_{s\to\infty}\gamma(s) 
    \coloneqq \bigcap_{S>0}\overline{\bigcup_{s\geq S} \{\gamma(s)\} },
\end{equation}
where $\gamma$ is a gradient flow trajectory of $E_t$ starting at $x$, i.e., $\gamma(0) = x$
and, for all $s > 0$,
\begin{equation}\label{eq:gradient_flow}
    E_t(x) - E_t(\gamma(s)) = \frac12 \int_0^s  |\partial E_t|^{2}  (\gamma(\tau)) + |\dot \gamma(\tau)|^2 \diff\tau.
\end{equation}
We recall that \cref{eq:evolution_rule_of_gradient_flow} requires $\bar{\omega}_t(x)$ to belong to the $\omega$-limit of the gradient flow line $s\mapsto \gamma(s)$.
We observe that, whenever the limit $\lim_{s\to\infty} \gamma(s)= x_\infty$ exists, it turns out that $\bar{\omega}_t(x) = x_\infty$.
For such transition rules arising from the gradient flow, we have, for $x \in \R^n$, $t \in (0, T)$, $t_1^n \uparrow t$ and $t_2^n \downarrow t$,
\begin{alignat}{2}
    E_t(x) - E_t( \bar{\omega}_{t} (x)) & = \cGF_t(x,  \bar{\omega}_{t}  (x)) \quad \forall x \in \R^n, & \quad & \text{and hence (by \cref{eq:def_mu_delta})} \nonumber \\
    \mu^\delta(\{t\})             & = \cGF_t \big( \lim_{n \to \infty}  \eta^\delta(t_1^n) ,  \lim_{n \to \infty}   \eta^\delta(t_2^n) \big), &       & \label{eq:action_equality}
\end{alignat}
where $\cGF_t$ is defined as follows:
\begin{equation}\label{eq:gradient_flow_action}
    \cGF_t(u_1,u_2)\coloneqq \inf\left\{
    \frac{1}{2}\int_a^b \left(  |\partial E_t|^{2}  \big(\gamma(s)\big)
    + |\dot \gamma (s)|^2\right) \diff s \;\middle|\;
    \substack{a < b \in \R \\ \gamma \in AC([a,b],\R^n) \\ \gamma(a) = u_1, \gamma(b) = u_2}
    \right\}.
\end{equation}
Let us assume for a moment that we can pass \cref{eq:action_equality} to the limit as $\delta \rightarrow 0$,
i.e., that we can prove that, for all $\delta_n \rightarrow 0$,
$t_1^n \rightarrow t$, $t_2^n \rightarrow t$,
\begin{equation}\label{eq:limit_action_equality}
    \begin{aligned}
        E_t\left( \lim_{n \to \infty}  \eta^{\delta_n}(t^n_1)\right) - E_t\left( \lim_{n \to \infty}  \eta^{\delta_n}(t^n_2)\right) & = \cGF_t\left( \lim_{n \to \infty}  \eta^{\delta_n}(t^n_1),  \lim_{n \to \infty} \eta^{\delta_n}(t^n_2)\right),  & \quad \text{and} \\
        \bar{\mu}(\{t\})                                                               & = \cGF_t\left( \lim_{n \to \infty}  \eta^{\delta_n}(t^n_1) ,  \lim_{n \to \infty} \eta^{\delta_n}(t^n_2)\right), &
    \end{aligned}
\end{equation}
where we recall that $\mu^{\delta_n} \weak^* \bar{\mu}$ as in \cref{eq:convergences}.
In the case~\eqref{eq:limit_action_equality} held, we would immediatly gain two results:
\begin{enumerate}
    \item Since $\cGF_t(x, x')$ is nonzero
            whenever $x$ and $x'$ belong to different path-connected components
            of critical points of $E_t$,
            \cref{eq:limit_action_equality} would immediatly imply that \cref{ax:limits_traject} holds.
            This would, in turn, guarantee the existence of a limit curve $\hat{\eta}$ through \cref{thm:traj_convergence}.
    \item Through the energy balance~\eqref{eq:ener_bal_limit} and \cref{eq:limit_action_equality},
            we would obtain the equality
            \begin{equation*}
                \hat{E}\left( \hat{\eta}^-(t)  \right) - \hat{E}\left( \hat{\eta}^+  (t)\right) = \cGF_t\left( {\eta}^-(t), {\eta}^+ (t)\right),
            \end{equation*}
             where $\hat{\eta}^{\pm} (t) = [(t, \eta^{\pm} (t))]$.
            Informally, this implies that the limit curve $\hat{\eta}$ does not jump through energy barriers.
\end{enumerate}
Some other interesting transition rules---e.g., those which can be derived from certain discretizations
of \cref{eq:gradient_flow}---allow for equalities similar to \cref{eq:action_equality,eq:gradient_flow_action} as well, with some modifications.
In this section, we develop a framework to deal with those transition rules in a unified way,
and to pass the equalities of the form of \cref{eq:action_equality} to the limit as $\delta \rightarrow 0$,
as in \cref{eq:limit_action_equality}.
On our way to do so, we first describe a sufficient condition to be able to pass to the limit for $\delta \rightarrow 0$,
which has the form
of a continuous-in-time triangle inequality of the action $c_t$ in \cref{ax:action_triangular_time_stable}.
Afterwards, we see that this continuous-in-time triangle inequality is fulfilled by actions tailored around \cref{eq:gradient_flow_action}.
\begin{definition}\label{def:action}
    An action $(c_t)_{t \in [0, T]}$, henceforth simply written as $c_t$, is a family of lower semicontinuous functions ${c_t\colon \Xx \times \Xx \rightarrow [0,   +\infty )}$,
    parametrized by $t \in [0, T]$, such that for all $x, x' \in \Xx$ and $t \in [0, T]$:
    \begin{equation} \label{eq:action_apriori_inequality}
        E_t(x) - E_t(x') \leq c_t(x, x').
    \end{equation}
     A transition rule $\bar{\omega}_t$ is \emph{compatible with an action $c_t$},
    if  for every $x \in \Xx$ and every $t \in [0, T]$:
    \begin{equation*}
        E_t(x) - E_t\big(\bar{\omega}_t(x)\big) = c_t\big(x, \bar{\omega}_t(x)\big).
    \end{equation*}
\end{definition}
\renewcommand{\theaxiom}{3'}
\begin{axiom}\label{ax:action_triangular_time_stable}
    Let the family of mappings $\bar{\omega}_t\colon \Xx \to \Xx$ indexed by $t \in [0, T]$ be the transition rule as in \cref{def:discr_evol}.
    Then, we require that  $\bar{\omega}_{t}$ is compatible with an action $c_t$, which has
    the following property: For all $\varepsilon, C > 0, t \in [0, T], K \subseteq \Xx$ compact,
    there exists $\Delta > 0$ such that:
    \begin{align}
        \forall x_1, \dots, x_n \in K,\,\forall t_1, \dots, t_n \in [t - \Delta, t + \Delta]:\ 
        c_t(x_1, x_n) \leq \sum_{i=1}^{n-1} c_{t_i}(x_i, x_{i+1}) + \varepsilon \label{eq:triangular_inequality}
    \end{align}
    whenever
    \begin{equation} \label{eq:upper_bound_traingular_time_stable}
        \sum_{i=1}^{n-1} c_{t_i}(x_i, x_{i+1}) \leq C.
    \end{equation}
    Furthermore, for $x, x' \in \Crit \coloneqq \{y \in \Xx \mid |  \partial E_t  |(y) = 0\}$,
    we have that $c_t(x, x') = 0$  if and only if $x$ and $x'$ belong to the same path-connected component of $\Crit$.
\end{axiom}
\addtocounter{axiom}{-1}

\begin{remark}
    We observe that \cref{def:action} and \cref{ax:action_triangular_time_stable} implicitly enforce in the transition rule some regularity with respect to the time variable.
\end{remark}
 
The following lemmas hold.

\begin{lemma}\label{lemma:axiom3'-axiom1}
\cref{ax:action_triangular_time_stable} implies \cref{ax:decreasing}.
\end{lemma}

\begin{proof}
As the compatibility of $\bar \omega_t$ with the action gives $E_t(x) - E_t\big(\bar{\omega}_t(x)\big) = c_t\big(x, \bar{\omega}_t(x)\big)$ for every $x\in \Xx$ and for every $t\in [0,T]$ (cf.~\cref{def:action}),  \cref{ax:decreasing} follows from the nonnegativity of $c_t$.
\end{proof}

\begin{lemma}\label{lemma:action_implies_traj}
    Let us assume \cref{ass:metric_space,ass:reg_E,ass:der_time,ass:unif_coerc}.
    Given a non-negative decreasing sequence $(\delta_n)_n$ such that $\delta_n\to 0$ as $n\to\infty$, let $ \eta^{\delta_n} \colon [0,T]\to\Xx$ be the discrete quasistatic evolutions constructed according to Definition~\ref{def:discr_evol} and starting from $ x_0 \in \Xx$.
    Let us further assume that along $( \eta^{\delta_n} )_n$ the convergences reported in~\eqref{eq:convergences} hold and that  \cref{ax:gradient,ax:decreasing,ax:action_triangular_time_stable} are satisfied.
    For every $t\in [0,T]$, let us consider  $x_{1}, x_{2} \in \Xx$ and two sequences $(t_1^n)_n,(t_2^n)_n\subset[0,T]$ such that $t_1^n\leq t_2^n$ for every $n$, $t_1^n,t_2^n\to t$ as $n\to\infty$, and such that $ \eta^{\delta_n}  (t_1^n)\to  x_1 $ and $ \eta^{\delta_n}  (t_2^n)\to  x_2$ as $n\to\infty$.
    Then,
    \begin{equation} \label{eq:action_mu_inequality}
        \bar \mu(\{t\}) \geq c_t( x_1, x_2 ), \qquad E_t( x_1 ) - E_t( x_2 )\geq c_t( x_1, x_2 ),
    \end{equation}
    i.e., \cref{ax:limits_traject} holds.
\end{lemma}
Before we prove \cref{lemma:action_implies_traj}, let us remark how it characterizes the energy jumps of the limit curve $\hat \eta$.
\begin{corollary}\label{cor:action_implies_traj}
    Let us assume  \cref{ass:metric_space,ass:reg_E,ass:der_time,ass:unif_coerc,ass:conn_comp} and that \cref{ax:decreasing,ax:gradient,ax:action_triangular_time_stable} are satisfied. 
    Then, the limit trajectory $\hat \eta$ constructed according to \cref{thm:traj_convergence}
    fulfills, for all $t \in [0, T]$:
    \begin{equation*}
        \hat E\big(\hat \eta^+(t)\big) -  \hat E\big(\hat \eta^-(t)\big) = \bar \mu(\{t\}) = \hat c_t\big(\hat \eta^+(t), \hat \eta^-(t)\big),
    \end{equation*}
    where $\hat c_t\big([(t, x)], [(t, x')]\big) \coloneqq c_t(x, x')$.
\end{corollary}
\begin{proof}
    The conclusion of \cref{lemma:action_implies_traj} is a strengthening of \cref{ax:limits_traject}.
    For the last part of the corollary,
    note that \cref{ax:action_triangular_time_stable} implies in particular that $c_t$
    fulfills the triangle inequality for a fixed $t$. Together
    with the last part of \cref{ax:action_triangular_time_stable}, this ensures that $\hat c_t$ factors through $\Xx$ as described by $\hat c_t$.
    Combining \cref{eq:action_apriori_inequality} with
    \cref{thm:traj_convergence}, we get that
    \begin{equation*}
        \hat E\big(\hat \eta^+(t)\big) -  \hat E\big(\hat \eta^-(t)\big) = \bar \mu(\{t\}) \leq \hat c_t\big(\hat \eta^+(t), \hat \eta^-(t)\big).
    \end{equation*}
    Further combining this inequality with the reverse inequality in \cref{eq:action_mu_inequality} finishes the proof.
\end{proof}
\begin{proof}[Proof of \cref{lemma:action_implies_traj}]
     Let us fix $t \in [0, T]$ and $\varepsilon >0$. From \cref{lem:bound_traj} we have that there exists a compact set $K \subset \Xx$ such that ${\rm Im} (\eta^{\delta_{n}}) \subset K$ for every $n \in \mathbb{N}$. Furthermore, by \cref{prop:energy_balance} and since $E$ and $\partial_t E$ are bounded in $K$, we can choose $C$ such that $\mu^{\delta_n}((0, T]) \leq C$ for all $n$. In particular, we notice that for every $n$
    \begin{displaymath}
         \sum_{t'_{i} \in J^{\delta_{n}}} c_{t_i'}\left(\eta^{\delta_n}(t'_i), \eta^{\delta_n}(t'_{i+1})\right) \leq \mu^{\delta_n}((0, T]) \leq C\,.        
    \end{displaymath}
    Let $\Delta>0$ be as in \cref{ax:action_triangular_time_stable}.For any $t_1, t_2$ such that
    $t - \Delta \leq t_1 < t_2 \leq t + \Delta$  we thus have that 
    \begin{align*}
        c_t\left( \eta^{\delta_n}(t_1), \eta^{\delta_n} (t_2)\right) & \leq \sum_{i=1}^{n-1} c_{t_i'}\left( \eta^{\delta_n}(t'_i), \eta^{\delta_n}(t'_{i+1})  \right) + \varepsilon \\
                                                        & = \mu^{\delta_n}((t_1, t_2]) + \varepsilon,
    \end{align*}
    where $t'_i \in J^{\delta_n}$ are the jump points of  $\eta^{\delta_n}$  in the interval $(t_1, t_2]$.
    Hence, for all $0 < \Delta' \leq \Delta$ it holds
    \begin{align*}
        \bar\mu([t-\Delta',t+\Delta']) &
        \geq \limsup_{n\to\infty} \mu^{\delta_n}([t-\Delta',t+\Delta'])
        \geq \limsup_{n\to\infty}
        \mu^{\delta_n}((t_1^n,t_2^n])                                          \\
                                        & \geq \limsup_{n\to\infty}
        c_t\left( \eta^{\delta_n}(t_1^n), \eta^{\delta_n}(t_2^n)  \right)  - \varepsilon \\
                                        & \geq c_t( x_1,x_2 ) - \varepsilon,
    \end{align*}
    where we used the lower semicontinuity of $c_t$ in the last line.
    If we let $\Delta' \rightarrow 0$, we see that $\bar\mu(\{t\}) \geq c_t(  x_1,x_2 ) - \varepsilon$.
    Taking the limit as $\varepsilon \rightarrow 0$, we get the first inequality in~\eqref{eq:action_mu_inequality}.
    For the second one, we observe that
    \begin{equation*}
        \begin{split}
            E_t( x_2) - E_t( x_1 ) & = \lim_{n\to\infty} \left( E_{t_2^n}\big( \eta^{\delta_n} (t_2^n)\big) - E_{t_1^n}\big( \eta^{\delta_n}  (t_1^n)\big) \right)= \lim_{n\to\infty} \left( \E_{t_2^n}( \eta^{\delta_n} ) - \E_{t_1^n}( \eta^{\delta_n}  ) \right) \\
                                & = \lim_{n\to\infty}
            \left( \int_{t_1^n}^{t_2^n} \mathcal{D}^{\delta_n}(s)\diff s -  \mu^{\delta_n}((t_1^n,t_2^n]) \right) =
            \lim_{n\to\infty} -\mu^{\delta_n}((t_1^n,t_2^n]),
        \end{split}
    \end{equation*}
    and we conclude by using the same arguments as before.
\end{proof}
To\label{part:start_of_informal_proof} construct actions fulfilling \cref{ax:action_triangular_time_stable},
we will now investigate a possible proof strategy to show that the transition rule arising from
the gradient flow, as in \cref{eq:evolution_rule_of_gradient_flow,eq:gradient_flow,eq:gradient_flow_action,eq:action_equality},
fulfills \cref{ax:action_triangular_time_stable}.
Discussing the proof strategy will allow us to distill sufficient properties
for more general transition rules for which \cref{ax:action_triangular_time_stable} holds true. \\
The resemblance of \cref{eq:triangular_inequality} in \cref{ax:action_triangular_time_stable} to the triangle inequality
already suggests a strategy to show that the transition rule~\eqref{eq:evolution_rule_of_gradient_flow} corresponding to the gradient flow fulfills \cref{ax:action_triangular_time_stable}:
Similar to the usual approach for proving the triangle inequality, we could attempt to
concatenate near-optimal curves from $x_i$ to $x_{i+1}$ to construct a competitor curve from
$x_1$ to $x_n$.
At first glance, this na\"ive proof strategy appears promising. Suppose that near-optimal curves
$\gamma_i \colon  [a_i, b_i] \rightarrow \mathbb{R}^n$ are chosen such that $b_i = a_{i+1}$ and denote their
concatenation by $\bar{\gamma}$. From \cref{eq:gradient_flow_action}, we have
\begin{align*}
    c_t(x_1, x_n) \leq    & \, \frac{1}{2} \int_{a_1}^{b_{n-1}} \left(  |\partial E_t|^{2}  (\bar{\gamma}(s))
    + |\dot{\bar{\gamma}}(s)|^2 \right) \diff s                                                            \\
    = \sum_{i=1}^{n-1}    & \, \frac{1}{2} \int_{a_i}^{b_i} \left(  |\partial E_t|^{2}  (\gamma_i(s))
    + |\dot{\gamma}_i(s)|^2 \right) \diff s                                                                     \\
    = \sum_{i=1}^{n-1}    & \, \frac{1}{2} \int_{a_i}^{b_i} \left(  |\partial E_{t_i}|^{2}  (\gamma_i(s))
    + |\dot{\gamma}_i(s)|^2 \right) \diff s                                                                     \\
                            & \hfill +
    \underbrace{ \sum_{i=1}^{n-1} \frac{1}{2} \int_{a_i}^{b_i} (  |\partial E_{t}|^{2}  (\gamma_i(s)) -  |\partial E_{t_i}|^{2}  (\gamma_i(s)) )
    \diff s }_{ \varepsilon' }                                                                                  \\
    \leq \sum_{i=1}^{n-1} & \, \left( c_{t_i}(x_i, x_{i+1}) + \varepsilon_i \right) + \varepsilon'.
\end{align*}
Here, $\varepsilon_i$ can be controlled by selecting near-optimal curves $\gamma_i$, and
$\varepsilon'$ can be managed by choosing $\Delta$ sufficiently small. However, it is impossible
to select a $\Delta$ that makes $\varepsilon'$ \emph{uniformly} small because we cannot control
the time spans $b_i - a_i$ of the curves $\gamma_i$. Moreover, $\varepsilon'$ is expected to grow with $n$,
yet we require a $\Delta$ that is effective for all $n$.
To be able to control $\varepsilon'$, we need to modify $\bar{\gamma}$
such that it does not spend too much time in areas where
$|\partial E_t|^2$ is significantly larger than $|\partial E_{t_i}|^2$.
This way, we hope to ensure that $\varepsilon'$ is bounded by a constant. \\
To modify $\bar{\gamma}$ to enable such a bound, our initial idea is to partition $K' \subseteq \Xx$---where $K'$ is some compact such that $\im(\bar{\gamma}) \subseteq K'$---into two
distinct regions:
\begin{itemize}[label={}]
    \item \textbf{Expensive Region}: Areas where $\frac{1}{2}  |\partial E_t|  (x)$---and, for sufficiently
            small $\Delta$, also $\frac{1}{2}  |\partial E_{t_i}|^{2}  (x)$---exceeds a certain threshold
            $P_{\mathrm{exp}}$.
    \item \textbf{Cheap Region}: Areas where $\frac{1}{2}  |\partial E_t|^{2}  (x)$ is below a threshold
            $P_{\mathrm{cheap}}$.
\end{itemize}
We leave $\bar{\gamma}$ unchanged in the expensive region,
following along the curves $\gamma_i$. Whenever $\bar{\gamma}$ enters the cheap region,
we take a shortcut to the point where $\bar{\gamma}$ exits the cheap region (see \cref{fig:proof_shortcut_candidate_curve}).
The total time spent in the
expensive region by $\bar{\gamma}$ is bounded, which consequently bounds $\varepsilon'$.
This is because, if we call $S_{\text{exp}}^i \subseteq [a_i, b_i]$ the time spent in the expensive region
by $\gamma_i$, we have, using the bound $C$ assumed in \cref{ax:action_triangular_time_stable} in \cref{eq:upper_bound_traingular_time_stable}:
\begin{align*}
    \left|\bigcup_i S_{\mathrm{exp}}^i\right| & = \frac{1}{P_{\mathrm{exp}}} \sum_i \int_{S_{\mathrm{exp}}^i} P_{\mathrm{exp}}\diff s
    \leq \frac{1}{P_{\mathrm{exp}}} \sum_i \int_{S_{\mathrm{exp}}^i} \frac{1}{2}  |\partial E_{t_i}|^{2}  (\gamma_i(s)) \diff s                                                                       \\
    & \leq \frac{1}{P_{\mathrm{exp}}} \sum_i \frac{1}{2}\int_{a_i}^{b_i}  \left( |\partial E_{t_i}|^{2}  (\gamma_i(s)) + |\dot\gamma_i(s)|^2\right)\diff s \\
    & \leq \frac{1}{P_{\mathrm{exp}}} \sum_i \big( c_{t_i}(x_i, x_{i+1}) + \varepsilon_i\big)
    \leq \frac{C + \sum_i \varepsilon_i}{P_{\mathrm{exp}}}.
\end{align*}
Since we can control $\varepsilon_i$ by choosing near-optimal curves $\gamma_i$, we can find a bound of the total time spent
in the expensive region---independendly of $n$.
Thus, we can make $\varepsilon'$ arbitrarily small by choosing $\Delta$ sufficiently small.
However, controlling the cost of shortcuts through the cheap region remains challenging.

To address this, we introduce a strategy of how to choose the two regions: We enlarge the set of critical
points in $K$---where $|\partial E_t| = 0$---by a radius $r$, forming a cheap region $\underline{\Crit}_r$
where $\frac{1}{2}  |\partial E_t|^{2}   (x) \leq P_{\text{cheap}}  \coloneqq \max_{x \in \underline{\Crit}_r} \frac{1}{2}  |\partial E_t|^{2}  (x)$. Traversing through sets of critical
points incurs no cost, which means that the cost of each shortcut is
determined by the part of the shortcut which does \emph{not} go through of critical points.
By taking the shortest route to and from the next critical point at the beginning and the end
of each shortcut, respectively, we can control the cost of each shortcut by $r$ and $P_{\text{cheap}}$.
Furthermore, we can make both $P_{\mathrm{cheap}}$ and $r$ small by choosing $r$: Due to the continuity of
$|\partial E_t|$, $P_{\text{cheap}} \rightarrow 0$ as $r \rightarrow 0$. \\
What remains to be bounded is the \emph{number} of shortcuts.
Because the set of critical points can be expressed as the disjoint union of well-seperated compacts by \cref{ass:conn_comp},
$K$ contains a finite number of connected components of critical points.
Enlarging each of these components by $r$ does not increase the number of connected components,
so $\underline{\Crit}_r$ contains a finite number of connected components as well.
Additionally, by consistently taking a shortcut to the last point where $\bar{\gamma}$ exits a connected component of $\underline{\Crit}_r$, we ensure that each
connected component of $\underline{\Crit}_r$ is traversed at most once---limiting the number of shortcuts.\\
Finally, leveraging the continuity of $|\partial E_t|$, $\frac{1}{2}  |\partial E_t|^{2} (x) $ is
bounded below in $K' \setminus \underline{\Crit}_r$, which makes $K' \setminus \underline{\Crit}_r$ an expensive region and enables the application of our initial strategy.\label{part:end_of_informal_proof}
\begin{figure}[H] 
    \centering 
    \pgfdeclaredecoration{raise close}{start}{
  \state{start}[width=+0pt,next state=iterate,persistent precomputation=\pgfdecoratepathhascornerstrue]
  {\pgfpathmoveto{\pgfpoint{0pt}{\pgfdecorationsegmentamplitude}}}
  \state{iterate}[width=+\pgfdecorationsegmentlength]{\pgfpathlineto{\pgfpoint{0pt}{\pgfdecorationsegmentamplitude}}}
  \state{final}{\pgfpathclose}
}
\pgfdeclaredecoration{raise open}{start}{
  \state{start}[width=+0pt,next state=iterate,persistent precomputation=\pgfdecoratepathhascornerstrue]
  {\pgfpathmoveto{\pgfpoint{0pt}{\pgfdecorationsegmentamplitude}}}
  \state{iterate}[width=+\pgfdecorationsegmentlength]{\pgfpathlineto{\pgfpoint{0pt}{\pgfdecorationsegmentamplitude}}}
  \state{final}{}
}
\begin{tikzpicture}[scale=1.5, every node/.style={font=\small}, pics/legend entry/.style={code={%
            \draw[pic actions]
            (-0.5,0.25) sin (-0.25,0.4) cos (0,0.25) sin (0.25,0.1) cos (0.5,0.25);}}]

  \def\radius{0.5} 

  \definecolor{Dcolor}{RGB}{100,149,237} 
  \definecolor{Drcolor}{RGB}{173,216,230} 
  \definecolor{outercolor}{RGB}{220,220,220} 
  \definecolor{resultcurve}{RGB}{220,0,0} 

  \fill[outercolor, rounded corners=1.5cm] (-3.5,-1.7) rectangle (1.7,2.7);

  \def\KcoordsFirst{
    (-1,0.5)
    (-0.2,0.9)
    (0.2,1.7)
    (-0.2,2.1)
    (-1,1.7)
    (-1.8,2.1)
    (-2.2,1.7)
    (-1.8,1.3)
    (-2.2,0.9)
    (-1.8,0.5)
  }
  \def\KcoordsSecond{
    (2.7,-1.0)
    (2.7,0.2)
    (1.9,0.6)
    (1.5,0.2)
    (1.5,-0.6)
  }
  \foreach \Kcoords in {\KcoordsFirst} {
      \begin{scope}
        \draw[Drcolor, line width=2*\radius cm, opacity=0.8]
        plot [smooth cycle, tension=0.8] coordinates {
            \Kcoords
          };
      \end{scope}
      \begin{scope}
        \fill[Dcolor, opacity=1]
        plot [smooth cycle, tension=0.8] coordinates {
            \Kcoords
          };
      \end{scope}
    }

  \coordinate (X1) at (-3, -1);
  \coordinate (X1in) at (-1.45, 0.1);
  \coordinate (X1inBar) at (-1.45, 0.438);
  \coordinate (X2) at (-1.1, 1);
  \coordinate (X3) at (0.5, 0.2);
  \coordinate (X1out) at (0.308, 1);
  \coordinate (X1outBar) at (0.03, 1.18);
  \coordinate (X4) at (1, 1);

  \draw[
    decorate,
    decoration={raise open, amplitude=0.017cm,segment length=1pt},
  ] (X1) to[in=-50] (X1in)
  to[out=130, in=-90] (-2, 0.5)
  to[out=90, in=180] (X2);
  \draw[
    thick,
    decorate,
    decoration={raise open, amplitude=-0.017cm,segment length=1pt},
    color=resultcurve
  ] (X1) to[in=-50] (X1in)
  to (X1inBar)
  to (X1outBar)
  to (X1out)
  to[out=0](X4);

  \draw[
    decorate,
    decoration={raise open, amplitude=0.017cm,segment length=1pt},
    smooth cycle, tension=0.8
  ]
  (X2)
  to[out=-70, in=180] (0.0, -0.5)
  to[out=0, in=-90] (1.2, -0.5)
  to[out=90, in=-90] (X3)
  ;

  \draw[
    decorate,
    decoration={raise open, amplitude=0.017cm,segment length=1pt},
  ] (X3) to[out=0, in=180] (X1out) to[out=0] (X4);

  \draw[fill=outercolor] (X1) circle (0.05);
  \node[below] at (X1) {$X_1$};
  \draw[fill=Dcolor]  (X2) circle (0.05);
  \node[above] at (X2) {$X_2$};
  \draw[fill=outercolor] (X3) circle (0.05);
  \node[left] at (X3) {$X_3$};
  \draw[fill=outercolor]  (X4) circle (0.05);
  \node[right] at (X4) {$X_4$};
  \draw (X1in) circle (0.02);
  \node[right] at (X1in) {$x_{in}^1$};
  \draw (X1inBar) circle (0.02);
  \node[left] at (X1inBar) {$\overline{x}_{in}^1$};
  \draw (X1out) circle (0.02);
  \node[above right] at (X1out) {$x_{out}^1$};
  \draw (X1outBar) circle (0.02);
  \node[above] at (X1outBar) {$\overline{x}_{out}^1$};

  \node[above left] at (-2, -0.5) {$\gamma_1$};
  \node[below] at (0, -0.5) {$\gamma_2$};
  \node[right] at (0.4, 0.6) {$\gamma_3$};
  \matrix [below right] at (2.3,2) {
    \fill[Dcolor] (-0.1,-0.36) rectangle (0.1,-0.16) ;              & \node[color=black] {$K_1\colon \frac{1}{2} |\partial E_t|^{2} (x) = 0$};                                                                         \\
    \fill[Drcolor] (-0.1,-0.36) rectangle (0.1,-0.16) ;             & \node[color=black] {$K_r\colon \frac{1}{2} |\partial E_t|^{2}(x) \leq P_{cheap}$};                                          \\
    \fill[outercolor] (-0.1,-0.36) rectangle (0.1,-0.16) ;          & \node[color=black] {$K' \backslash K_r\colon \frac{1}{2} |\partial E_{t_i}|^{2}(x) \geq P_{exp}$}; \\
    \draw[thick, color=resultcurve] (-0.1, -0.26) to (0.1, -0.26) ; & \node[color=black] {candidate curve};                                                                                                                    \\
  };
\end{tikzpicture}
    \captionsetup{justification=centering, width=\textwidth}
    \caption{
        The proof strategy for \cref{lemma:good_curves_fulfill_continuous_triangular_inequality} and the scheme
        for shortening competitor curves: We create a shortcut from the first point where a curve enters a
        neighborhood $B_r(\Crit_i)$ to the last point where one exits $B_r(\Crit_i)$, where $\Crit_i$ is a connected component of critial points. This approach ensures
        that $B_r(\Crit_i)$ is traversed at most once. In the formal proof, rather than constructing
        a new competitor curve by concatination, we employ the triangle inequality to estimate
        $c_t(X_1, X_4) \leq \Cc_t(X_1, p_{\mathrm{in}}^1) + \Cc_t(p_{\mathrm{in}}^1, \overline{p}_{\mathrm{in}}^1) + \Cc_t(\overline{p}_{\mathrm{in}}^1, \overline{p}_{\mathrm{out}}^1) + \Cc_t(\overline{p}_{\mathrm{out}}^1, p_{\mathrm{out}}^1) + \Cc_t(p_{\mathrm{out}}^1, X_4)$.
    }\label{fig:proof_shortcut_candidate_curve}
\end{figure}
From this discussion, we are now in a position to distill the essential properties of the gradient flow
that allow for a proof of \cref{ax:action_triangular_time_stable}.
As those properties are satisfied by a range of transition rules---whether they are discrete or additionally incorporate momentum---, the application of the resulting framework is not limited to the gradient flow.
Aiming at a unified exposition of such actions, we consider curves in the phase space.
As a guide to the formal proof, \cref{fig:proof_shortcut_candidate_curve} should be
understood as a schematic representation of this phase space. \\
In the following, we denote $(a, b] \cap \Ii$ and $[a, b] \cap \Ii$ as $(a, b]_{\Ii}$ and $[a, b]_{\Ii}$, respectively, whenever
$\Ii \in \{\R, \Z\}$ and $a, b \in \Ii$.
\begin{definition}\label{def:generated_by_curves}
    An action $c_t$ is \emph{generated by curves} if there exists an index set
    ${\Ii \in \{\Z, \R\}}$, a \emph{phase space} $\Pp$ equipped with a metric $d_{\Pp}$ and functions
    \begin{equation*}
        \begin{array}{rlrlrlrl}
            \Gamma: &  & \Pp \times \Pp                                          & \, &  & \rightarrow & \, & 2^{\{ \gamma: [a, b]_\Ii \rightarrow \Pp \mid a < b \in \Ii\}} \\
                    &  & (p_1, p_2)                                              &    &  & \mapsto     &    & \Gamma(p_1, p_2),                                              \\
            i:      &  & \Xx                                                     &    &  & \rightarrow &    & \Pp,                                                           \\ [0.3em]
            \Cost:  &  & \Pp \times [0, T]                                            &    &  & \rightarrow &    & [0, +\infty)                                                   \\
                    &  & (p, t)                                                  &    &  & \mapsto     &    & \Cost_t(p),                                                    \\
        \end{array}
    \end{equation*}
     such that  for every $\gamma \in \Gamma(p_1, p_2)$ it holds $\gamma(a) = p_1$ and $\gamma(b) = p_2$, $i$ is a closed immersion, $\Cost$ is continuous, and $c_t \colon \Xx \times \Xx \to [0, \infty)$ is given as
    \begin{equation} \label{eq:action_infimum}
    \begin{split}
        & c_t(x_1, x_2) = \mathfrak{c}_{t} (i(x_{1}), i(x_{2}))\qquad \text{for $x_{1}, x_{2} \in \Xx$},\\
        & \mathfrak{c}_{t} (p_{1}, p_{2}) \coloneqq \inf_{\gamma \in \Gamma(p_{1}, p_{2})} \int_{(a, b]_\Ii} \Cost_t(\gamma_s) \diff \nu(s) \qquad \text{for $p_{1}, p_{2} \in \Pp$},\\
        \end{split}
    \end{equation}
    where $\nu$ is the Lebesgue measure if $\Ii = \R$ and the counting measure if $\Ii = \mathbb{Z}$.
    We call the value of $\Gamma$ the \emph{set of admissible curves} for $p_1,p_2 \in \Pp$, and the value of $\Cost$ the \emph{price}
    for $p \in \Pp$ and $t\in [0,T]$. \\
    Moreover, we require that the following properties hold:
    \begin{properties}
        \item \label{prop:curve_restriction} Restrictions, concatenations and shifts of admissible curves are again admissible curves.
        In other words, for $p_1, p_2, p_3 \in \Pp$ and $a \leq s \leq b \in \Ii$, the following holds:
        \begin{itemize}
            \item \emph{Restrictions:} For all $\gamma \in \Gamma(p_1, p_3)$ such that $\dom(\gamma) = [a, b]_\Ii$ and $\gamma(s) = p_2$, we have that $\gamma|_{[a, s]_\Ii} \in \Gamma(p_1, p_2)$ and $\gamma|_{[s, b]_\Ii} \in \Gamma(p_2, p_3)$.
            \item \emph{ Concatenations:} For all $\gamma \in \Gamma(p_1, p_2)$ and $\gamma' \in \Gamma(p_2, p_3)$ such that $\dom(\gamma) = [a, s]_\Ii$ and $\dom(\gamma') = [s, b]_\Ii$, we have that $\gamma \cup \gamma' \in \Gamma(p_1, p_3)$,
                  where $\gamma \cup \gamma'\colon [a, b]_\Ii \rightarrow \Pp$ is defined as
                  \begin{equation*}
                      \gamma \cup \gamma'(s) \coloneqq
                      \begin{cases}
                          \gamma(s)  & \text{if } s \in [a, s]_\Ii, \\
                          \gamma'(s) & \text{if } s \in (s, b]_\Ii.
                      \end{cases}
                  \end{equation*}
            \item \emph{Shifts:} For all $\gamma \in \Gamma(p_1, p_2)$ such that $\dom(\gamma) = [a, b]$ and all $r \in \Ii$, we have that $\gamma'\colon  [a - r, b - r] \rightarrow \Pp$ with $\gamma'(t) = \gamma(t + r)$ is in $\Gamma(p_1, p_2)$.
        \end{itemize}
        \item\label{prop:cost_nonzero} For $p \in \Pp$, $\Cost_t(p)=0$ if and only if there exists $x \in \Xx$ such that $i(x) = p$ and
        $ |\partial E_t|(x) = 0$.
        \item\label{prop:no_cost_in_within_conn_component}
        For $x_1 , x_2 \in \Xx$ with $x_1\neq x_2$, $c_t(x_1, x_2) = 0$ if and only if $i(x_1)$ and $i(x_2)$
        lie in the same connected component of $\{p \in \Pp \mid \Cost_t(p) = 0\}$.
        \item\label{prop:curves_coercive} For all $K \subseteq \Pp$ compact, there exists some $K' \subseteq \Pp$ compact
        such that: For all $p_1, p_2 \in K$ and all $\varepsilon > 0$,
        there exists $\gamma \colon [a, b]_\Ii \rightarrow \Pp$,  $\gamma \in \Gamma(p_1, p_2)$ such that $\im(\gamma) \subseteq K'$ and 
        \begin{equation*}
            \int_{(a, b]_\Ii} \Cost_t\big(\gamma(s)\big) \diff\nu(s) \leq  \mathfrak{c}_t  (p_1, p_2) + \varepsilon.
        \end{equation*}
        This means that for points $p_1, p_2 \in K$, we may take the infimum in \cref{eq:action_infimum}
        over only those admissible curves whose image lies in $K'$.
        \item\label{prop:cheap_fuel_cheap_action} For all $K \subseteq \Pp$ compact and $\varepsilon > 0$, there exists some $L, P > 0$ such that: For all $p_1, p_2 \in K$
        with $\Cost_t(p_1) \leq P$, $\Cost_t(p_2) \leq P$, and $d_\Pp(p_1, p_2) \leq L$, we have that $ \Cc_t  (p_1, p_2) \leq \varepsilon$.
        Informally, this means that the value of the action is small between two points which are close to each other and where $\Cost_t$ is small.
    \end{properties}
\end{definition}

\begin{remark}
    Notice that in what follows, the expressions of the functions appearing in~\cref{def:generated_by_curves} will be energy-dependent.
\end{remark}

\subsection{Examples of transition rules} \label{sec:examples_of_evolution_rules}
Before we show how the abstract \cref{def:generated_by_curves} relates to \cref{ax:action_triangular_time_stable},
we present some examples of actions---and the corresponding transition rules---that are generated by curves.
The fact that those actions indeed satisfy the required properties is contained in \cref{lemma:gradient_flow_summary_lemma,lemma:MMS_summary_lemma,lemma:BDF2_summary_lemma}, whose proofs are postponed to \cref{sec:examples_of_evolution_rules_elaborate}.
\subsubsection{The gradient flow}
\newcommand{\GFprime}{^{^{_{F'}}}}
\newcommand{\ruleGF}{\bar{\omega}\GF}
\newcommand{\SpeedGF}{\Speed\GF}
\newcommand{\PriceGF}{\Price\GF}
\newcommand{\CostGF}{\Cost\GF}
\newcommand{\PpGF}{\Pp\GF}
\newcommand{\OLCostGF}{\bar{\Cost}\GF}
\newcommand{\cGFprime}{c\GFprime}
\begin{definition} \label{def:gradient_flow_evolution_rule}
    A \emph{transition rule of the gradient flow} is a transition rule $\ruleGF_t\colon \Xx\to \Xx$ for which there
    exists, for each $x \in \Xx$, a curve  $\phi \in AC([0, \infty) , \Xx)$ and a sequence $(s_k)_{k\geq 1}$ with $s_k\nearrow+\infty$ as $k\to\infty$ such that
    $\phi(0) = x$, $\lim_{k \to \infty} d\left( \phi(s_k), \ruleGF_t(x) \right) = 0$, and, for all $s \in [0, \infty)$,
    \begin{equation*}
        E_t(x) - E_t(\phi(s)) = \frac12 \int_0^s  |\partial E_t|^{2}\big(\phi(\sigma ) \big)  + |\dot{\phi}( \sigma  )|^2 \diff  \sigma.
    \end{equation*}
    The \emph{action of the gradient flow} is the action $\cGF_t$ given by
    \begin{equation}\label{eq:def_gradient_flow_action}
        \cGF_t(x_1,x_2)\coloneqq \inf\left\{
        \frac{1}{2}\int_a^b \left(  |\partial E_t|^{2} \big(\phi(s)\big) 
        + |\dot \phi (s)|^2\right) \diff s \;\middle|\;
        \substack{a < b \in \R \\ \phi \in AC([a,b],\Xx) \\ \phi(a) = x_1, \phi(b) = x_2}
        \right\}.
    \end{equation}
\end{definition}
\begin{proposition} \label{lemma:gradient_flow_summary_lemma}
    Let us assume \cref{ass:unif_coerc,ass:metric_space,ass:reg_E,ass:path+Lipsch,ass:conn_comp,ass:path+Lipsch}. A transition rule $\ruleGF_t$ of the gradient flow is
    compatible---as defined in \cref{def:action}---with the action $\cGF_t$ of the gradient flow.
    Furthermore, the action $\cGF_t$ is generated by curves as per \cref{def:generated_by_curves},
    where $\Ii = \R$ and where we use the following definitions:
    \begin{align}
        \PpGF                     & \coloneqq \Xx \times \R^+ \text{, where } d_{\PpGF}\big((x_1, v_1), (x_2, v_2)\big) = d(x_1, x_2) + |v_1 - v_2|           \label{eq:def_phase_space_gf} \\
        i\GF (x)                       & \coloneqq (x, 0)                                                                                                \label{eq:embedding_gf}       \\
        \Gamma(p_1, p_2)          & \coloneqq \left\{ \gamma \colon [a, b] \rightarrow \PpGF \,\middle|\, \substack{a < b \in \R, \gamma(a) = p_1,\, \gamma(b) = p_2,                              \\ \gamma_x \in AC([a, b]; \Xx),\, |\dot{\gamma}_x| \equiv  \gamma_v \text{a.e.}}\right\} \label{eq:gamma_gradient_flow}\\
        \CostGF_t\big((x, v)\big) & \coloneqq  \frac{1}{2}  |\partial E_t|^{2}  (x)  + \frac{1}{2} v^2 \label{eq:def_cost_gf}\\
         \Cc^{F}_{t} (p_{1}, p_{2}) & \coloneqq \inf_{\gamma \in \Gamma (p_{1}, p_{2})} \int_{a}^{b} \CostGF_{t} (\gamma(s)) \diff s \,. \label{eq:action-lift-GF}
    \end{align}
    Here, for every interval $I \subseteq \R$, we denote the components of a curve $\gamma\colon I \to \PpGF$ by $\gamma_x\colon I \to \Xx$ and $\gamma_v\colon I \to \R^+$.
\end{proposition}
\subsubsection{The minimizing movement scheme}
\newcommand{\MMS}{^{^{_M}}}
\newcommand{\superscriptleft}[2]{%
    \mathllap{\raisebox{0.35ex}{$^{#1}$}\!\!}#2%
}
\newcommand{\raisedTau}[1]{\superscriptleft{\tau}{#1}}
\newcommand{\ruleM}{\superscriptleft{\tau}{\bar{\omega}\MMS}}
\newcommand{\SpeedM}{\Speed\MMS}
\newcommand{\PriceM}[1][\tau]{\superscriptleft{#1}{\Price\MMS}}
\newcommand{\CostM}[1][\tau]{\superscriptleft{#1}{\Cost}\MMS}
\newcommand{\PpM}{\Pp\MMS}
\newcommand{\OLCostM}[1][\tau]{\superscriptleft{#1}{\bar{\Cost}}\MMS}
\newcommand{\cM}[1][\tau]{\superscriptleft{#1}{c}\MMS}
\newcommand{\cMg}[1][\tau]{\superscriptleft{#1}{\mathfrak{c}}\MMS}
\newcommand{\EM}{\raisedTau{E}\MMS}
\newcommand{\ttau}{_{t,\tau}}
\newcommand{\vx}{\underline{x}}
\begin{definition}\label{def:evolution_rule_MMS}
    For $\tau > 0$, $x \in \Xx$ and $t \in [0, T]$, a \emph{minimizing movement scheme sequence starting at point $x$ at time $t$}
    is a sequence $(u_s)_{s \in \NN}$ in $\Xx$
    such that $u_0 = x$, and,
    for all $s \in \NN$,
    \begin{equation*}
        u_{s+1} \in \argmin_{y \in \Xx} \left\{  E_t(y) + \frac{1}{2\tau} d(y, u_s)^2 \right\}.
    \end{equation*}
    A \emph{transition rule of the minimizing movement scheme} is a transition rule $\ruleM_t \colon \Xx\to \Xx$
    for which there exists, for each $x \in \Xx$ and $t \in [0, T]$, a minimizing movement scheme sequence $(u_s)_{s \in \NN}$ in $\Xx$ starting at point $x$ at time $t$
    such that $\ruleM_t(x)$ is a cluster point of $(u_s)_{s \in \NN}$.
    The \emph{action of the minimizing movement scheme} is the function  $\cM_t \colon \Xx \times \Xx \to [0, \infty)$ given by
    \begin{align}\label{eq:action_infimum_MMS}
        \cM_t(x_1,x_2)\coloneqq \inf\left\{
        \sum_{s=a}^b E_t(u_s) - E\ttau\MMS(u_s) + \frac{1}{2\tau} d(u_{s+1}, u_s)^2
        \;\middle|\;
        \substack{a \leq b \in \Z, \\ (u_s)_{s \in [a, b]_\Z}, \\ u_a = x_1, u_b = x_2}
        \right\},
    \end{align}
    where we set $u_{b+1} = u_b = x_{2}$ and
    \begin{equation}\label{eq:def_E_MMS}
        E\ttau\MMS(x) \coloneqq \inf_{y \in \Xx} \left\{ E_t(y) + \frac{1}{2\tau} d(x, y)^2 \right\}.
    \end{equation}
\end{definition}

We mention that in \cite[Definition~5.14]{Scilla_2018} the authors introduced an action for the minimizing movement scheme in Euclidean setting which is fully compatible with \cref{eq:action_infimum_MMS,eq:def_E_MMS}.

\begin{proposition} \label{lemma:MMS_summary_lemma}
    Let us assume \cref{ass:metric_space,ass:reg_E,ass:unif_coerc,ass:der_time,ass:conn_comp,ass:path+Lipsch}. For $\tau > 0$, a transition rule $\ruleM_t$ of the minimizing movement scheme is compatible---as defined in \cref{def:action}---with the action $\cM_t$ of the minimizing movement scheme.\\    
     Moreover, if $\tau \leq \frac{1}{L}$,
    the  map $\cM_t\colon \Xx \times \Xx \to [0, \infty)$ is an action generated by curves as per \cref{def:generated_by_curves},
    where $ \mathbb{I} = \Z$ and where we use the following definitions:
    \begin{align}
        \PpM              & \coloneqq \Xx \times \Xx \text{, where } d_{\PpM}(\vx, \vx') = \frac{1}{2} \left(d(x_0, x'_0) + d(x_1, x'_1)\right) \label{eq:def_phase_space_mms} \\
         i\MMS (x) & \coloneqq  (x, x)                                                                                      \label{def:embedding_MMS}        \\
        \Gamma(\vx, \vx') & \coloneqq \left\{ \gamma\colon  [a, b]_\Z \to \PpM \,\middle|\, \substack{a < b \in \Z,\, \gamma(a) = \vx, \gamma(b) = \vx',                                    \\\gamma_1(s) = \gamma_0(s + 1) \text{ for } s \in [a, b-1]_\Z}\right\} \label{eq:def_GammaM}\\
        \CostM_t(\vx)     & \coloneqq E_t(x_0) - E_{t, \tau}\MMS(x_0) + \frac{1}{2\tau} d(x_0, x_1)^2,\label{eq:CostM_def}\\
         \cMg_{t} (\underline{x}, \underline{x}') & \coloneqq \inf_{\gamma \in \Gamma (\underline{x}, \underline{x}')} \, \sum_{s=a+1}^{b} \CostM_{t} (\gamma(s)), \label{eq:extended-cost-MMS}
    \end{align}
    where $E_{t, \tau}\MMS$ is defined as in \cref{eq:def_E_MMS}. Here, for every $I\subseteq \Z$ we denote the components of a curve $\gamma\colon I  \to \PpM$ by $\gamma_0$ and $\gamma_1$, and we denote the components of points $\vx, \vx' \in \PpM = \Xx \times \Xx$ by $(x_0, x_1) \coloneqq \vx$ and $(x'_0, x'_1) \coloneqq \vx'$.
\end{proposition}
\subsubsection{The BDF2 method for the gradient flow}
\newcommand{\BDF}{^{^{_B}}}
\newcommand{\ruleB}{\superscriptleft{\tau}{\bar{\omega}\BDF}}
\newcommand{\SpeedB}{\Speed\BDF}
\newcommand{\PriceB}{\raisedTau{\Price\BDF}}
\newcommand{\CostB}{\raisedTau{\Cost}\BDF}
\newcommand{\PpB}{\Pp\BDF}
\newcommand{\OLCostB}{\raisedTau{\bar{\Cost}}\BDF}
\newcommand{\cB}{\raisedTau{c}\BDF}
\newcommand{\cBg}{\raisedTau{\mathfrak{c}}\BDF}
\newcommand{\EB}{\raisedTau{E}\BDF}
\begin{definition}\label{def:evolution_rule_BDF2}
    For $\tau > 0$, $x \in \Xx$ and $t \in [0, T]$, a \emph{BDF2 discretization sequence starting at point $x$ at time $t$} is a sequence $(u_s)_{s \in \NN}$ in $\Xx$
    such that $u_0 = x$ and,
    for all $s \in \NN$,
    \begin{equation*}
        u_{s+1} \in \argmin_{y \in \Xx} \left\{  E_t(y) + \frac{1}{\tau} d(y, u_s)^2 - \frac{1}{4\tau} d(y, u_{s-1})^2\right\},
    \end{equation*}
    where we set $u_{-1} = x$.  A \emph{transition rule of the BDF2 discretization} is a transition rule $\ruleB_t \colon \Xx\to \Xx$ for which there exists, for each $x \in \Xx$ and $t \in [0, T]$, a BDF2 discretization sequence $(u_s)_{s \in \NN}$ in $\Xx$ starting at point $x$ at time $t$
    such that $\ruleB_t(x)$ is a cluster point of $(u_s)_{s \in \NN}$.
    The \emph{action of the BDF2 discretization} is the map $\cB_t \colon \Xx \times \Xx \to [0, \infty)$ given by
    \begin{equation}\label{eq:action_infimum_BDF2}
    \begin{split}
        \cB_t(x_1,x_2)\coloneqq 
        \inf\left\{
        \sum_{s=a}^{b+1} \right.
        \bigg( &
        E_t(u_s) - E\BDF_{t, \tau}(u_s, u_{s-1}) + \frac{1}{2\tau} d(u_{s+1}, u_s)^2  \\
        & + \frac{1}{2\tau} d(u_{s}, u_{s-1})^2 - \frac{1}{4\tau} d(u_{s+1}, u_{s-1})^2 \bigg)
        \left. \;\middle|\;
        \substack{a \leq b \in \Z, \\ (u_s)_{s \in [a, b]_\Z}, \\ u_a = x_1, u_b = x_2}
        \right\},
        \end{split}
    \end{equation}
    where where we set $u_{a-1} = x_1$, $u_{b+1} = u_{b+2} = x_2$ and
    \begin{equation}\label{eq:def_E_BDF}
        E\BDF_{t, \tau}(x, x') \coloneqq \inf_{y \in \Xx} \left\{ E_t(y) + \frac{1}{\tau} d(y, x)^2 - \frac{1}{4\tau} d(y, x')^2 \right\}.
    \end{equation}
\end{definition}

To the best of our knowledge, the action for the BDF2 scheme provided by \cref{eq:action_infimum_BDF2,eq:def_E_BDF} is completely original.

\begin{proposition} \label{lemma:BDF2_summary_lemma}
    Let us sassume \cref{ass:metric_space,ass:reg_E,ass:unif_coerc,ass:der_time,ass:conn_comp,ass:path+Lipsch}. For $\tau > 0$, a transition rule $\ruleB_t$ of the BDF2 discretization is compatible---as defined in \cref{def:action}---with the action $\cB_t$ of the BDF2 discretization.\\    
    Moreover, for every $\tau \leq \frac{1}{L}$,
    the map $\cB_t \colon \Xx \times \Xx \to [0, \infty)$ is an action generated by curves as per \cref{def:generated_by_curves},
    where $ \mathbb{I} = \Z$ and where we use the following definitions:
    \begin{align}
         \PpB              & \coloneqq \Xx \times \Xx \times \Xx \text{, where } d_{\PpB}(\vx, \vx') = \frac{1}{3} \sum_{i=-1}^{1} d(x_i, x'_i)                                              \label{eq:def_phase_space_BDF} \\
         i^{B} (x)             & \coloneqq  (x, x, x)                                                                                                                          \label{eq:embedding_BDF}              \\
        \Gamma(\vx, \vx') & \coloneqq \left\{ \gamma \colon [a, b]_\Z \to  \PpB  \,\middle|\, \substack{a < b \in \Z,\, \gamma(a) = \vx, \gamma(b) = \vx',                                                                                \\\gamma_1(s-1) = \gamma_0(s) = \gamma_{-1}(s+1) \text{ for } s \in [a+1, b-1]_\Z}\right\} \label{eq:def_GammaB} \\
        \CostB_t(\vx)     & \coloneqq E_t(x_0) - E\BDF_{t, \tau}(x_0, x_{-1}) + \frac{1}{2\tau} \left(d(x_0, x_1)^2 + d(x_{-1}, x_0)^2\right) - \frac{1}{4\tau}d(x_{-1}, x_{1})^2,\label{eq:def_costB}\\
        \cBg_{t} ( \vx, \vx') & \coloneqq \inf_{\gamma \in \Gamma (\underline{x}, \underline{x}')} \, \sum_{s = a+1}^{b} \CostB_{t} (\gamma(s)), \label{eq:extendded-cost-BDF} 
    \end{align}
    where $E\BDF_{t, \tau}$ is defined as in \cref{eq:def_E_BDF}. Here, for every $I\subseteq \Z$, we denote the components of a curve $\gamma\colon I \to \PpB$ by $\gamma_{-1}$, $\gamma_0$ and $\gamma_1$, and we denote the components of points $\vx, \vx' \in \PpB = \Xx \times \Xx \times \Xx$ by $(x_{-1}, x_0, x_1) \coloneqq \vx$ and $(x'_{-1}, x'_0, x'_1) \coloneqq \vx'$.
\end{proposition}
\subsection{A rigorous proof of \texorpdfstring{\cref{ax:action_triangular_time_stable}}{Axiom \ref{ax:action_triangular_time_stable}}}
With those examples in mind, we now show that actions generated by curves fulfill \cref{ax:action_triangular_time_stable}.
We postpone the proof of the following ancillary lemma to the end of this section.
\begin{lemma}\label{lemma:triangular_inequ_on_P}
    Let $c_t \colon \Xx \times \Xx \to [0, \infty)$ be an action generated by curves as per \cref{def:generated_by_curves} and let $\Cc_{t} \colon \Pp \times \Pp \to [0, \infty)$ be the map defined in \eqref{eq:action_infimum}.
    Then, $\Cc_t$ and $c_{t}$ fulfill the triangle inequality,
    i.e., for every $p_1, p_2, p_3 \in \Pp$, every $x_{1}, x_{2}, x_{3} \in \Xx$, and every $t \in [0, T]$ we have
    \begin{equation*}
        \Cc_t(p_1, p_3) \leq \Cc_t(p_1, p_2) + \Cc_t(p_2, p_3) \qquad \text{and} \qquad c_{t} (x_{1}, x_{3}) \leq c_{t} (x_{1}, x_{2}) + c_{t} (x_{2}, x_{3})\,.
    \end{equation*}
\end{lemma}
We are now ready to prove the key lemma of this section.
\begin{lemma}\label{lemma:good_curves_fulfill_continuous_triangular_inequality}
    Let us assume \cref{ass:metric_space,ass:reg_E,ass:der_time,ass:unif_coerc,ass:conn_comp}. If a transition rule $\bar{\omega}_{t}$ is compatible with an action $c_{t}$ generated by curves, then \cref{ax:action_triangular_time_stable} holds.
\end{lemma}
\begin{proof}
    We follow the idea outlined on pages~\pageref{part:start_of_informal_proof}--\pageref{part:end_of_informal_proof}
    and illustrated in \cref{fig:proof_shortcut_candidate_curve}.
    In the first part of the proof, we obtain a $\Delta>0$ as demanded by \cref{ax:action_triangular_time_stable}.
    In the second part, we confirm that with such a choice of $\Delta>0$ and for any $x_1, \dots, x_n$,
    $t_1, \dots, t_n$ as in \cref{ax:action_triangular_time_stable}, \cref{eq:triangular_inequality} is indeed fulfilled. \\
    Let $\varepsilon, C >0$, $t \in [0, T]$ and $K \subseteq \Xx$ compact be given as in \cref{ax:action_triangular_time_stable}.
    We first notice that $i(K)$ is compact, as $i$ is continuous.  Hence, we may find a compact set $K' \subseteq \Pp$ related to $i(K)$ as prescribed by \cref{prop:curves_coercive}. 
    By \cref{prop:cost_nonzero} and \cref{ass:conn_comp}, since $i$ is a closed immersion, the connected components of $\{ p \in K' \mid \Cost_t(p) = 0\}$ are well-separated,
    and hence there are only finitely many of them.
    We enumerate those connected components as $\Crit_1, \dots, \Crit_m$ and use
    \cref{prop:cheap_fuel_cheap_action} to choose $L, \Pcheap > 0$ such that for
    all $p_1, p_2 \in K'$ with $\Cost_t(p_1) \leq \Pcheap$, $\Cost_t(p_2) \leq \Pcheap$ and $ d_{\Pp}  (p_1, p_2) \leq L$, we have that
    \begin{equation}
         \Cc_t ( p_{1}, p_{2}) \leq \frac{\varepsilon}{6(m+1)}. \label{eq:cheap_fuel_cheap_action}
    \end{equation}
    For $r > 0$, we write $\underline{\Crit}_r \coloneqq \bigcup_{j=1}^m B_r(\Crit_j)$,  where $B_{r} (\Crit_{j}) \coloneqq \{ p \in \Pp \mid d_{\Pp} (p, \Crit_{j}) <r \}$.
    By the continuity of $\Cost$ and the compactness of $K'$,
    $\sup_{p \in \overline{(\underline{\Crit}_r)}} \Cost_t(p) \rightarrow 0$ as $r \rightarrow 0$.
    Thus, we can choose $0 < r < L$ such that $\sup_{p \in \overline{(\underline{\Crit}_r)}} \Cost_t(p) \leq \Pcheap$. Moreover, by further reducing $r$ if needed, we can assume that $\overline{B_r(\Crit_j)}\cap \overline{B_r(\Crit_{j'})} = \emptyset$ whenever $j\neq j'$.
    Using a similar argument, we see that
    \begin{equation*}
        \inf_{p \in K' \setminus \underline{\Crit}_r} \Cost_t(p) > 0.
    \end{equation*}
    Even more, we can choose some $\Delta''$ such that
    \begin{equation}\label{eq:P_exp}
        \Delta_{\Cost} \coloneqq \inf_{\substack{ p  \in K' \setminus \underline{\Crit}_r \\ t' \in [t - \Delta'', t + \Delta'']}} \Cost_{t'}( p ) > 0.
    \end{equation}
    Using continuity one last time, we can choose some $\Delta'>0$ such that for all $p \in K'$ and $t' \in [t - \Delta', t + \Delta']$,
    we have that
    \begin{equation}\label{eq:delta_prime}
        |\Cost_{t'}(p) - \Cost_t(p)| \leq \varepsilon_{\Cost} \coloneqq  \frac{\varepsilon}{3} \cdot \left(\frac{C + \frac{\varepsilon}{3}}{\Delta_{\Cost}} + 2(m+1)\right)^{-1}.
    \end{equation}
    We claim that $\Delta \coloneqq \min \{ \Delta', \Delta'' \}$ is the $\Delta>0$ desired by \cref{ax:action_triangular_time_stable}.\\
    To prove this, we fix points $x_1, \dots, x_n \in K$ and instants $t_1, \dots, t_n \in [t -  \Delta  , t +  \Delta  ]$
    such that $\sum_{i=1}^{n-1} c_{t_i}(x_i, x_{i+1}) \leq C$.
    We first choose $\gamma_i \in \Gamma \big( i(x_i), i(x_{i+1}) \big)$ according to \cref{prop:curves_coercive}
    such that $\im(\gamma_i) \subseteq K'$  and
    \begin{equation} \label{eq:choice_of_curve}
        \int_{(a_i, b_i]_\Ii} \Cost_{t_i}(\gamma_i(s)) \diff \nu(s) \leq c_{t_i}(x_i, x_{i+1}) + \frac{\varepsilon}{3n},
    \end{equation}
    where $a_i < b_i \in  \mathbb{I} $  are such that $\dom(\gamma_i) = [a_i, b_i]_{ \mathbb{I}}$, so that
    \begin{equation} \label{eq:curve_triangle}
        C + \frac\e3 \geq \sum_{i=1}^{n-1} \int_{(a_i, b_i]_\Ii} \Cost_{t_i}(\gamma_i(s)) \diff \nu(s).
    \end{equation}
    We set $IS \coloneqq \{(i, s) \mid i \in \{1, \dots, n\}, s \in [a_i, b_i]_\Ii\}$ and
    \begin{align*}
        IS_{\mathrm{cheap}} \coloneqq \left\{(i, s) \in IS \;\middle|\;\gamma_i(s) \in \overline{(\underline{\Crit}_r)} \right\}.
    \end{align*}
    Note that for all $i$, $\{s \in \mathbb{I} \mid (i, s) \in IS_{\mathrm{cheap}}\}$ is closed
    and that for all $(i, s) \in IS \backslash IS_{c}$, we have by \cref{eq:P_exp} that
    \begin{equation}\label{eq:cost_is_high_in_non_cheap}
        \Cost_{t_i}(\gamma_i(s)) \geq \Delta_{\Cost}.
    \end{equation}
    For the ease of notation, we use the lexicographic order on the pairs $(i, s) \in IS$,
    where
    $(i, s) < (j, s')$ if $i < j$ or if $i = j$ and $s < s'$.
    For $(i, s) \in IS_{\mathrm{cheap}}$, we also set $J(i, s)$ to be the unique index $j\in \{1, \ldots, m \}$ such that $\gamma_i(s) \in B_r(\Crit_j)$.\\
    As a next step, we inductively define breakpoints, which we will eventually use to get an upper bound
    on $c_t(x_1, x_2)$ by repeatedly using the triangle inequality.
    We start by setting
    \begin{equation} \label{eq:breakpoint_initial}
        (i_{\mathrm{out}}^0, s_{\mathrm{out}}^0) \coloneqq \begin{cases}
            \max\{(i, s) \in IS_{\mathrm{cheap}} \mid J(i, s) = J(1, a_1)\} & \text{if } (1, a_1) \in IS_{\mathrm{cheap}}, \\
            (1, a_1)                                               & \text{otherwise},
        \end{cases}
    \end{equation}
    and inductively set, for each $k$
    \begin{align*}
        (i_{\mathrm{in}}^{k+1}, s_{\mathrm{in}}^{k+1})   & \coloneqq \min\{(i, s) \in IS_{\mathrm{cheap}} \mid (i, s) > (i_{\mathrm{out}}^{k}, s_{\mathrm{out}}^{k})\}, \\
        (i_{\mathrm{out}}^{k+1}, s_{\mathrm{out}}^{k+1}) & \coloneqq \max\{(i, s) \in IS_{\mathrm{cheap}} \mid J(i, s) = J(i_{\mathrm{in}}^k, s_{\mathrm{in}}^k)\}.
    \end{align*}
    We terminate this scheme at step $d$ if there is no possible choice for $(i_{\mathrm{in}}^{d+1}, s_{\mathrm{in}}^{d+1})$.
    Formally, we set $(i_{\mathrm{in}}^{d+1}, s_{\mathrm{in}}^{d+1}) = (n, b_n)$ and $(i_{\mathrm{in}}^0, s_{\mathrm{in}}^0) = (1, a_1)$.
    The resulting breakpoints have the following properties:
    \begin{enumerate}
        \item $d \leq m$ (where $m$ is the number of components $\Crit_1,\ldots,\Crit_m$).
        \item $(i_{\mathrm{out}}^k, s_{\mathrm{out}}^k) \in IS_{\mathrm{cheap}}$ and $(i_{\mathrm{in}}^{k}, s_{\mathrm{in}}^{k}) \in IS_{\mathrm{cheap}}$
              for all $1 \leq k \leq d$.
        \item $J(i_{\mathrm{in}}^k, s_{\mathrm{in}}^k) = J(i_{\mathrm{out}}^k, s_{\mathrm{out}}^k)$ for all $1 \leq k \leq d$.
        \item For all $(i_{\mathrm{out}}^k, s_{\mathrm{out}}^k) < (i, s) < (i_{\mathrm{in}}^{k+1}, s_{\mathrm{in}}^{k+1})$, we have that
              $(i, s) \notin IS_{\mathrm{cheap}}$.\label{item:between_breakpoints_not_cheap}
    \end{enumerate}
    We set, for $0 \leq k \leq d+1$, $p_{\mathrm{in}}^k \coloneqq \gamma_{i_{\mathrm{in}}^k}(s_{\mathrm{in}}^k)$ and $p_{\mathrm{out}}^k \coloneqq \gamma_{i_{\mathrm{out}}^k}(s_{\mathrm{out}}^k)$.
    In the rest of this proof, we first estimate $ \Cc_t  (p_{\mathrm{out}}^k, p_{\mathrm{in}}^{k+1})$, then $ \Cc_t  (p_{\mathrm{in}}^k, p_{\mathrm{out}}^k)$,
    and then we combine those estimates using the triangle inequality (provided by \cref{lemma:triangular_inequ_on_P}) to get an upper bound on $c_t(x_1, x_n)$.\\
    \textbf{Estimating $ \Cc_t  (p_{\mathrm{out}}^k, p_{\mathrm{in}}^{k+1})$:}
    We set, for $1 \leq i \leq n-1$ and $0 \leq k \leq d$,
    \begin{align*}
        S_{\mathrm{exp}}^{i,k}            & \coloneqq \{s \in (a_i, b_i]_\Ii \mid (i_{\mathrm{out}}^k, s_{\mathrm{out}}^k) < (i, s) < (i_{\mathrm{in}}^{k+1}, s_{\mathrm{in}}^{k+1})\}        \\
        \bar{S}_{\mathrm{exp}}^{i,k} & \coloneqq \{s \in (a_i, b_i]_\Ii \mid (i_{\mathrm{out}}^k, s_{\mathrm{out}}^k) \leq (i, s) \leq (i_{\mathrm{in}}^{k+1}, s_{\mathrm{in}}^{k+1})\},
    \end{align*}
    where some of the $S_{\mathrm{exp}}^{i,k}$ and $\bar{S}_{\mathrm{exp}}^{i,k}$ might be empty.
    By item~\ref{item:between_breakpoints_not_cheap} above and \cref{eq:cost_is_high_in_non_cheap},
    we see that for all $1 \leq i \leq n-1$,  $0 \leq k \leq d$ and $s \in S_{\mathrm{exp}}^{i,k}$, we have that
    $\Cost_{t_i}(\gamma_i(s)) \geq \Delta_{\Cost}$. Thus, recalling \cref{eq:curve_triangle}, we have
    \begin{align*}
        C + \frac{\varepsilon}{3}
         & \geq \sum_{i=0}^{n-1} \int_{(a_i, b_i]_\Ii} \Cost_{t_i}(\gamma_i(s)) \diff \nu(s)
        \geq \sum_{k=0}^{d} \sum_{i=i_{\mathrm{out}}^k}^{ i_{\mathrm{in}}^{k+1}} \int_{S_{\mathrm{exp}}^{i, k}} \Cost_{t_i}(\gamma_i(s)) \diff \nu(s) \\
         & \geq \sum_{k=0}^{d} \sum_{i=i_{\mathrm{out}}^k}^{ i_{\mathrm{in}}^{k+1}} \int_{S_{\mathrm{exp}}^{i, k}} \Delta_{\Cost} \diff \nu(s)
        = \Delta_{\Cost} \cdot \sum_{k=0}^{d} \sum_{i=i_{\mathrm{out}}^k}^{ i_{\mathrm{in}}^{k+1}} \nu(S_{\mathrm{exp}}^{i, k}).
    \end{align*}
    From the chain of inequalities above and the fact that for a fixed $0 \leq k \leq d$, we have that $\sum_{i=i_{\mathrm{out}}^k}^{i_{\mathrm{in}}^{k+1}} \nu(\overline{S}_{\mathrm{exp}}^{i,k}) - \nu(S_{\mathrm{exp}}^{i,k}) \leq 2$,we get
    \begin{equation} \label{eq:cost_estimate}
        \begin{split}
            \varepsilon_{\Cost} \cdot \left(\sum_{k=0}^{d} \sum_{i=i_{\mathrm{out}}^k}^{i_{\mathrm{in}}^{k+1}} \nu(\overline{S}_{\mathrm{exp}}^{i, k})\right)
             & \leq \varepsilon_{\Cost} \cdot \left(\sum_{k=0}^{d} \left( \sum_{i=i_{\mathrm{out}}^k}^{i_{\mathrm{in}}^{k+1}} \nu(S_{\mathrm{exp}}^{i, k}) + 2\right)\right) \\
             & \leq \varepsilon_{\Cost} \cdot \left(\frac{C + \frac{\varepsilon}{3}}{\Delta_{\Cost}} + 2(m+1)\right) = \frac{\varepsilon}{3}.
        \end{split}\end{equation}
    With this estimate, we have that
    \begin{align*}
             & \sum_{k=0}^{d} \Cc_t (p_{\mathrm{out}}^k, p_{\mathrm{in}}^{k+1})
        \leq \sum_{k=0}^{d}\left( \sum_{i=i_{\mathrm{out}}^k}^{i_{\mathrm{in}}^{k+1}} \int_{\overline{S}_{\mathrm{exp}}^{i, k}} \Cost_{t}(\gamma_i(s)) \diff \nu(s)\right)                                                                                                                                                                                 \\
        \leq & \sum_{i=1}^{n-1} \underbrace{\int_{(a_i, b_i]_\Ii} \Cost_{t_i}(\gamma_i(s)) \diff \nu(s)}_{\leq \frac{\varepsilon}{3n} +  c_{t_{i}}  (x_i, x_{i+1}) \text{ by~\eqref{eq:choice_of_curve}}}
        + \underbrace{\sum_{k=0}^{d}\left(\sum_{i=i_{\mathrm{out}}^k}^{i_{\mathrm{in}}^{k+1}} \int_{\overline{S}_{\mathrm{exp}}^{i, k}} \underbrace{|\Cost_{t_i}(\gamma_i(s)) - \Cost_{t}(\gamma_i(s))|}_{\leq \varepsilon_{\Cost} \text{ by~\eqref{eq:delta_prime}}}\diff\nu(s)\right)}_{\leq \frac{\varepsilon}{3} \text{ by~\eqref{eq:cost_estimate}}} \\
         \leq   & \frac{\varepsilon}{3} + \sum_{i=1}^{n-1} c_{t_i}(x_i, x_{i+1}) + \frac{\varepsilon}{3} =\sum_{i=1}^{n-1}  c_{t_{i}} (x_i, x_{i+1}) + \frac{2\varepsilon}{3}.
    \end{align*}
    \textbf{Estimating $ \Cc_t (p_{\mathrm{in}}^k, p_{\mathrm{out}}^{k})$:}
    We fix $0 \leq k \leq d$. If $k=0$,
    we only consider the case where  $p_{\mathrm{in}}^0 \neq p_{\mathrm{out}}^0$, and we are thus in the
    first case of \cref{eq:breakpoint_initial}, i.e., $i_{\mathrm{in}}^k, s_{\mathrm{in}}^k \in IS_{\mathrm{cheap}}$.
    If $p_{\mathrm{in}}^0 = p_{\mathrm{out}}^0$, we may simply ignore this term in the triangle inequality at the end.
    We set $j \coloneqq J(i_{\mathrm{in}}^k, s_{\mathrm{in}}^k) = J(i_{\mathrm{out}}^k, s_{\mathrm{out}}^k)$.
    We can choose $\overline{p}_{\mathrm{in}}^k, \overline{p}_{\mathrm{out}}^k \in \Crit_j$ such that
    \begin{equation*}
        d_{\Pp}  (\overline{p}_{\mathrm{in}}^k, p_{\mathrm{in}}^k) \leq r \leq L \quad \text{and} \quad  d_{\Pp} (\overline{p}_{\mathrm{out}}^k, p_{\mathrm{out}}^k) \leq r \leq L,
    \end{equation*}
    so that we can use \cref{eq:cheap_fuel_cheap_action} to see that
    \begin{equation*}
        \Cc_t (\overline{p}_{\mathrm{in}}^k, p_{\mathrm{in}}^k) \leq  \frac{\varepsilon}{6(m+1)} \quad \text{and} \quad  \Cc_t (\overline{p}_{\mathrm{out}}^k, p_{\mathrm{out}}^k) \leq  \frac{\varepsilon}{6(m+1)}.
    \end{equation*}
    Using the triangle inequality and \cref{prop:no_cost_in_within_conn_component}, we have that
    \begin{equation*}
        \Cc_t (p_{\mathrm{in}}^k, p_{\mathrm{out}}^k) \leq \Cc_t (p_{\mathrm{in}}^k, \overline{p}_{\mathrm{in}}^k) +  \Cc_t  (\overline{p}_{\mathrm{in}}^k, \overline{p}_{\mathrm{out}}^k) +  \Cc_t  (\overline{p}_{\mathrm{out}}^k, p_{\mathrm{out}}^k) \leq \frac{\varepsilon}{3(m+1)}.
    \end{equation*}
    \textbf{Conclusion:}
    Since $(i_{\mathrm{in}}^{d+1}, s_{\mathrm{in}}^{d+1}) = (n, b_n)$ and $(i_{\mathrm{in}}^0, s_{\mathrm{in}}^0) = (1, a_1)$, we have that
    \begin{align*}
        c_t(x_1, x_n) & \leq \sum_{k=0}^{d}  \Cc_t (p_{\mathrm{in}}^k, p_{\mathrm{out}}^{k}) + \sum_{k=0}^{d}  \Cc_t (p_{\mathrm{out}}^k, p_{\mathrm{in}}^{k+1})                                                            \\
                      & \leq (m+1) \frac{\varepsilon}{3(m+1)} + \sum_{i=1}^{n-1}  c_{t_{i}}  (x_i, x_{i+1}) + \frac{2\varepsilon}{3} = \sum_{i=1}^{n-1}  c_{t_{i}}  (x_i, x_{i+1}) + \varepsilon.
    \end{align*}
    The fact that for $x, x' \in \Crit \coloneqq \{y \in \Xx \mid |\partial_t E_t|(y) = 0\}$,  $c_t(x, x') = 0$ if and only if $x$ and $x'$ belong to the same path-connected components of $\Crit$
    follows from \cref{prop:no_cost_in_within_conn_component} together with \cref{prop:cost_nonzero}.
\end{proof}
To end this section, we proof the ancillary lemma \cref{lemma:triangular_inequ_on_P}.
\begin{proof}[Proof of \cref{lemma:triangular_inequ_on_P}]
    Let $c_t$ be an action generated by curves as per \cref{def:generated_by_curves}.
    We fix $p_1, p_2, p_3 \in \Pp$ and $\varepsilon > 0$.
    We choose $\gamma^{(i)} \in \Gamma(p_i, p_{i+1})$ for $i \in \{1, 2\}$ such that
    \begin{equation}\label{eq:gamma_i_choice}
        \Cc_t (p_i, p_{i+1}) + \frac{\varepsilon}{2} \geq \int_{(a_i, b_i]_\Ii} \Cost_t\big(\gamma^{(i)}(s)\big) \diff \nu(s),
    \end{equation}
    where $a_i < b_i \in  \mathbb{I}$ are  such that $\dom(\gamma^{(i)}) = [a_i, b_i]_\Ii$.  
    In particular, $\gamma^{(i)}(a_i) = x_i$, $\gamma^{(i)}(b_i) = x_{i+1}$ and $\gamma^{(1)}(b_1) = \gamma^{(2)}(a_2)$.
    We set $\underline{\gamma}^{(2)}$ to be the curve $\gamma^{(2)}$ reparametrized to $[b_1, b_2 + b_1 - a_2]_\Ii$, i.e,
    $\underline{\gamma}^{(2)}(s) = \gamma^{(2)}(s + b_1 - a_2)$. By \cref{prop:curve_restriction},
    $\underline{\gamma}^{(2)} \in \Gamma(p_2, p_3)$ and likewise,
    the  concatenation $\gamma^{(1)} \cup \underline{\gamma}^{(2)}$ as defined in \cref{prop:curve_restriction} is in $\Gamma(p_1, p_3)$.
    We have that
    \begin{align*}
        \Cc_t (p_1, p_3) & \leq \int_{(a_1, b_2 + b_1 - a_2]_\Ii} \Cost_t\big(\gamma^{(1)} \cup \underline{\gamma}^{(2)}(s)\big) \diff \nu(s)                          \\
                      & \leq \int_{(a_1, b_1]_\Ii} \Cost_t\big(\gamma^{(1)}(s)\big) \diff \nu(s) + \int_{(a_2, b_2]_\Ii} \Cost_t\big(\gamma^{(2)}(s)\big) \diff \nu(s) \\
                      & \leq  \Cc_t  (p_1, p_2) +  \Cc_t (p_2, p_3) + \varepsilon.
    \end{align*}
    Since $\varepsilon > 0$ was arbitrary, we have that $\Cc_t$ fulfills the triangle inequality. The triangle inequality for $c_{t}$ follows by definition of~$c_{t}$ in~\eqref{eq:action_infimum}.
\end{proof}
We are now ready to state the main result of this paper, recalling the definition of $\X$ and $q \colon [0, T] \times \Xx \mapsto \X$ from \cref{eq:def_quot_space}.
{
    \renewcommand{\thetheorem}{1~(complete)}
    \begin{theorem}\label{thm:main_result_complete}
        Let us assume \cref{ass:metric_space,ass:reg_E,ass:der_time,ass:unif_coerc,ass:conn_comp}.
        Furthermore, let the family of mappings $\bar{\omega}_t\colon \Xx \to \Xx$ indexed by $t \in [0, T]$ be the transition rule as in \cref{def:discr_evol}, corresponding to an action $c_t$ and
        complying with \cref{ax:decreasing,ax:gradient,ax:limits_traject}.
        Then, for any positive vanishing sequence $(\delta_n)_{n \in \NN}$, and for the corresponding discrete quasistatic evolutions
        $\eta^{\delta_n}$ constructed according to \cref{def:discr_evol} we can---without relabeling---extract a subsequence
        such that:
        \begin{enumerate}
            \item\label{item:main_thm_weak_convergences} There exists a positive Radon measure $\bar\mu \in \mathcal{M}([0, T])$ and  $\mathcal{D}\in L^\infty([0,T],\R)$ such that
            \begin{itemize}
                \item $\mu^{\delta^n} \weak^* \bar\mu$, where $\mu^{\delta_n}$ is as defined in \cref{eq:def_mu_delta}, and
                \item $\mathcal{D}^{\delta_n} \weak^* \mathcal{D} \quad \mbox{in } L^\infty([0,T],\R)$, where $\mathcal{D}^{\delta_n} (t) \coloneqq \partial_t E_t\big( \eta^{\delta_n}(t) \big)$.
            \end{itemize}
            \item\label{item:main_thm_pw_convergence} The compositions $q \circ (\mathrm{id} \times \eta^{\delta_n})$ converge pointwise to a piecewise continuous limiting curve $\hat\eta\colon [0, T] \to \X$.
            \item\label{item:main_thm_sided_limits} The left and right limits $\hat\eta^-(t)$ and $\hat\eta^+(t)$ of $\hat\eta$ exist for every $t \in (0, T)$, and so do the limits $\hat\eta^+(0)$ and $\hat\eta^-(T)$.
            \item\label{item:main_thm_energy_balance} The limiting curve $\hat\eta$ fulfills, for all $0 \leq s < t \leq T$, the energy balance identiy
            \begin{equation}\label{eq:energy_balance_complete_theorem}
                \hat E\big(\hat \eta^+(t)\big) -  \hat E\big(\hat \eta^-(s)\big) =
                \int_s^t \mathcal{D}(\tau) \diff \tau - \bar\mu([s,t]),
            \end{equation}
            \item \label{item:main_thm_critical_points} For all $t \in [0, T]$, $\hat{\eta} (t) \subset \{ (t, x) \in [0, T] \times \Xx \mid |\partial E_{t}| (x) = 0\}$.
            \item \label{item:main_thm_piecewise_cont} The limiting curve $\hat \eta$ is continuous in $[0, T] \setminus J$, where we define $J \coloneqq \{ t \in [0, T] \mid \bar{\mu} (\{ t\} ) >0\}$.
        \end{enumerate}
        If we assume \cref{ass:strengthening_6} in addition to \cref{ass:der_time,ass:metric_space,ass:reg_E,ass:unif_coerc,ass:conn_comp}---here, \cref{ax:decreasing,ax:gradient,ax:limits_traject} are sufficient---, we have that
        \begin{enumerate}[start=7]
            \item \label{item:main_thm_energy_balance_with_lift} For almost all $t \in [0, T]$, for all $(t, x) \in \hat\eta(t)$, we have that
            $\partial_t E_t(x) = \mathcal{D}(t)$. Picking any lifts $h_t\colon \X \rightarrow [0, T], h_\Xx\colon \X \rightarrow \Xx$ such that $q \circ (h_t \times h_\Xx) \equiv \mathrm{id}_\X$, we can thus rewrite the energy balance~\eqref{eq:energy_balance_complete_theorem} as
            \begin{equation}
                \hat E(\hat\eta^+(t)) -  \hat E(\hat\eta^-(s)) = \int_s^t \partial_t E_\tau\Big(h_\Xx\big(\hat\eta(\tau)\big)\Big) \diff \tau - \bar \mu([s, t]).
            \end{equation}
        \end{enumerate}
        If the transition rule fulfills the stronger \cref{ax:action_triangular_time_stable} in place of \cref{ax:limits_traject},---which is the case whenever we assume in addition that the transition rule is compatible with an action generated by curves---we have, furthermore,
         \begin{enumerate}[start=8]
        \item\label{item:main_thm_jump_point_characterization} For every $t \in J$, 
        \begin{equation*}
            \hat{E} (\hat\eta^{-} (t)) - \hat{E}(\hat \eta^{+} (t)) = \bar{\mu} (\{t\}) = \hat{c}_{t} (\hat{\eta} ^{-} (t) , \hat{\eta}^{+} (t)),
        \end{equation*}
        where $\hat{c}_{t} ([(t, x_{1}] , [(t, x_{2})] ) \coloneqq c_{t} (x_{1}, x_{2})$.       
         \end{enumerate}
        If we assume \cref{ass:PL+time_der_slope,ass:path+Lipsch} in addition to \cref{ass:der_time,ass:metric_space,ass:reg_E,ass:unif_coerc,ass:conn_comp,ass:t_der_quotient}---here, \cref{ax:decreasing,ax:gradient,ax:limits_traject} are sufficient and \cref{ass:strengthening_6} is not needed---, we even have that $\partial_t E_{\cdot}\colon [0, T] \times \Xx \to \R$ factors through $\X$; we call the resulting functional $\partial_t \hat E\colon \X \to \R$. In this case, we even have
        \begin{enumerate}[start=9]
            \item \label{item:main_thm_energy_balance_with_factorization} The limiting curve $\hat\eta$ fulfills, for all $0 \leq s < t \leq T$, the energy balance identiy
            \begin{equation} \label{eq:energy_balance_with_factorization}
                  \hat E  (\hat\eta^+(t)) -   \hat E (\hat\eta^-(s)) = \int_s^t \partial_t \hat E(\hat\eta(\tau)) \diff \tau - \bar \mu([s, t]).
            \end{equation}
            Furthermore, we have that $\supp\, \bar \mu = J$, i.e., $\bar \mu$ is purely atomic.
        \end{enumerate}
    \end{theorem}
}
\begin{proof}
    \cref{item:main_thm_weak_convergences} follows from \cref{prop:limit_energies} under \cref{ass:metric_space,ass:reg_E,ass:der_time,ass:unif_coerc,ass:conn_comp}. 
    \cref{item:main_thm_pw_convergence,item:main_thm_sided_limits,item:main_thm_energy_balance,item:main_thm_piecewise_cont,item:main_thm_critical_points} follows from \cref{thm:traj_convergence} under \cref{ass:metric_space,ass:reg_E,ass:der_time,ass:unif_coerc,ass:conn_comp} whenever the action fulfills \cref{ax:decreasing,ax:gradient,ax:limits_traject}.
    \cref{item:main_thm_jump_point_characterization} and that \cref{ax:action_triangular_time_stable} implies \cref{ax:decreasing} follows from \cref{lemma:axiom3'-axiom1} and \cref{cor:action_implies_traj} under \cref{ass:metric_space,ass:reg_E,ass:der_time,ass:unif_coerc,ass:conn_comp}.
    The fact that \cref{ax:action_triangular_time_stable} is guaranteed whenever we assume in addition that the transition rule is compatible with an action which is generated by curves is the content of \cref{lemma:good_curves_fulfill_continuous_triangular_inequality}.\\
    \cref{item:main_thm_energy_balance_with_lift} follows from \cref{cor:action_implies_traj} under  \cref{ass:metric_space,ass:reg_E,ass:der_time,ass:unif_coerc,ass:conn_comp} whenever the action fulfills \cref{ax:decreasing,ax:gradient,ax:limits_traject}.\\
    The fact that assuming \cref{ass:t_der_quotient,ass:path+Lipsch,ass:PL+time_der_slope} in addition to \cref{ass:der_time,ass:metric_space,ass:reg_E,ass:unif_coerc,ass:conn_comp}, we can rewrite the energy balance as in \cref{eq:energy_balance_with_factorization} follows from \cref{lemma:bound_der_lim}. In \cref{prop:mu_atomic}, we show that under the same assumptions, we have that $\supp \, \bar \mu = J$.
\end{proof}

\section{Examples of transition rules} \label{sec:examples_of_evolution_rules_elaborate}
The goal of this section is to show that the transition rules of the examples given in \cref{sec:examples_of_evolution_rules} are indeed actions which are generated by curves
as per \cref{def:generated_by_curves}, i.e., to prove \cref{lemma:gradient_flow_summary_lemma,lemma:MMS_summary_lemma,lemma:BDF2_summary_lemma}.
Before we dive into the specific examples, we first show a general lemma which will help us to prove \cref{prop:no_cost_in_within_conn_component} for all three examples.
\begin{lemma}\label{lemma:zero_cost_in_conn_comp_hd_dim}
    Assume \cref{ass:unif_coerc,ass:conn_comp}. Let $c_t\colon \Xx \times \Xx \to [0, \infty)$ be an action which fulfills the triangle inequality, i.e., for which, for $x, x', x'' \in \Xx$, we have that
    \begin{equation*}
        c_t(x, x'') \leq c_t(x, x') + c_t(x', x'').
    \end{equation*}
    Let us furthermore assume that there exists $\alpha > 0$ such that for each compact $K \subseteq \Xx$ there exists a constant $C_K > 0$ for which,
    for all $x, x' \in K$ with $x, x' \in \{y \in \Xx \mid | \partial  E_t|(y) = 0 \} \eqqcolon \Crit$, we have that $c_t(x, x') \leq C_K \cdot d(x, x')^\alpha$.
    If $x, x'$ lie in the same component of $\Crit$, and there exists some connected subset $U \subseteq \Crit$ such that
    $x, x' \in U$ and $U$ has Hausdorff dimension $\dim_H(U) < \alpha$, then $c_t(x, x') = 0$. \\
    Moreover, if we have that $\frac{c_t( y, y')}{d(y, y')^\alpha} \rightarrow 0$ as $d(y, y') \rightarrow 0$, then the same condition holds if the set $U$ to have finite $\alpha$-dimensional Hausdorff measure.
\end{lemma}
\begin{proof}
    Let us pick $x, x'$ and $U$ as in the assumptions. We can assume, without loss of generality, that
    $U$ is contained in a single component of $\Crit$, and thus, by \cref{ass:conn_comp,ass:unif_coerc},
    $U \subseteq K$ for some compact $K \subseteq \Xx$. We pick $C_K$ as in the assumptions.
    Next, we note that if $U_1, \dots, U_n$ is an open cover of $U$, there
    exists---possibly after reordering---a sequence of distinct $y_1, \dots, y_{n'} \in U$
    such that $n' \leq n + 1$, $\{y_i, y_{i+1}\} \subseteq U_i$ and $y_1 = x$, $y_{n'} = x'$.
    To see this, we first may assume without loss of generality that $x \in U_1$ and
    pick $n' \in \NN$ such that $x' \in U_{n'}$.
    We consider the graph $G$ whose vertices are given as $N \coloneqq \{U_1, \dots, U_n\}$
    and whose edges are given by $\mathcal{E} \coloneqq \{(U_i, U_j) \mid U_i \cap U_j \neq \emptyset\}$.
    Since $U$ is connected, $G$ is connected and we can find a minimal path from $U_1$ to $U_{n'}$ in $G$.
    After reordering, we may assume that this path is given as $U_1, \dots, U_{n'}$.
    Setting $y_1 = x$, $y^{n'} = x'$ and picking $y^i \in U_i \cap U_{i+1}$ for $i \in \{1, \dots, n'-1\}$,
    we have found the desired sequence.
    By applying the triangle inequality repeatedly, we have that
    \begin{align*}
        c_t(x, x') & \leq \sum_{i=1}^{n'-1}c_t(y_i, y_{i+1})
        \leq \sup_{1 \leq j < n'} \frac{c_t( y_j, y_{j+1})}{d( y_j, y_{j+1})^\alpha} \sum_{i=1}^{n'-1} d(y_i, y_{i+1})^\alpha.
    \end{align*}
    If either the Hausdorff dimension of $U$ is smaller than $\alpha$ or if  $\frac{c_t(y, y')}{d(y, y')^\alpha} \rightarrow 0$ as $d(y, y') \rightarrow 0$ and the $\alpha$-dimensional Hausdorff measure of $U$ is finite, we can make the right hand side arbitrarily small, proving the claim.
\end{proof}
\subsection{The gradient flow} \label{subsec:gradient_flow}
In this section, we prove \cref{lemma:gradient_flow_summary_lemma}. We use freely  \cref{eq:def_phase_space_gf,eq:embedding_gf,eq:gamma_gradient_flow,eq:def_cost_gf,eq:action-lift-GF} in \cref{lemma:gradient_flow_summary_lemma}. 
Substituting those definitions into \cref{eq:action_infimum} and simplifying, we obtain the  function $\cGF_t \colon \Xx \times \Xx \to [0, \infty)$ defined in \cref{eq:def_gradient_flow_action}.
\begin{lemma} \label{lemma:equivalent_minimization}
    The map $\cGF_t$ coincides with  $\cGFprime_t \colon \Xx \times \Xx \to [0, \infty)$ defined as
    \begin{equation} \label{eq:equivalent_minimization}
        \cGFprime_t (x_1, x_2) \coloneqq \inf_{\phi \in AC([0, 1]; \Xx)} \int_{[0, 1]} |\partial E_t| (\phi(s))|\dot{\phi}(s)| \diff s. 
    \end{equation}
    Furthermore, the integral on the right-hand side is invariant to reparametrization of the curve:
    For all $\phi \in AC([0, 1]; \Xx)$ and all continuously differentiable and monotone increasing $f\colon [a, b] \rightarrow [0, 1]$, setting
    $\overline{\phi} \coloneqq \phi \circ f$ we have that for every $t \in [0, T]$
    \begin{equation} \label{eq:equivalent_minimization_reparametrization}
        \int_{[0, 1]}  |\partial E_t| (\phi(s))  |\dot{\phi}(s)| \diff  s  = \int_{[a, b]}  |\partial E_t|  (\overline{\phi}(s)) |\dot{\overline{\phi}}(s)| \diff  s.
    \end{equation}
    As a result, we have that
    \begin{equation} \label{eq:equivalent_minimization_2}
        \cGF_t(x_1,x_2)\coloneqq \inf\left\{
        \int_a^b  |\partial E_t| \big(\phi(s)\big) \diff s \;\middle|\;
        \substack{a < b \in \R \\ \phi \in AC([a,b],\Xx) \\ \phi(a) = x_1,\, \phi(b) = x_2,\, |\dot \phi| \equiv 1}
        \right\}.
    \end{equation}
\end{lemma}
\begin{proof}
    \cref{eq:equivalent_minimization_reparametrization} is a simple application of the chain rule.
    For \cref{eq:equivalent_minimization}, we note that
    for $x, y \in \R$, we have that
    \begin{equation}\label{eq:quadratic_formula}
        \frac{1}{2} x^2 + \frac{1}{2} y^2 = xy + \frac{1}{2} (x-y)^2 \geq xy.
    \end{equation}
    From~\eqref{eq:quadratic_formula}, it follows directly that $\cGF_t \geq \cGFprime_t$.
    For the other direction, let us assume that $c_{t}^{F} (x_{1}, x_{2})>0$. We first fix some $x_{1}, x_{2} \in \Xx \times \Xx$ and  $\kappa \in (0, \frac{1}{3})$. Next, we choose some $\varepsilon' > 0$, $ b> 0$, and a curve $\phi \in AC([0, b] , \Xx)$  such that 
    \begin{equation}
          \int_{[0, b]} |\partial E_t| (\phi(s))|\dot{\phi}(s)| \diff s \leq \cGFprime_t (x_1, x_2) + \varepsilon'.
    \end{equation}
    By \cref{eq:equivalent_minimization_reparametrization}, we can assume that  $|\dot \phi(s)| \leq 1$ for a.e.~$s \in [0, b]$. Let us fix $\varepsilon \in (0, 1) $ and set 
    \begin{equation*}
        r_\varepsilon (s) \coloneqq  \int_{0}^s \frac{|\dot \phi(\sigma)| + \varepsilon}{|\partial E_t|(\phi(\sigma)) + \varepsilon^{1+\kappa}} \diff \sigma \qquad \text{for $s \in  [0, b]$.}
    \end{equation*} 
    We define $\phi_\varepsilon \coloneqq \phi \circ r^{-1}_\varepsilon$ and note that $D \coloneqq \dom(\phi_\varepsilon) \subseteq  [0, \frac{(M + 1)b}{\varepsilon^{1+\kappa}}] $. 
    If we set $D' \coloneqq \{s \in D \mid \frac{|\dot \phi(r_\varepsilon^{-1}(s))|}{ |\partial E_t|  (\phi(r_\varepsilon^{-1}(s)))} \leq  b  \varepsilon^\kappa\}$, we see that for all $\sigma \in D'$, $ \dot{r}_\varepsilon  (r_\varepsilon^{-1}(\sigma)) \leq \frac{b+1}{\varepsilon^\kappa}$,
    and thus $|D'| \leq \varepsilon^\kappa$, where $|D'|$ denotes the Lebesgue measure of~$D'$. On the other hand, for $s \in D \setminus D'$, we have---setting $s' \coloneqq r_\varepsilon^{-1}(s)$---that
    \begin{align}
    \label{eq:some-estimate}
        \bigg||\partial E_t|\big(\phi_\varepsilon(s)\big) - |\dot\phi_\varepsilon|(s)\bigg| &= \bigg| |\partial E_t|\big(\phi(s')\big) - |\dot \phi|(s') \cdot \frac{|\partial E_t|\big(\phi(s')\big) + \varepsilon^{1+\kappa}}{|\dot \phi|(s') + \varepsilon}\bigg| \\
        &= \bigg||\partial E_t|(\phi(s')) - \left(1 - \frac{\varepsilon}{|\dot \phi|(s') + \varepsilon}\right) \Big(|\partial E_t|(\phi(s')) + \varepsilon^{1+\kappa}\Big)\bigg| \nonumber\\
        &\leq 2 \varepsilon + \varepsilon \frac{|\partial E_t|\big(\phi(s')\big)}{|\dot \phi|(s')} \leq  2\varepsilon^{1 + \kappa} + \frac{ \varepsilon^{1-\kappa}}{b}. \nonumber
    \end{align}
    Furthermore, we can find a bound $N$ such that $N \geq \big||\partial E_t|\big(\phi_\varepsilon(s)\big) - |\dot\phi_\varepsilon|(s)\big|$ for all $s \in D$, independendly of the chosen $\varepsilon \in (0, 1)$, as both $|\dot \phi|$ and $|\partial E_t|\big(\phi(s')\big)$ are bounded.
    Using \cref{eq:some-estimate} and \cref{eq:quadratic_formula}, we see that
    \begin{align}
        \cGF_t(x_1, x_2) & \leq \frac{1}{2} \int_{D}  |\partial E_t|^{2} (\phi_\varepsilon(s)) + |\dot\phi_\varepsilon(s)|^2 \diff  s  \nonumber\\
                            & = \int_{D}  |\partial E_t| (\phi_\varepsilon(s)) \,|\dot\phi_\varepsilon(s)| \diff  s  + \int_{D} \frac{1}{2} \big( |\partial E_t|(\phi_\varepsilon(s))  - |\dot\phi_\varepsilon(s)|\big)^2 \diff  s  \nonumber\\
                            & \leq \cGFprime_t(x_1, x_2) + \varepsilon' + \int_{D \setminus D'} \frac{1}{2}  \bigg( 2\varepsilon^{1 + \kappa} + \frac{ \varepsilon^{1-\kappa}}{b}\bigg)^{2}  \diff  s   + \int_{D'} \frac{1}{2} N^2 \diff  s  \nonumber\\
                            & \leq \cGFprime_t(x_1, x_2) + \varepsilon' +  \frac{(M+1)b  \varepsilon ^{2 - 2\kappa} ( 2\varepsilon^{2\kappa} + \frac{1}{b})^{2}}{2 \varepsilon^{1+\kappa}}  + \frac{N^2 \varepsilon^\kappa}{2}. \label{eq:last-grad-flow}
    \end{align}
    As $\kappa$ was chosen in $(0, \frac{1}{3})$, we may pass to the limit in \eqref{eq:last-grad-flow} as $\varepsilon \to 0$, to obtain the inequality $\cGF_t(x_1, x_2) \leq \cGFprime_t(x_1, x_2) + \varepsilon'$. Letting $\varepsilon' \to 0$ in a second step, we obtain the thesis.
\end{proof}
\begin{lemma} \label{lemma:gradient_flow_ineq_energy_diff}
    Let us assume \cref{ass:reg_E,ass:metric_space}. Let us consider $a < b \in \R$, $x_1, x_2 \in \Xx$, and $\phi \in AC([a, b]; \Xx)$ such that
    $\phi(a) = x_1$, $\phi(b) = x_2$.
    Then
    \begin{equation} \label{eq:gradient_flow_ineq_energy_diff}
        \int_a^b \CostGF_{t}\big(\big(\phi(s), \dot{\phi}(s)\big)\big)  \diff  s \geq |E_t(x_1) - E_t(x_2)|.
    \end{equation}
    In particular, for any two points $x_1, x_2 \in \Xx$ and $v_1, v_2 \in  \R^{+}$, we have that
    \begin{equation} \label{eq:gradient_flow_ineq_energy_diff_2}
        \Cc^{F}_t \big((x_1, v_1), (x_2, v_2)\big) \geq |E_t(x_1) - E_t(x_2)|.
    \end{equation}
\end{lemma}
\begin{proof}
    By \cref{eq:quadratic_formula}, we have that
    \begin{equation*}
        \int_a^b \CostGF_{t}\big(\big(\phi(s), \dot{\phi}(s)\big)\big)  \diff  s \geq \int_a^b  |\partial E_t|(\phi(s))  |\dot{\phi}(s)| \diff  s.
    \end{equation*}
    Hence, \cref{eq:gradient_flow_ineq_energy_diff} follows from $|\partial E_t| (\phi(s))$ being a strong upper gradient of $E_t$ as per \cref{rmk:slope_strong_upper_gradient} (see~\cite[Definition 1.2.1]{greenbook}).\ \cref{eq:gradient_flow_ineq_energy_diff_2} follows by taking the infimum over all curves $\phi \in \Gamma\big((x_1, v_1), (x_2, v_2)\big)$ in \cref{eq:gradient_flow_ineq_energy_diff}.
\end{proof}
Before we show that $\cGF_t$ is an action generated by curves, we will show an intermediate result.
\begin{lemma}\label{lemma:cheap_fuel_cheap_action_gradient_flow}
    Let us assume \cref{ass:path+Lipsch}, and let $x \in \Xx$ and $P, r \in \R$ such that $ |\partial E_t| (y) \leq P$ for all $y \in B_r(x)$.
    Then, for all $x' \in B_r(x)$, we have that
    \begin{equation*}
        \cGF_t(x, x') \leq P\cdot d(x, x').
    \end{equation*}
\end{lemma}
\begin{proof}
    We pick $\varepsilon>0$ arbitrary and, using \cref{ass:path+Lipsch}, an $\varepsilon$-almost geodesic with natural parametrization $\gamma$ between $x$ and $x'$ ---where $d(x, x') + \varepsilon \leq r$---, i.e., $ | \dot\gamma| \equiv 1$,
    $\gamma(0) = x$ and $\gamma(b) = x'$, where $b = d(x, x') + \varepsilon$.
    Since $ |\partial E_t| \big(\gamma(s)\big) \leq P$ for $s \in [0, b]$,
    we can follow $\cGF_t\big(x, x') \leq P \big(d(x, x') + \varepsilon\big)$ directly by using \cref{lemma:equivalent_minimization}. Since $\varepsilon$ was arbitrary, the claim follows.
\end{proof}
Let us show that $\cGF_t$ is actually an action.
\begin{lemma}\label{lemma:gradient_flow_is_action}
    Let us assume \cref{ass:metric_space,ass:reg_E,ass:path+Lipsch}. Then $\cGF_t$ is an action according to \cref{def:action}. In fact, $\cGF_t$ is even continuous.
\end{lemma}
\begin{proof}
    To prove the continuity of $\cGF_t$,
    we first note that in the spirit of \cref{lemma:triangular_inequ_on_P},
    the triangle inequality follows from the fact that $\Gamma$ is closed under
    restrictions, shifts and concatenations.
    Combining the triangle inequality with \cref{lemma:cheap_fuel_cheap_action_gradient_flow},
    we see that $|c_t(x_1, x_2) - c_t(x_1', x_2')| \leq C_1 d(x_1, x_1') + C_2 d(x_2, x_2')$ for $(x_1', x_2')$ close to $(x_1, x_2)$, where $C_1 =  |\partial E_t| (x_1) + 1$ and $C_2 = |\partial E_t| (x_2) + 1$.
    Finally, \cref{eq:action_apriori_inequality}, which requires that $E_t(x) - E_t(x') \leq c_t(x, x')$, follows as a special case of \cref{eq:gradient_flow_ineq_energy_diff_2}.
\end{proof}
Before we prove that the action $\cGF_t$ is generated by curves, we first investigate under which circumstances we can apply \cref{lemma:zero_cost_in_conn_comp_hd_dim} to $\cGF_t$.
\begin{lemma} \label{lemma:zero_cost_in_conn_comp_gradient_flow}
    Let us assume \cref{ass:metric_space,ass:reg_E,ass:path+Lipsch,ass:unif_coerc,ass:conn_comp}, and let $x \mapsto | \partial E_t| (x)$ be locally  $\alpha$-H\"older continuous with $\alpha \in (0, 1]$. Furthermore, assume that for each $x, x'$ which lie in the same component of $\Crit \coloneqq \{x \in \Xx \mid  |\partial E_t| (x)  = 0\}$,
    there exists some connected $U \subseteq \Crit$ such that $x, x' \in U$ and $\dim_H(U) < 1+\alpha$.
    Then, $\cGF_t(x, x') = 0$ whenever $x, x'$ lie in the same connected component of $\Crit$. \\
    Furthermore, if $| \partial E_t|$ is only continuous, we have that $\cGF_t(x, x') = 0$ whenever $x, x'$ lie in some connected $U \subseteq \Crit$ with
    finite $1$-dimensional Hausdorff measure.
\end{lemma}
\begin{proof}
    We start with the case $ \alpha \in (0, 1)$.
    By virtue of \cref{lemma:zero_cost_in_conn_comp_hd_dim}, we only need to show that for all compact $K \subseteq \Xx$,
    there exists a constant $C_K > 0$ such that for all $x, x' \in K \cap \Crit$ we have that $c_t(x, x') \leq C_K \cdot d(x, x')^{\alpha+1}$. \\
    Let us fix a compact $K \subset \Xx$ and let $ D>0$ be such that the $\alpha$-Hölder inequality holds for all $y, y' \in K_{1} \coloneqq \overline{B_{1}(K)}$:
    \begin{equation*}
        \big|  |\partial E_t| (y) - |\partial E_t | (y') \big| \leq D \cdot d(y, y')^\alpha.
    \end{equation*}
    Let $x, x' \in \Crit \cap K$. We set $\underline{d} \coloneqq d(x, x')$ and pick $\varepsilon \in ( 0, 1)$ and a curve $\phi$ such that
    $\dom(\phi) = [0, \underline{d} + \varepsilon]$, $\phi(0) = x$, $\phi(\underline{d} + \varepsilon) = x'$ and $|\dot\phi| \equiv 1$. 
    In particular, $\phi (t) \in K_{1}$ for every $t \in [0, \overline{d} + \varepsilon]$.
    We furthermore set the constants $\underline{\alpha} \coloneqq 1 - \alpha$ and $ s^\star \coloneqq \left(\frac{\underline{d}+ \varepsilon}{\underline{\alpha}}\right)^{\underline{\alpha}}$, the function $m\colon [0,  s^\star ] \to [0, \underline{d} + \varepsilon]$ defined as $m(s) \coloneqq (\underline{\alpha} \cdot s)^{\frac{1}{\underline{\alpha}}}$, and $\underline{\phi} \coloneqq \phi \circ m$.
    Then, for every $s \in [0, s^\star]$ we have that
    \begin{equation*}
        |\partial E_t| (\underline{\phi}(s)) \leq D \cdot d(x, \underline{\phi}(s))^\alpha \leq  D \cdot m( s )^\alpha = D \cdot (\underline{\alpha} \cdot  s )^{\frac{\alpha}{\underline{\alpha}}},
    \end{equation*}
    and that
    \begin{equation*}
        |\dot{\underline{\phi}}( s )| \equiv m'( s ) = (\underline{\alpha} \cdot s )^{\frac{1}{\underline{\alpha}} - 1} = (\underline{\alpha} \cdot s )^ \frac{\alpha}{\underline{\alpha}}.
    \end{equation*}
    We now calculate
    \begin{align*}
        \cGF(x, x') &\leq \int_0^{ s^\star  }  |\partial E_t|^{2} (\underline{\phi}(s)) + |\dot{\underline{\phi}}(s)|^2 \diff  s  
        \leq \underbrace{(D^2 + 1) \cdot \underline{\alpha}^{\frac{2\alpha}{\underline{\alpha}}}}_{\eqqcolon D} \int_0^{ s^\star }  s^{\frac{2\alpha}{\underline{\alpha}}} \diff s  \\
        &= \frac{\overline{\alpha}D'}{1 + \alpha}  s^\frac{1+\alpha}{\underline{\alpha}} \Big|_0^{ s^\star} =  \frac{\overline{\alpha}D'}{1 + \alpha}  ( s^\star )^\frac{1 + \alpha}{\underline{\alpha}} = \underbrace{ \frac{\overline{\alpha}D'}{1 + \alpha} \cdot \frac{1}{\underline{\alpha}^{1 + \alpha}}}_{\eqqcolon C_K} (\underline{d} + \varepsilon)^{1 + \alpha}.
    \end{align*}
    Since $\varepsilon > 0$ was arbitrary, we have that $c_t(x, x') \leq C_K \cdot d(x, x')^{1 + \alpha}$. \\
    For $\alpha = 1$, we can do the same calculation as above, but setting $m$ to be the function $ s \mapsto e^s$ instead. \\
    To show the second part of the lemma, we again do a similar calculation as above with $m$ being the identity function to see
    that $c_t(y, y') \to 0$ as $d(y, y') \to 0$ for $y, y' \in U$. If we apply the second part of \cref{lemma:zero_cost_in_conn_comp_hd_dim} with $\alpha=0$,
    the thesis follows.
\end{proof}
We are now ready to prove that the action $\cGF_t$ is generated by curves.
\begin{lemma}\label{lemma:generated_by_curves_gradient_flow}
    Let us assume \cref{ass:unif_coerc,ass:metric_space,ass:reg_E,ass:path+Lipsch,ass:conn_comp,ass:path+Lipsch}.
    Then, the action $\cGF_t$ is generated by curves as in \cref{def:generated_by_curves}.
\end{lemma}
\begin{proof}
    The continuity of $\CostGF_t$ on $\PpGF \times [0, T]$
    follows immediately from the continuity of $|\partial E_t|$ on $\Xx \times [0, T]$;
    likewise, the fact that $i$ is a closed immersion is immediate.
    As for the properties, \cref{prop:curve_restriction,prop:cost_nonzero} are immediate;
    \cref{prop:cheap_fuel_cheap_action} follows from \cref{lemma:cheap_fuel_cheap_action_gradient_flow}:
    For $\varepsilon > 0$ and $K \subseteq \PpGF$ compact, $K_\Xx \coloneqq p_\Xx(K)$, the projection
    of $K$ onto $\Xx$, is compact as well. 
    Due to the continuity of $x \mapsto |\partial E_t|(x)$, we can find some $r', P$ such that for all $x \in K_\Xx$, $x' \in \Xx$ for which $d(x, x') \leq r'$ and $ |\partial E_t| (x) \leq P$,
    we have that $|\partial E_t| (x') \leq \sqrt{\varepsilon}$ and we can set $L \coloneqq \min(r', \sqrt{\varepsilon})$.
    To finish, we note that $ \Cc^{F}_{t} (p_1, p_2)$ does \emph{not} depend on the second components of $p_1, p_2 \in \PpGF$.
    Let us now prove the other two properties. \\
    \textbf{\cref{prop:no_cost_in_within_conn_component}: }
    To see that $\cGF_t(x_1, x_2) = 0$ whenever $x_1$ and $x_1$ lie in the same
    connected component of the set of critical points $\Crit \coloneqq \{x \in \Xx \mid  |\partial E_t| (x) = 0\}$, we note that $x_1$ and $x_2$ are connected by a rectifiable curve by $\cref{ass:conn_comp}$. Since a rectifiable curve has finite 1-dimensional Hausdorff measure, we can
    apply \cref{lemma:zero_cost_in_conn_comp_gradient_flow} to see that $\cGF_t(x_1, x_2) = 0$.
    For the converse, let $x_1, x_2 \in \Xx$ such that they do \emph{not} lie in the same connected component of
    $K'$.
    Then either there is an $i \in \{1, 2\}$ such that  $\CostGF_t(i\GF(x_i)) > 0$, or $x_1$ and
    $x_2$ lie in different connected components of $\Crit$.
    In the former case, there exists an $r \leq d(x_1, x_2)$ such that $P \coloneqq \min_{x \in B_r(x_i)}|\partial E_t(x) | > 0$
    and, by reparametrizing any curve $\gamma$ from $x_1$ to $x_2$ such that $|\dot{\gamma}| \equiv 1$
    and using~\eqref{eq:equivalent_minimization_reparametrization}, we see that $\cGF_t(x_1, x_2) = \cGFprime_t(x_1, x_2) \geq r P > 0$.
    If $x_1, x_2$ lie in
    different components of $\Crit$, we set $\Crit_1$ to be the component containing $x_1$ and use \cref{ass:conn_comp} to choose $r$ such that
    $|\partial E_t(x)| > 0$ for all $x \in \overline{B_r(\Crit_1)} \setminus \Crit_1$. We set $D \coloneqq \overline{B_r(\Crit_1)} \setminus B_{\frac{r}{2}}(\Crit_1)$
    and $P \coloneqq \min_{x \in D} |\partial E_t(x) | > 0$. Since any curve $\gamma$ connecting $x_1$ and $x_2$ has to pass through $D$,
    we can use~~\eqref{eq:equivalent_minimization_reparametrization} again to reparametrize $\gamma$ such that $|\dot\gamma| = 1$
    and see that $\cGFprime_t(x_1, x_2) \geq \frac{r}{2}P > 0$.\\
    \textbf{\cref{prop:curves_coercive}: }
    Fix some $K \subseteq \PpGF$ compact. Then $K_\Xx \coloneqq p_\Xx(K)$, the projection
    of $K$ onto $\Xx$, is compact as well.
    Since $\cGF_t$ is continuous by \cref{lemma:gradient_flow_is_action}, we can set $C \coloneqq \max_{x_1, x_2 \in K_\Xx} \cGF_t(x_1, x_2)$.
    We furthermore set $E_{max} \coloneqq \max_{x \in K_\Xx} E_t(x)$.
    For each $\varepsilon>0$ and $p_1, p_2 \in K$, we set $x_1\coloneqq p_\Xx(p_1)$ and $x_2\coloneqq p_\Xx(p_1)$
    the projection onto $\Xx$ of $p_1$ and $p_2$, respectively.
    With this notation, let us show that the first component of every quasi-optimal curve $\gamma \in \Gamma(p_1, p_2)$, $\gamma \colon [a, b] \to \PpGF$ such that
    \begin{equation*}
        \int_a^b  \CostGF_{t}  (\gamma(s))  \diff  s  \leq  \Cc_t^{F}(p_1, p_2) + \varepsilon
    \end{equation*}
    stays within $K'' \coloneqq E_t^{-1}((-\infty, Z]) \subseteq \Xx$, where  $Z \coloneqq E_{\max} + \frac{1}{2} (C + \varepsilon + 1)$.
    To see this, assume otherwise and pick $s^\star \in (a, b)$ such that $E_t( \gamma_x  (s^\star)) > Z$.
    Using \cref{eq:gradient_flow_ineq_energy_diff}, we would then have the following contradiction:
    \begin{align*}
        C + \varepsilon & \geq \int_a^b \CostGF_{t}(\gamma(s))  \diff  s 
        = \int_a^{s^\star} \CostGF_{t}(\gamma(s))  \diff  s  + \int_{s^\star}^b \CostGF_{t}(\gamma(s))  \diff  s  \\
        & > (Z - E_{max}) + (Z - E_{max}) = 2Z - 2E_{max} = C + \varepsilon + 1.
    \end{align*}
    By \cref{ass:unif_coerc}, $K''$ is compact.
    What is left to show is that we can choose quasi-optimal curves with bounded velocity.
    To this end, note that in the proof of \cref{lemma:equivalent_minimization},
    we chose near-optimal curves $\gamma$ whose velocity fulfilled
    \begin{equation*}
        \big||\dot{\gamma}(s)| -  |\partial E_t|  (\gamma(s)) \big| \leq \varepsilon'
    \end{equation*}
    for arbitrarily $\varepsilon'>0$. Thus we can set
    $v_{max} \coloneqq \max_{x \in K'}  |\partial E_t|  (x) + \varepsilon'$, for any $\varepsilon' > 0$,
    and $K' \coloneqq K'' \times [0, v_{max}] \subseteq \PpGF$.
\end{proof}
To finish this section, we prove \cref{lemma:gradient_flow_summary_lemma}.
\begin{proof}[Proof of \cref{lemma:gradient_flow_summary_lemma}]
    The fact that $\cGF_t$ is an action as in \cref{def:action} is the content of \cref{lemma:gradient_flow_is_action};
    the fact that $\cGF_t$ is generated by curves as in \cref{def:generated_by_curves} is the content of the above \cref{lemma:generated_by_curves_gradient_flow}.
    What is left to prove is that $\ruleGF_t$ as defined in \cref{def:gradient_flow_evolution_rule} is compatible with $\cGF_t$ as in \cref{def:action},
    i.e., that for all $x_1, x_2 \in \Xx$ and $\phi \in AC([0, \infty), \Xx)$
    such that $\phi(0) = x_1$, $\lim_{s \rightarrow \infty} \phi(s) = x_2$ and, for all $s \in [0, \infty)$,
    \begin{equation*}
        E_t(x) - E_t(\phi(s)) =  \frac12  \int_0^s  |\partial E_t|^{2}  (\phi(s)) + |\dot{\phi}(s)|^2 \diff s,
    \end{equation*}
    we have that
    \begin{equation*}
        E_t(x_1) - E_t(x_2) = \cGF_t(x_1, x_2).
    \end{equation*}
    This follows directly from the continuity of $E_t$ and $\cGF_t$.
\end{proof}

\subsection{The minimizing movement scheme} \label{subsec:elaborate_mms}
In this section, we will explore the minimizing movement scheme and
prove \cref{lemma:MMS_summary_lemma}.  We use freely the definitions \cref{eq:def_phase_space_mms,def:embedding_MMS,eq:def_GammaM,eq:CostM_def,eq:extended-cost-MMS}, and in the rest of this chapter adopt the convention where we denote the components of a curve $\gamma\colon I \subseteq \Z \to \PpM = \Xx \times \Xx$ by $\gamma_0$ and $\gamma_1$, and we denote the components of points $\vx, \vx' \in \PpM$, by $(x_0, x_1) \coloneqq \vx$ and $(x'_0, x'_1) \coloneqq \vx'$.
Substituting \cref{eq:def_phase_space_mms,def:embedding_MMS,eq:def_GammaM,eq:CostM_def} into \cref{eq:action_infimum} and simplifying, we obtain the expression of $\cM_t \colon \Xx \times \Xx \to [0, \infty)$ written in \cref{eq:action_infimum_MMS}.
As we describe in \cref{rmk:lipschitz_tau_relation}, whenever we assume that \cref{ass:path+Lipsch} holds, we also assume that $\tau \leq \frac{1}{L}$, where $L$ is the Lipschitz constant mentioned in \cref{ass:path+Lipsch}.
Finally, we report that in \cite[Definition~5.14]{Scilla_2018} the authors introduced an action for the minimizing movement scheme in Euclidean setting which is fully compatible with $\cM_t$.

\begin{remark}\label{rmk:cost_lambda_continuous}
    If we assume \cref{ass:metric_space,ass:reg_E,ass:der_time,ass:unif_coerc}, then
    the function $(t, x, y) \mapsto  E_t(y) + \frac{1}{2 \tau}d(x,y)^2$ is continuous in $(t, x, y)$.
    Furthermore, only points in the relatively compact sublevel set $\{y \in \Xx \mid E_t(y) \leq E_t(x)\}$ contribute to the infimum in~\eqref{eq:def_E_MMS}.
    Thus, $E\ttau\MMS$ is itself continuous in $(t, x)$, as it is locally an infimum of equicontinuous functions. Hence, also $\CostM_t$ is continuous in $(t, \vx) \in [0, T] \times \PpM$.
    Lastly, we note that $E\ttau\MMS(x) \leq E_t(x)$ for all $x \in \Xx$:
    To see this, we use $y=x$ as a competitor in \cref{eq:def_E_MMS}.
    We clarify the relation between the conditions $E\ttau\MMS(x) = E_t(x)$ and $|\partial E_t| (x) = 0$ for
    $x \in \Xx$ in \cref{lemma:zero_gradient_zero_cost}.
\end{remark}
Let us first see how this definition relates to the minimizing movement scheme.
\begin{lemma} \label{lemma:minimizing_scheme_energy_diff}
    Let us assume \cref{ass:metric_space,ass:reg_E,ass:der_time,ass:unif_coerc}. Then, for all $\tau > 0$, $\vx, \vx' \in \PpM$ and $\gamma \in \Gamma(\vx, \vx')$, we have that
    \begin{equation} \label{eq:energy_bound_minimizing_scheme_curve_phase_space}
        \sum_{s=a+1}^{b} \CostM_{t}\big(\gamma(s)\big) \geq E_t(x_1) - E_t(x'_1),
    \end{equation}
    If $\vx =  i\MMS (x)$ and $\vx' =  i\MMS (x')$ for some $x, x' \in \Xx$, then
    \begin{equation}\label{eq:energy_bound_minimizing_scheme_curve}
        \sum_{s=a+1}^{b} \CostM_{t}\big(\gamma(s)\big) \geq E_t(x) - E\ttau\MMS(x'),
    \end{equation}
    where equality is attained if and only if
    \begin{equation}\label{def:minimizing_movement}
        \gamma_0(s+1) = \gamma_1(s) \in \argmin_{y \in \Xx} \left\{ E_t(y) + \frac{1}{2\tau} d(\gamma_0(s), y)^2 \right\}
    \end{equation}
    for all $s \in [a+1, b-1]_\Z$.
    In particular, for every $x, x' \in \Xx$ we have
    \begin{equation} \label{eq:energy_bound_minimizing_scheme}
        \cM_t(x, x') \geq E_t(x) - E\ttau\MMS(x') \geq E_t(x) - E_t(x').
    \end{equation}
\end{lemma}
\begin{proof}
    To prove \cref{eq:energy_bound_minimizing_scheme_curve_phase_space}, it suffices to show that
    $\CostM_{t}\big(\gamma(s)\big) \geq E_t\big(\gamma_0(s)\big) - E_t\big(\gamma_1(s)\big)$ for all $s \in [a+1, b-1]_\Z$.
    To this end, we write
    \begin{equation}\label{eq:energy_bound_minimizing_scheme_curve_proof}
        \begin{split}
                   & \CostM_t\big(\gamma(s)\big) - \bigg( E_t\big(\gamma_0(s)) - E_t\big(\gamma_1(s)\big)\bigg)                         \\
            = \,   & E_t\big(\gamma_0(s)\big) - \inf_{y \in \Xx} \left\{ E_t(y) + \frac{1}{2\tau} d\big(\gamma_0(s), y\big)^2 \right\}  \\
                   & + \frac{1}{2\tau} d\big(\gamma_1(s), \gamma_0(s)\big)^2 - E_t\big(\gamma_0(s)\big) + E_t\big(\gamma_1(s)\big)      \\
            \geq\, & E_t\big(\gamma_0(s)\big) - E_t\big(\gamma_1(s)\big) - \frac{1}{2\tau} d\big(\gamma_1(s), \gamma_0(s)\big)^2        \\
                   & + \frac{1}{2\tau} d\big(\gamma_1(s), \gamma_0(s)\big)^2 - E_t\big(\gamma_0(s)\big) + E_t\big(\gamma_1(s)\big) = 0,
        \end{split}
    \end{equation}
    where the inequality follows by using $y = \gamma_1(s)$ as a competitor for the infimum.
    \cref{eq:energy_bound_minimizing_scheme_curve} follows by applying \cref{eq:energy_bound_minimizing_scheme_curve_phase_space}
    to $\gamma|_{[a, b-1]}$ and noting that $\CostM_{t}\big(\gamma(b)\big) = \CostM_t\big((x', x')\big) = E_t(x') - E_{t, \tau}\MMS(x')$.
    In the lower bound \cref{eq:energy_bound_minimizing_scheme_curve}, equality is attained if and only if $\gamma_0(s)$ fulfills~\eqref{def:minimizing_movement}.
    Finally, \cref{eq:energy_bound_minimizing_scheme} follows by taking the infimum
    over $\Gamma( i\MMS(x), i\MMS(x'))$ in \cref{eq:energy_bound_minimizing_scheme_curve} and using \cref{rmk:cost_lambda_continuous} for the second inequality.
\end{proof}
Before we can show that $\cM_{t}$ is an action according to \cref{def:action}, we need to show two intermediate results. The first lemma is related to \cref{prop:cheap_fuel_cheap_action}, the second lemma concerns \cref{prop:curves_coercive} of \cref{def:generated_by_curves}.

\begin{lemma} \label{lemma:bounded_cost_minimizing_scheme}
    Let us assume \cref{ass:metric_space,ass:reg_E,ass:der_time,ass:unif_coerc}.
    For all $\vx, \vx' \in \PpM$ and $L, P > 0$
    such that $d_{\PpM}(\vx, \vx') \leq L$, $\CostM_t(\vx) \leq P$, $\CostM_t(\vx') \leq P$ and $\CostM_t( i\MMS (x_1)) \leq P$
    we have that $ \cMg_t (\vx, \vx') \leq \frac{1}{2\tau}(2L + \sqrt{2\tau P})^2 + 3P$.
\end{lemma}
\begin{proof}
    Using \cref{rmk:cost_lambda_continuous}, we see that
    \begin{equation}
        \frac{1}{2\tau}d(x_0, x_1)^2 \leq E_t(x_0) - E_{t, \tau}\MMS(x_0) + \frac{1}{2\tau}d(x_0, x_1)^2 = \CostM_t(\vx) \leq  P,
    \end{equation}
    Combining this with $d(x_0, x'_0) \leq 2 d_{\PpM}(\vx, \vx') \leq 2L$, which follows directly from the definition of $d_{\PpM}$ in \cref{eq:def_phase_space_mms},
    we see---using the triangle inequality---that $d(x_1, x'_0) \leq 2L + \sqrt{2\tau P}$.
    Furthermore, we have for $\vx'' \coloneqq (x_1, x'_0)$, that
    $\CostM_t(\vx'') \leq  \frac{1}{2\tau} d(x_1, x'_0)^2 + P$,  since $\CostM_{t} (i\MMS (x_{1})) \leq P$.
    The lemma now follows by considering as a competitor
    for $\cMg_t$ the curve $\gamma\colon [1, 3]_{\mathbb{Z}} \to \PpM$ defined as $\gamma(1)  \coloneqq \vx$, $\gamma(2) \coloneqq \vx''$, $\gamma(3) \coloneqq  \vx'$ and noting that, owing to the hypotheses,
    $\sum_{s=2}^3 \CostM_t(\gamma(s)) \leq \frac{1}{2\tau}(2L + \sqrt{2\tau P})^2 + 3P$.
\end{proof}
\begin{lemma} \label{lemma:coerciveness_minimizing_scheme}
    Let us assume \cref{ass:metric_space,ass:reg_E,ass:der_time,ass:unif_coerc}.
    For all $K \subseteq \PpM$ compact and $C > 0$,
    there exists $K' \subseteq \PpM$ compact such that the following implication holds
    for all $\vx, \vx' \in K$ and $\gamma \in \Gamma(\vx, \vx')$:
    \begin{equation} \label{eq:curves_bounded_coercive_mms}
        \sum_{s=a+1}^b \CostM_{t}(\gamma(s)) \leq C \implies \im(\gamma) \subseteq K'.
    \end{equation}
    In particular, for each $K \subseteq \PpM$ there exists a compact set $K \subseteq \PpM$ such that for every $\varepsilon \in ( 0, 1)$ the following implication holds
    for all $\vx, \vx' \in K$ and $\gamma \in \Gamma(\vx, \vx')$:
    \begin{equation} \label{eq:curves_coercive_implications_mms}
        \sum_{s=a+1}^b \CostM_{t}(\gamma(s)) \leq  \cMg_t (\vx, \vx') + \varepsilon \implies \im(\gamma) \subseteq K'.
    \end{equation}
\end{lemma}
\begin{proof}
    We start with~\eqref{eq:curves_bounded_coercive_mms} and fix $K \subseteq \PpM$ and $C > 0$.
    Since $\im(\gamma_1|_{[a, b-1]}) = \im(\gamma_0|_{[a+1, b]})$ for all $\gamma \in \Gamma(\vx, \vx')$,
    it suffices to show that $\im(\gamma_1)$ is compact in $\Xx$.
    To proceed, we use the coercivity of $E_t$ (cf.~\cref{ass:unif_coerc}) and \cref{eq:energy_bound_minimizing_scheme_curve_phase_space}
    to get the following chain of inequalities, for any $s^\star \in [a+1, b-1]_{\mathbb{Z}}$:
    \begin{align*}
        C & \geq \sum_{s=a+1}^b \CostM_{t}\big(\gamma(s)\big) \geq \sum_{s=s^\star + 1}^b \CostM_{t}\big(\gamma(s)\big)
        \geq E_t\big(\gamma_1(s^\star)\big) - E_t(x'_1)                                                                 \\
          & \geq E_t\big(\gamma_1(s^\star)\big) - \max_{(y_0, y_1) \in K} E_t(y_1).
    \end{align*}
    To show \eqref{eq:curves_coercive_implications_mms}, we infer from \cref{lemma:bounded_cost_minimizing_scheme},
    from the continuity of $\CostM_t$---see \cref{rmk:cost_lambda_continuous}--- and from the compactness of $K$ that
    $\cMg_t$ is bounded on $K \times K$ by some $C'$. Applying~\eqref{eq:curves_bounded_coercive_mms} with $C = C' + 1$ finishes the proof.
\end{proof}
We are now in a position to prove that the map $\cM_t \colon \Xx \times \Xx \to [0, \infty)$ is an action.
\begin{lemma}\label{lemma:action_minimizing_scheme}
    Let us assume \cref{ass:metric_space,ass:reg_E,ass:der_time,ass:unif_coerc}. Then, $\cM_t$ is an action according to \cref{def:action}.
\end{lemma}
\begin{proof}
    \newcommand{\cMAlt}{\superscriptleft{\tau}\underline{c}\MMS}
    We start by choosing an arbitrary compact subset
    $K \in \Xx$. It suffices to show that $\cM_t$ is lower semicontinuous on $K \times K$.
    To this end, we show the continuity of $\cMAlt_t \colon \Xx \times \Xx \to [0, \infty)$ on $K$, which is defined as follows:
    \begin{equation*}
        \cMAlt_t(x, x') \coloneqq \inf\left\{ \sum_{s=a+1}^b \CostM_{t}\big(\gamma(s)\big) \,\middle|\, \gamma \in \underline{\Gamma}\big( i\MMS(x), i\MMS (x')\big) \right\},
    \end{equation*}
    where $\underline{\Gamma}(p, p') = \{\gamma \in \Gamma(p, p') \mid \dom(\gamma) = [a, b]_\Z, a < b - 1\}$.
    Beside the condition $a < b - 1$, the definition of $\cMAlt_t$ coincides with the definition of $\cM_t$ in~\eqref{eq:action_infimum_MMS}.
    The continuity of $\cMAlt_t$ on $K \times K$  implies  the lower semicontinuity of $\cM_t$ on $K \times K$. 
    Indeed, this follows from the facts that (1.) $\cMAlt_t \geq \cM_t$ and (2.) the set of points where $\cMAlt_t \neq \cM_t$ is the diagonal in $K \times K$, which is closed
    ---and on this diagonal, $\cM_t(x, x) = \CostM_t(x, x)$ is continuous. \\
    For any $a < b -1 \in \Z$, $x, x' \in K$ and $f\colon [a+1, b-1]_\Z \to \Xx$, we set
    $F(f, x, x')\colon  {[a, b]_\Z \to \PpM}$ to be the finite sequence
    \begin{equation} \label{eq:definition_extension_F}
        \bigg( i\MMS (x), \big(x, f(a+1)\big), \big(f(a+1), f(a+2)\big), \dots, \big(f(b), x'\big), i\MMS (x')\bigg).
    \end{equation}
    Note here that we allow the case $a = b-2$, in which case the above display should be read as
    $\big( i\MMS(x), (x, x'), i\MMS (x')\big)$.
    Setting $D \coloneqq \{f \colon [a+1, b-1]_\Z \to \Xx \mid a < b - 1\}$,
    we observe that $F(\cdot, x, x')$ is a bijection from $D$ to $\underline{\Gamma}\big( i\MMS (x), i\MMS (x')\big)$.
    Thus, we can write $\cMAlt_t$ as the infimum over such extensions:
    \begin{equation} \label{eq:alternative_infimum_formulation_mms_1}
        \cMAlt_t(x, x') = \inf\bigg\{ \sum_{s=a+1}^{b} \CostM_{t}\big(F(f, x, x')\big) \,\bigg|\, f \in D, \dom(f) = [a+1, b-1]_\Z\bigg\}.
    \end{equation}
    Furthermore, by virtue of \cref{lemma:coerciveness_minimizing_scheme}, we can take the infimum
    over only those $f$ for which $\im(f) \subseteq K'$ for some compact $K' \subseteq \PpM$.
    Setting $D' \coloneqq \{f\colon  [a+1, b-1]_\Z \to K' \mid a < b-1\}$ and defining the function
    $\mathfrak{C}\colon D' \times K \times K \to \R$ by
    \begin{equation*}
         \mathfrak{C} ( (f, x, x) ) \coloneqq  \sum_{s \in \dom\big(F(f, x, x)\big)} \CostM_{t}\big(F(f, x, x')(s)\big).
    \end{equation*}
    We can thus rewrite \cref{eq:alternative_infimum_formulation_mms_1} as
    \begin{equation} \label{eq:alternative_infimum_formulation_mms_2}
        \cMAlt_t(x, x') = \inf_{f \in D'}\big\{ \mathfrak{C}(f, x, x') \big\}.
    \end{equation}
    For each $f \in D'$, $\mathfrak{C}(f, \cdot, \cdot)$ is continuous on $K \times K$,
    due to the continuity of $\CostM_t$ ---see \cref{rmk:cost_lambda_continuous}.
    If we can show that this continuity is uniform over $f \in D'$,
    then $\cMAlt_t$ is continuous on $K \times K$, as it is the infimum over equicontinuous functions.
    To this end, we pick
    $x, x', \overline{x}, \overline{x}' \in K$, $f \in D'$ with $\dom(f) =  [a+1, b-1]_{\mathbb{Z}}$ and use the definition of $F$ in \cref{eq:definition_extension_F} and write
    \begin{align*}
          & \mathfrak{C}(f, x, x') - \mathfrak{C}(f, \overline{x}, \overline{x}')
        \\
        = &
        \bigg(\CostM_t\big(  i\MMS  (x)\big) - \CostM_t\big( i\MMS (\overline{x})\big)\bigg) +
        \bigg(\CostM_t\big(\big(x, f(a+1)\big)\big) - \CostM_t\big(\big(\overline{x}, f(a+1)\big)\big)\bigg)       \\
          & + \bigg(\CostM_t\big(\big(f(b), x'\big)\big) - \CostM_t\big(\big(f(b), \overline{x}'\big)\big)\bigg) +
        \bigg(\CostM_t\big( i\MMS  (x')\big) - \CostM_t\big( i\MMS (\overline{x}'\big))\bigg).
    \end{align*}
    All the above terms go to zero as $(x, x') \to (\overline{x}, \overline{x}')$,
    \emph{uniformly} over $f(a+1), f(b) \in K'$, which shows the claimed equicontinuity.
\end{proof}
Our next goal is to show that the action $\cM_t$ is generated by curves.
To this end,
we again need to show two intermediate results.
\begin{lemma} \label{lemma:action_zero_hausdorff}
    Let us assume \cref{ass:metric_space,ass:reg_E,ass:der_time,ass:unif_coerc}, and let $\vx, \vx' \in \PpM$ such that there exists a compact connected set $U \subseteq \PpM$ for which
    $dim_H(U) < 2$, $\{\vx, \vx'\} \subseteq U$ and $\CostM_t|_U \equiv 0$.
    Then $ \cMg_t (\vx, \vx') = 0$.
\end{lemma}
\begin{proof}
    Since $\CostM_t|_U \equiv 0$ implies that $U \subseteq \im(i)$, we can apply \cref{lemma:zero_cost_in_conn_comp_hd_dim}
    with $\alpha=2$ by using \cref{lemma:bounded_cost_minimizing_scheme}.
\end{proof}
\begin{lemma} \label{lemma:zero_gradient_zero_cost}
    Let us assume \cref{ass:metric_space,ass:reg_E,ass:der_time,ass:unif_coerc,ass:path+Lipsch},
    and let $x, x^\star \in \Xx$ and $R>0$ such that $\{x' \in \Xx \mid E_t(x') \leq E_t(x)\} \subseteq B_{\frac{R}{4}}(x^\star)$.
    Then, we have that
    \begin{equation*}
        |\partial E_t| (x) = 0 \iff \CostM_t( i\MMS (x)) = 0.
    \end{equation*}
\end{lemma}
\begin{proof}
    We start proving the implication $\CostM_t( i\MMS (x)) \implies |\partial E_t|  (x) = 0$.
    If $|\partial E_t| (x) > 0$, then there
    exists a sequence $(x'_n)_{n \in \NN} \to x$ such that
    $\lim_{n \rightarrow \infty} \frac{E_t(x'_n) - E_t(x)}{d(x'_n, x)} > 0$.  We  thus pick $\varepsilon > 0$ and a subsequence without relabeling such that
    $\frac{E_t(x'_n) - E_t(x)}{d(x'_n, x)} \geq \varepsilon$ for all  $n$.
    For such a sequence, we have, for $n$ large enough, that
    \begin{equation*}
        \CostM_t\big( i\MMS (x)\big) \geq E_t(x) - E_t(x'_n) - \frac{d(x, x'_n)^2}{2 \tau} \geq d(x, x'_n) \left( \varepsilon - \frac{d(x, x'_n)}{2 \tau} \right) > 0.
    \end{equation*}
    For the  opposite implication, we assume that $|\partial E_t| (x) = 0$.
    Then, for all $y \neq x$ for which $ E_{t} (y) < E_{t}  (x)$, we set $\underline{d} \coloneqq d(x, y)$ and use \cref{ass:path+Lipsch} to pick some $0 < \varepsilon \leq \frac{R}{2}$ and some $\varepsilon$-geodesic $\gamma$ between $x$ and $y$
    with natural parametrization, i.e.,
    such that $\gamma(0) = x$, $\gamma(\underline{d} + \varepsilon) = y$ and $|\dot\gamma| \equiv 1$ a.e.
    Since $x, y \in B_{\frac{R}{4}} (x^{\star})$ by hypothesis, we note that $\im(\gamma) \subseteq B_{R}(x^\star)$ and observe that by \cref{rmk:slope_strong_upper_gradient} and \cref{ass:path+Lipsch}, 
    \begin{equation*}
        \begin{aligned}
            E_t(x) - E_t(y) - \frac{d(x, y)^2}{2 \tau} &\leq \left| \int_0^{\underline{d} + \varepsilon} |\partial E_t| (\gamma(s)) \diff s\right| - \frac{\underline{d}^2}{2 \tau}
            \leq \int_0^{\underline{d} + \varepsilon} L \cdot s \,\diff s - \frac{d^2}{2 \tau}                                                                                             \\
            & \!\!\!\!\!\!\!\!
            \leq \frac{1}{2} L (\underline{d} + \varepsilon)^2 - \frac{d^2}{2 \tau}
            \leq \frac{d^2}{2} \left( L - \frac{1}{\tau} \right) + \varepsilon L \left(\underline{d} + \frac{\varepsilon}{2}\right)\leq \varepsilon L \left(\underline{d} + \frac{\varepsilon}{2}\right).
        \end{aligned}
    \end{equation*}
    Since $\varepsilon > 0$ was arbitrary, we have that $E_t(x) - E_t(y) - \frac{d(x, y)^2}{2 \tau} \leq 0$.
    Since $y$ was arbitrary as well, we have shown that $\CostM_t( i\MMS (x)) = 0$ for such $x$.
\end{proof}
We are now ready to prove that the action $\cM_t$ is generated by curves.
\begin{lemma} \label{lemma:generated_by_curves_MMS}
    Let us assume \cref{ass:metric_space,ass:reg_E,ass:unif_coerc,ass:der_time,ass:conn_comp,ass:path+Lipsch}.
    Then, for all $\tau \leq \frac{1}{L}$,
    the action $\cM_t$ is generated by curves as in \cref{def:generated_by_curves}.
\end{lemma}

\begin{proof}
    The fact that $i\MMS$ is a closed immersion is immediate, and the continuity of  $\CostM_{t}$ was shown
    in \cref{rmk:cost_lambda_continuous}.
    As for the properties,
    \cref{prop:curve_restriction} is immediate,
    \cref{prop:cost_nonzero} results from \cref{lemma:zero_gradient_zero_cost},
    \cref{prop:curves_coercive} follows directly from \cref{lemma:coerciveness_minimizing_scheme}
    and \cref{prop:cheap_fuel_cheap_action} descends from \cref{lemma:bounded_cost_minimizing_scheme}.
    We will now show the missing \cref{prop:no_cost_in_within_conn_component}. \\
    The assumptions and \cref{lemma:action_zero_hausdorff} imply that
    $\cM_t(x, x') = 0$ whenever $ i\MMS (x)$ and $ i\MMS (x')$ lie in the same connected component of $ i\MMS (\Crit) \subseteq \PpM$, where $\Crit \coloneqq \{x \in \Xx \mid  |\partial E_t| (x) = 0\}$.
    In this case, by \cref{ass:conn_comp} we have that $x$ and $x'$ are connected by a curve $\gamma$ in $\Xx$ whose image in $\PpM$ under $i\MMS$ has Hausdorff dimension
    smaller than $2$. Thus, we can apply \cref{lemma:action_zero_hausdorff}.
    For the converse, we assume that $ i\MMS (x)$ and $ i\MMS  (x')$ are \emph{not} in the same connected component of $ i\MMS (\Crit)$---because $ i\MMS $ is an isometry, this is equivalent to $x$ and $x'$ not being in the same connected component of $\Crit$.
    We choose $\varepsilon > 0$, $K' \subseteq \PpM$ compact
    such that implication~\eqref{eq:curves_coercive_implications_mms} in \cref{lemma:coerciveness_minimizing_scheme} holds for $K = \{ i\MMS (x), i\MMS  (x')\}$ and define the compact set $ \mathcal{K}  \coloneqq K'_0 \cup K'_1$, where $K'_0, K'_1 \subseteq \Xx$ are the projections of $K'$  over $\Xx$. 
    We set $\Crit_1, \dots, \Crit_n$ to be the connected components of $\Crit$ which intersect  $\mathcal{K}$. 
    We furthermore set $r \coloneqq \min_{i \neq j} \min_{y \in \Crit_i, y' \in \Crit_j} d(y, y')$,
    which is nonzero by \cref{ass:conn_comp}.
    If $x' \not\in \Crit$, we immediately see that $c_t(x, x') \geq \Cost_t\big( i\MMS (x')\big) > 0$.
    So we may assume, after reordering, that $x' \in \Crit_1$ and $x \not \in \Crit_1$.
    We set $D \coloneqq  \mathcal{K} \setminus \bigcup_{i=1}^n B_{\frac{r}{4}}(\Crit_i)$
    and $P_{\min} \coloneqq \min_{y \in D} \Cost_t( i\MMS  (y))$, which is  positive by compactness of $\mathcal{K}$ and continuity of $\CostM_t$. \\
    Now pick any $\gamma \in \Gamma\big( i\MMS (x), i\MMS (x')\big)$ such that
    \begin{equation} \label{eq:almost-optimal}
        \sum_{s= a+1}^b \CostM_{t}(\gamma(s)) \leq  \cM_{t} (x, x') + \varepsilon
    \end{equation} 
    where $a < b \in \Z$  are such that $\dom(\gamma) = [a, b]_\Z$.
    Let $s^\star$ be the maximal $s$ such that $\gamma_0(s) \in B_{\frac{r}{4}}(\Crit_1)$.
    Then $\gamma_1(s^\star) \not\in B_{\frac{r}{4}}(\Crit_1)$, and either $\gamma_1(s^\star) \in D$ or $\gamma_1(s^\star) \in B_{\frac{r}{4}}(\Crit_i)$ for some $i \neq 1$.
    In the former case, we have that 
    \begin{equation*}
       \sum_{s=a}^b \CostM_{t}\big(\gamma(s)\big) \geq \CostM_t\big(\gamma(s^\star)\big) \geq \CostM_t\big( i\MMS  (\gamma_1(s^\star))\big) \geq P_{\min}. 
    \end{equation*}
    In the latter case, we have that 
    \begin{equation*}
        \sum_{s=a}^b \CostM_{t}\big(\gamma(s)\big) \geq \frac{1}{2\tau} d\big(\gamma_0(s^\star) , \gamma_1(s^\star)\big)^2 \geq \frac{r^2}{8\tau}.
    \end{equation*}
    Since \eqref{eq:almost-optimal} holds and $\varepsilon >0$ is arbitrary, we infer from the last two inequalities that $\cM_t(x_1, x_2) \geq \min(\frac{r^2}{8\tau}, P_{\min}) > 0$.
\end{proof}
To finish this section, we prove \cref{lemma:MMS_summary_lemma}.
\begin{proof}[Proof of \cref{lemma:MMS_summary_lemma}]
    Let $\tau > 0$.
    The fact that $\cM_{t}$ is an action as in \cref{def:action} is the content of \cref{lemma:action_minimizing_scheme};
    the fact that for $|\partial E_t|$ Lipschitz continuous with Lipschitz constant $L$ and $\tau \leq \frac{1}{L}$,
    $\cM_t$ is generated by curves as in \cref{def:generated_by_curves} is the content of the above \cref{lemma:generated_by_curves_MMS}.
    What is left to prove is that $\ruleM_{t}$ as defined in \cref{def:evolution_rule_MMS} is compatible with $\cM_t$ as in \cref{def:action},
    i.e., that for all $x, x' \in \Xx$, $(u_j)_{j \in \NN} \in \Xx$
    and increasing sequences $(s_j)_{j \in \NN} \in \NN$ 
    such that $u_0 = x$, $\lim_{j \rightarrow \infty} u_{s_j} = x'$ and
    \begin{equation}
        u_{j+1} \in \argmin_{y \in \Xx} \left\{ E_t(y) + \frac{1}{2\tau} d\big(u_{j}, y\big)^2 \right\}
    \end{equation}
    for all $s \in [0, \infty)_\Z$, we have that
    \begin{equation*}
        E_t(x) - E_t(x') = \cM_t(x, x').
    \end{equation*}
    However, this follows directly by applying \cref{lemma:minimizing_scheme_energy_diff} to the finite curves
    \begin{equation*}
        \big(i(x), (u_{0}, u_{1}), \dots, (u_{s_j-1}, u_{s_j}), i\MMS (u_{s_j})\big) \in \Gamma\big( i\MMS (x), i\MMS (u_{s_j})\big),
    \end{equation*}
    using continuity of $E_t$ and $\cM_t$
    and noticing that the convergence of $E_t(u_{j})$ implies that $E_t(u_{s_i}) - E\ttau\MMS(u_{s_i}) \rightarrow 0$
    as $s \rightarrow \infty$.
\end{proof}

\subsection{The BDF2 method for the gradient flow} \label{subsec:elaborate_BDF2}
In this section, we will explore the BDF2 discretization of the gradient flow and prove \cref{lemma:BDF2_summary_lemma}.
To the best of our knowledge, we derive here for the first time the action related to the BDF2 scheme.
We use freely the notation in \cref{eq:def_phase_space_BDF,eq:embedding_BDF,eq:def_GammaB,eq:def_costB,eq:extendded-cost-BDF}.  
In the rest of this section we denote the components of a curve $\gamma\colon I \subseteq \Z \to \PpB = \Xx \times \Xx \times \Xx$ by $\gamma_{-1}$, $\gamma_0$ and $\gamma_1$, and the components of points $\vx, \vx' \in \PpB$ by $(x_{-1}, x_0, x_1) \coloneqq \vx$ and $(x'_{-1}, x'_0, x'_1) \coloneqq \vx'$. As we described in \cref{rmk:lipschitz_tau_relation}, whenever we assume that \cref{ass:path+Lipsch} holds, we also assume that $\tau \leq \frac{1}{L}$, where $L$ is the Lipschitz constant mentioned in \cref{ass:path+Lipsch}.
Substituting \cref{eq:def_phase_space_BDF,eq:embedding_BDF,eq:def_GammaB,eq:def_costB} into \cref{eq:action_infimum} and simplifying, we obtain the action \cref{eq:action_infimum_BDF2}.

\begin{remark}
In the definitions concerning the minimizing movement scheme in \cref{eq:CostM_def,eq:def_E_MMS},
$\CostM_t$ resembled $E\MMS_{t, \tau}$. In contrast, in the definitions in \cref{eq:def_costB,eq:def_E_BDF},
$\CostB_{t}$ and $E\BDF_{t, \tau}$ look quite different, with differing fractions involving $\tau$.
Based on what we have seen on the minimizing movement scheme, one might instead expect the following definition instead
of \cref{eq:def_costB}:
\begin{equation}\label{eq:def_priceB_alternative}
    \widetilde{\CostB_t}(\vx) \coloneqq E_t(x_0) - E\BDF_{t, \tau}(x_0, x_{-1}) + \frac{1}{\tau} d(x_0, x_1)^2 - \frac{1}{4\tau}d(x_{-1}, x_1)^2.
\end{equation}
In fact, we could have chosen the definition of $\widetilde{\CostB_{t}}$ in~\eqref{eq:def_priceB_alternative}
in place of $\CostB_{t}$ in~\eqref{eq:def_costB}:
For every $x, x' \in \Xx$, every $(\gamma \in \Gamma(i\BDF(x), i\BDF (x'))$ with $\gamma \colon [a, b]_{\mathbb{Z}} \to \PpB$, and every $\tau > 0$,
we can use the fact that $d\big(\gamma_0(s), \gamma_1(s)\big) = d\big(\gamma_{-1}(s+1), \gamma_0(s+1)\big)$ for $s \in [a+1, b-1]_{\mathbb{Z}}$, while $d\big(\gamma_{-1}(a), \gamma_0(a)\big) = d\big(\gamma_0(a), \gamma_1(a)\big) = 0$ ---and likewise for $\gamma(b)$--- to see that
\begin{equation} \label{eq:BDF2_equivalence}
    \sum_{s=a+1}^{b}\CostB_t(\gamma(s)) = \sum_{s=a+1}^{b}\widetilde{\CostB_t}\big(\gamma(s)\big).
\end{equation}
On the other hand, our formulation has the advantage that $\CostB_{t}$ is nonnegative, as we show in the next lemma.
\end{remark} 

\begin{lemma} \label{lemma:BDF2_estimates}
    For any $\vx \in \PpB$, we have that
    \begin{equation} \label{eq:BDF2_nonnegativity}
        \CostB_t(\vx) \geq E_t(x_0) - E\BDF_{t, \tau}(x_0, x_{-1}),
    \end{equation}
    where equality can only be attained if $d(x_{-1}, x_0) = d(x_0, x_1)$.
    Furthermore, we have that
    \begin{equation}\label{eq:BDF2_price_nonnegative}
        E_t(x_0) - E\BDF_{t, \tau}(x_0, x_{-1}) \geq \frac{1}{4\tau} d(x_0, x_{-1})^2 \geq 0.
    \end{equation}
\end{lemma}
\begin{proof}
    The inequality~\eqref{eq:BDF2_nonnegativity} follows directly from the definition of $\CostB_{t}$ in \cref{eq:def_costB}
    and the fact that, for $x, x', x'' \in \Xx$, setting $\underline{d}_1 \coloneqq d(x, x')$
    and $\underline{d}_2 \coloneqq d(x', x'')$:
    \begin{equation}\label{eq:weird_square_formula}
        \frac{1}{4} d(x, x'')^2 \leq \frac{1}{4} (\underline{d}_1 + \underline{d}_2)^2
        \leq \frac{1}{2} (\underline{d}_1^2 + \underline{d}_2^2),
    \end{equation}
    where the first bound comes from the triangle inequality and the second bound
    is attained if $\underline{d}_1 = \underline{d}_2$.
    We get inequality~\eqref{eq:BDF2_price_nonnegative} using $x_{0}$ as a competitor for the infimum in the definition
    of $E\BDF_{t, \tau}$ in \cref{eq:def_E_BDF}.
\end{proof}
Using the reformulation \cref{eq:BDF2_equivalence}, we see that the BDF2 discretization of the gradient flows is indeed compatible with $\cB_t$.
\begin{lemma}\label{lemma:energy_diff_BDF2}
    For any $\vx$ in $\PpB$, we have that
    \begin{equation}\label{eq:BDF2_curve_cost_bigger_than_energy_diff_phase_space}
        \widetilde{\CostB_t}(\vx) \geq E_t(x_0) - E_t(x_1),
    \end{equation}
    where equality is attained if and only if $\vx = (x_{-1}, x_0, x_1)$ satisfies
    \begin{equation}\label{eq:BDF2_definition}
        x_1 \in \argmin_{y \in \Xx} \left\{ E_t(y) + \frac{1}{\tau} d(y, x_0)^2 - \frac{1}{4\tau} d(y,x_{-1})^2 \right\}.
    \end{equation}
    In particular, for any $x, x' \in \Xx$, we have that
    \begin{equation} \label{eq:BDF2_action_bigger_than_energy_diff}
        \cB_{t} (x, x') \geq E_t(x) - E\BDF_{t, \tau}(x', x') \geq E_t(x) - E_t(x').
    \end{equation}
\end{lemma}
\begin{proof}
    To show~\eqref{eq:BDF2_curve_cost_bigger_than_energy_diff_phase_space}, we start with the following calculation: for every $\underline{x} \in \PpB$
    \begin{equation*}
        \begin{split}
        \widetilde{\CostB_t}(\vx) & = E_t(x_0) - E\BDF_{t, \tau}(x_0, x_{-1}) + \frac{1}{\tau} d(x_1, x_0)^2 - \frac{1}{4\tau}d(x_1, x_{-1})^2                                    \\
        & =                           E_t(x_0) - \inf_{y \in \Xx} \left\{ E_t(y) + \frac{1}{\tau} d(y, x_0)^2 - \frac{1}{4\tau} d(y, x_{-1})^2 \right\} \\
        & \qquad
        + \frac{1}{\tau} d(x_1, x_0)^2 - \frac{1}{4\tau}d(x_1, x_{-1})^2                                                               \\
            & \geq                        E_t(x_0) - E_t(x_1).
    \end{split}
    \end{equation*}
    In particular, equality is attained if and only if the infimum is achieved at $x_1$, i.e., if $\vx$ fulfills \cref{eq:BDF2_definition}.
    To show \cref{eq:BDF2_action_bigger_than_energy_diff}, we first recursively apply
    \cref{eq:BDF2_curve_cost_bigger_than_energy_diff_phase_space} to any $\gamma \in \Gamma(\vx, \vx')$ for
    $\vx, \vx' \in \PpB$ to see that
    \begin{equation*}
        \sum_{s=a+1}^{b}\widetilde{\CostB_t}\big(\gamma(s)\big) \geq E_t\big(\gamma_1(a)\big) - E_t\big(\gamma_1(b)\big).
    \end{equation*}
    \cref{eq:BDF2_action_bigger_than_energy_diff} then follows from the definition of $\cB_{t}$,
    \cref{eq:BDF2_equivalence} and the fact that $ i\BDF (x_1)_1 = x_1$ and $ i\BDF  (x_2)_1 = x_2$.
\end{proof}
Next, we investigate the relation between $\cB_t$ and $\cM_t$. To avoid confusion, we denote the admissible curves for the BDF2 method, defined in \cref{eq:def_GammaB}, by $\Gamma\BDF$,
and we denote the admissible curves for the minimizing movement scheme, defined in \cref{eq:def_GammaM}, by $\Gamma\MMS$.
We start by defining a bijection between certain
sets of admissible curves. For $\vx, \vx' \in \PpB$, we define $F\colon \Gamma\BDF(\vx, \vx') \to \Gamma\MMS\big((x_0, x_1), (x'_{-1}, x'_0)\big)$ as follows:
\begin{equation}\label{eq:BDF2_MMS_curves_bijection}
    F(\gamma) \coloneqq \bigg(\big(\gamma_0(a), \gamma_1(a)\big), \dots, \big(\gamma_0(b-1), \gamma_1(b-1)\big)\bigg),
\end{equation}
where we choose the domain of $F(\gamma)$ to be $[a, b-1]_{\mathbb{Z}}$ and where $a < b \in \Z$ are chosen such that $\dom(\gamma) = [a, b]_\Z$.
From the definitions of $\Gamma\BDF$ and $\Gamma\MMS$, both injectivity and surjectivity of $F$ follow right away.
\begin{lemma}\label{lemma:BDF2_relation_to_minimizing_scheme}
    For every $\tau > 0$, every $\vx, \vx' \in \PpB$, and every $\gamma \in \Gamma\BDF (\underline{x}, \underline{x}')$ we have that
    \begin{equation} \label{eq:BDF2_cost_relation_to_minimizing_scheme}
        \frac{1}{5} \sum_{s=a+1}^{b-1}\,\,\CostM[\frac{\tau}{2}]_t(F(\gamma)(s)) \leq \sum_{s=a+1}^{b}\CostB_t(\gamma(s)) \leq 3 \sum_{s=a+1}^{b-1}\CostM_t(F(\gamma)(s)) + \CostB_t\big( \gamma(a)  \big) .
    \end{equation}
    In particular, for $x, x' \in \Xx$, we have that
    \begin{equation}\label{eq:BDF2_relation_to_minimizing_scheme}
        \frac{1}{5}\,\,\,\cM[\frac{\tau}{2}\,]_{t}(x, x') \leq \cB_{t, \tau}(x, x') \leq 3\, \cM_{t, \tau}(x, x') + \CostB_t( i\BDF  (x)).
    \end{equation}
    Furthermore, for every $x \in \Xx$, we have that
    \begin{equation} \label{eq:BDF2_stillstand_cost_relation_to_minimizing_scheme}
        \CostB_{t}( i\BDF  (x)) = \,\,\, \CostM[\frac{2\tau}{3}]_{t}( i\MMS (x)).
    \end{equation}
\end{lemma}
\begin{proof}
    For the first inequality, let $\gamma \in \Gamma\BDF(\underline{x}, \underline{x}')$ be such that $\gamma \colon [a, b] _\Z \to \PpB$ and choose a sequence of quasi-optimal points $(y_s)_{s=a+1}^b \in \Xx$ such that
    \begin{equation*}
         E_t(y_s) + \frac{1}{\tau} d\big(\gamma_0(s), y_s\big)^2 \leq E_{t, \frac{\tau}{2}}\MMS (\gamma_{0} (s)) + \frac{\varepsilon}{b-a+1}.
    \end{equation*}
    We then have
    \begin{equation*}
        \begin{split}
          \sum_{s=a+1}^{b-1}\,\, & \CostM[\frac{\tau}{2}]_t(F(\gamma)(s))  - \varepsilon \\
         &
        \leq \sum_{s=a+1}^{b-1} \bigg(E_t(\gamma_0(s)) - \underbrace{\big(E_t(y_s) + \frac{1}{\tau} d(\gamma_0(s), y_s)^2\big)}_{\geq E\BDF_{t, \tau}(\gamma_0(s), \gamma_{-1}(s))} +  \frac{1}{\tau}  d\big(\gamma_0(s), \gamma_1(s)\big)^2 \bigg)\\
         & \leq \sum_{s=a+1}^{b-1} \bigg(E_t\big(\gamma_0(s)\big) - E\BDF_{t, \tau}\big(\gamma_0(s), \gamma_{-1}(s)\big)\bigg) \\
         & \qquad 
         + 4  \sum_{s=a+1}^{b-1}  \underbrace{\frac{1}{4\tau} d\big(\gamma_{-1}(s+1), \gamma_0(s+1)\big)^2}_{\leq  E_t\big( \gamma_{0}(s) \big) - E\BDF_{t, \tau}\big(\gamma_0(s), \gamma_{-1}(s)\big) \text{ by \cref{lemma:BDF2_estimates}}} \\
         & \leq 5 \sum_{s=a+1}^b \bigg(E_t\big(\gamma_0(s)\big) - E\BDF_{t, \tau}\big(\gamma_0(s), \gamma_{-1}(s)\big)\bigg) \leq 5 \sum_{s=a+1}^b \CostB_t\big(\gamma(s)\big),
         \end{split}
    \end{equation*}
    where we used \cref{eq:BDF2_nonnegativity} in the last inequality.
    Since $\varepsilon$ was arbitrary, the left inequality in~\eqref{eq:BDF2_cost_relation_to_minimizing_scheme} follows.\\
    To prove the right inequality in~\eqref{eq:BDF2_cost_relation_to_minimizing_scheme},
    we likewise fix $\gamma\colon [a, b]_\Z \to \PpB$, $\gamma \in \Gamma\BDF (\underline{x}, \underline{x}')$ and choose a sequence of quasi-optimal points $(y_s)_{s=a+1}^b$ such that
    \begin{equation*}
        E\ttau\BDF\big(\gamma_0(s), \gamma_{-1}(s)\big) + \frac{\varepsilon}{b-a+1} \geq E_t(y_s) + \frac{1}{\tau} d\big(\gamma_0(s), y_s\big)^2 - \frac{1}{4\tau} d\big(\gamma_{-1}(s), y_s\big)^2.
    \end{equation*}
    Owing to the estimate in \cref{eq:weird_square_formula}, we find
    \begin{align*}
        \begin{aligned}\sum_{s=a+1}^{b-1} \OLCostB(\gamma) - \varepsilon \vphantom{\frac{1}{4\tau}} \, \\
            \vphantom{\frac{1}{4\tau}}
        \end{aligned} & \begin{aligned}
                            \leq \sum_{s=a+1}^{b-1} \bigg( & E_t\big(\gamma_0(s)\big) - E_t(y_s) - \frac{1}{\tau} d\big(\gamma_0(s), y_s\big)^2 + \frac{1}{4\tau} d\big(\gamma_{-1}(s), y_s\big)^2 \\
                                                           & + \frac{1}{2\tau}\big(\overline{d}_s^2 + \underline{d}_s^2\big) - \frac{1}{4\tau}\overline{\underline{d}}_s^2\bigg)
                        \end{aligned}                                                                                                                                             \\
                    & \!\!\!\!\!\!\!\!\!
                    \leq \sum_{s=a+1}^{b-1} \Bigg(E_t\big(\gamma_0(s)\big) - \underbrace{\big(E_t(y_s) + \frac{1}{2\tau} d(\gamma_0(s), y_s)^2\big)}_{\geq E\ttau\MMS(\gamma_0(s))} + \frac{1}{2\tau}\underline{d}_s^2 + \frac{1}{2\tau}\big(\overline{d}_s^2 + \underline{d}_s^2\big)\Bigg) \\
                    & \!\!\!\!\!\!\!\!\! 
                    \leq \sum_{s=a}^{b-1} \left(E_t\big(\gamma_{-1}(s)\big) - E\ttau\MMS\big(\gamma_0(s)\big) + 3 \frac{1}{2\tau}\overline{d}^2\right)
        \leq 3 \sum_{s=a}^{b-1} \CostM_t\big(F(\gamma)(s)\big),
    \end{align*}
    where we set $\underline{d}_s \coloneqq d\big(\gamma_{-1}(s), \gamma_0(s)\big)$, $\overline{d}_s \coloneqq d\big(\gamma_0(s), \gamma_1(s)\big)$ and
    $\overline{\underline{d}}_s \coloneqq d\big(\gamma_{-1}(s), \gamma_1(s)\big)$, and where we used
    \cref{eq:weird_square_formula} in the second line.
    Since $\varepsilon$ was arbitrary, the second inequality in~\eqref{eq:BDF2_cost_relation_to_minimizing_scheme} follows.
    \cref{eq:BDF2_relation_to_minimizing_scheme} follows directly by applying the infimum to the preceding inequality,
    while \cref{eq:BDF2_stillstand_cost_relation_to_minimizing_scheme} follows from a simple
    substitution into the definitions of $\CostM_t$ and $\CostB_t$ in \cref{eq:CostM_def,eq:def_costB}, respectively.
\end{proof}
The preceding lemma allows us to adapt some results from the minimizing movement scheme directly.
\begin{lemma}\label{lemma:coerciveness_BDF2}
    Let us assume \cref{ass:metric_space,ass:reg_E,ass:der_time,ass:unif_coerc}.
    For all $K \subseteq \PpB$ compact and $C > 0$,
    there exist $K' \subseteq \PpB$ compact such that the following implication holds
    for all $\vx, \vx' \in K$ and $\gamma \in \Gamma(\vx, \vx')$:
    \begin{equation} \label{eq:curves_bounded_coercive_bdf}
        \sum_{s=a+1}^b \CostB_{t}(\gamma(s)) \leq C \implies \im(\gamma) \subseteq K'.
    \end{equation}
    In particular, for each $K \subseteq \PpB$ there exists a compact set $K' \subseteq \PpB$ such that for every $\varepsilon \in (0, 1)$ the following implication holds
    for all $\vx, \vx' \in K$ and $\gamma \in \Gamma(\vx, \vx')$:
    \begin{equation} \label{eq:curves_coercive_implications_bdf}
        \sum_{s=a+1}^b \CostB_{t}(\gamma(s)) \leq  \cBg_t  (\vx, \vx') + \varepsilon \implies \im(\gamma) \subseteq K'.
    \end{equation}
\end{lemma}
\begin{proof}
    For the implication in \eqref{eq:curves_bounded_coercive_bdf}, we use the first inequality in~\eqref{eq:BDF2_cost_relation_to_minimizing_scheme}
    and the corresponding implication~\eqref{eq:curves_bounded_coercive_mms} in~\cref{lemma:coerciveness_minimizing_scheme} for  $\CostM[\frac{\tau}{2}]_{t}$.
    For~\eqref{eq:curves_coercive_implications_bdf}, we use the second inequality in~\eqref{eq:BDF2_relation_to_minimizing_scheme}
    and the continuity of $\cM_{t}$, $\CostM_t$ and $\CostB_t$ to find an upper bound $C'$ on $\cBg_{t} (\vx, \vx')$ for $\vx, \vx' \in K$,
    and we apply~\eqref{eq:curves_bounded_coercive_bdf} with $C = C' + \varepsilon$.
\end{proof}
\begin{lemma}\label{lemma:BDF2_is_action}
    Let us assume \cref{ass:metric_space,ass:reg_E,ass:der_time,ass:unif_coerc}. Then, $\cB_t$ is an action according to \cref{def:action}.
\end{lemma}
\begin{proof}
    The proof is analogous to the one for the corresponding result in the minimizing movement scheme,
    namely \cref{lemma:action_minimizing_scheme}.
\end{proof}
We are now ready to prove that the action $\cB_t$ is generated by curves.
\begin{lemma}\label{lemma:generated_by_curves_BDF2}
    Let us assume \cref{ass:metric_space,ass:reg_E,ass:unif_coerc,ass:der_time,ass:conn_comp,ass:path+Lipsch}.
    Then, for all $\tau \leq \frac{1}{L}$,
    the action $\cB_t$ is generated by curves as in \cref{def:generated_by_curves}.
\end{lemma}
\begin{proof}
    First note that by \cref{lemma:BDF2_estimates}, $\CostB_t$ is nonnegative.
    For the continuity of $E\BDF_{t, \tau}$ ---and thus of $\CostB_t$--- we
    argue as in \cref{rmk:cost_lambda_continuous} and note that
    only points which are either in the set $\{y \in \Xx  \mid d(y, x) \leq d(x, x')\}$ or in the set $\{y \in \Xx \mid E_t(y) \leq E_t(x)\}$ contribute to the infimum of
    \begin{equation}\label{eq:redefinition_infimum_E_BDF}
        y \mapsto E_t(y) + \frac{1}{\tau} d(x, y)^2 - \frac{1}{4\tau} d(x', y)^2 \geq E_t(y) + \frac{1}{\tau} d(x, y)^2 - \frac{1}{4\tau} \big(d(x, y) + d(x, x')\big)^2
    \end{equation}
    in \cref{eq:def_E_BDF}: If $y$ is in neither set, then the value of~\eqref{eq:redefinition_infimum_E_BDF}
    is large than when using $y=x$ as a competitor.
    Due to the coercivity of $E_t$, both those sets are compact.
    Thus, $\CostB_t$ is itself continous on $\PpB \times [0, T]$, as it is locally a supremum of equicontinous functions.
    As for the properties, \cref{prop:curve_restriction} is immediate
    and \cref{prop:curves_coercive} follows directly from \cref{lemma:coerciveness_BDF2}.
    We now prove the remaining three properties.\\
    \textbf{\cref{prop:cost_nonzero}: }
    We use the first statement in \cref{lemma:BDF2_estimates} to see that the assumption
    \begin{equation*}
        \CostB_t(\vx) = 0
    \end{equation*}
    implies that $E_t(x_0) - E\BDF_{t, \tau}(x_0, x_{-1}) = 0$ and $d(x_{-1}, x_0) = d(x_0, x_1)$.
    Using the second statement in \cref{lemma:BDF2_estimates}, we see that he former implies that $x_{-1} = x_0$,
    which by the latter implies that also $x_{-1} = x_0 = x_1$ and thus that $\vx =  i\BDF (x)$ for some $x \in \Xx$.
    By \cref{eq:BDF2_stillstand_cost_relation_to_minimizing_scheme}, we then also have that
    $\,\,\CostM[\frac{2\tau}{3}]_{t}( i\MMS (x)) = 0$, which by the assumptions $\tau \leq \frac{1}{L}$
    and \cref{lemma:zero_gradient_zero_cost} implies that $ |\partial E_t| (x) = 0$. \\
    \textbf{\cref{prop:no_cost_in_within_conn_component}: }
    Note that under the given assumptions, both $\cM_t$ and $\cM[\frac{\tau}{2}]_t$ are generated by curves
    as in \cref{def:generated_by_curves}.
    Furthermore, by virtue of \cref{lemma:zero_gradient_zero_cost} and \cref{eq:BDF2_stillstand_cost_relation_to_minimizing_scheme},
    under the given assumptions $\tau \leq \frac{1}{L}$, we have that
    \begin{equation*}
        |\partial E_t| (x) = 0 \iff \CostM[\frac{\tau}{2}]_t( i\MMS (x)) = 0 \iff \CostB_t( i\BDF  (x)) = 0 \iff \CostM_t( i\MMS  (x)) = 0.
    \end{equation*}
    Using these equivalences, the ``if'' direction of \cref{prop:no_cost_in_within_conn_component} follows directly from the upper bound by  $\cM_{t}$ in~\eqref{eq:BDF2_relation_to_minimizing_scheme} and the respective \cref{prop:no_cost_in_within_conn_component}
    of  $\cM_{t}$. 
    The ``only if'' direction follows from the lower bound by $\cM[\frac{\tau}{2}]_{t}$ in~\eqref{eq:BDF2_relation_to_minimizing_scheme}
    and the respective \cref{prop:no_cost_in_within_conn_component}
    of  $\, \cM[\frac{\tau}{2}]_{t}$.\\
    \textbf{\cref{prop:cheap_fuel_cheap_action}: }
    Let $K \subseteq \PpB$ be compact and $\varepsilon > 0$ be chosen. By continuity of $\CostB_t$,
    we can choose $L'$ such that for all $\vx, \vx' \in K$, the following implication holds:
    \begin{equation} \label{eq:cheap_close_cheap_bdf2}
        \CostB_t(\vx) \leq \frac{\varepsilon}{8}, \quad d_{\PpB} (\vx, \vx') \leq L' \implies \CostB_t(\vx') \leq \frac{\varepsilon}{4}.
    \end{equation}
    Since $\CostB_t$ is continuous and vanishes only on the diagonal of $\PpB$,
    we can further pick some $\varepsilon' \leq \frac{\varepsilon}{8}$ such that for all $\vx \in K$,
    \begin{equation}\label{eq:cheap_small_diameter}
        \CostB_t(\vx) \leq \varepsilon' \implies \max\big\{d(x_{-1}, x_0), d(x_0, x_1), d(x_1, x_{-1})\big\} \leq \frac{L'}{2}.
    \end{equation}
    Then, for any $\vx, \vx' \in K$ such that $\CostB_t(\vx^i) \leq \varepsilon'$ for $i \in \{1, 2\}$ and $d(\vx, \vx') \leq \frac{L'}{2}$,
    we set $\widetilde{\vx} = (x_0, x_1, x'_{-1})$ and $\widetilde{\vx}' = (x_1, x'_{-1}, x'_0)$.
    Then, by \cref{eq:cheap_small_diameter} and using the triangle inequality, $ d_{\PpB} (\widetilde{\vx}, \vx) \leq L'$
    and $ d_{\PpB} (\widetilde{\vx}', \vx') \leq L'$. Thus, using \cref{eq:cheap_close_cheap_bdf2}, we have that
    $\CostB_t(\widetilde{\vx}) \leq \frac{\varepsilon}{4}$ and $\CostB_t(\widetilde{\vx}') \leq \frac{\varepsilon}{4}$.
    We finish the proof by noting that $(\vx, \widetilde{\vx}, \widetilde{\vx}', \vx') \in \Gamma(\vx, \vx')$.
\end{proof}
To finish this section, we proof \cref{lemma:BDF2_summary_lemma}.
\begin{proof}[Proof of \cref{lemma:BDF2_summary_lemma}]
    Let $\tau > 0$.
    The fact that $\cB_{t}$ is an action as in \cref{def:action} is the content of \cref{lemma:BDF2_is_action};
    the fact that for $|\partial E_t|$ $L$-Lipschitz ($L > 0$) and $\tau \leq \frac{1}{L}$,
    $\cB_{t}$ is generated by curves as in \cref{def:generated_by_curves} is the content of the above \cref{lemma:generated_by_curves_BDF2}.
    What is left to prove is that $\ruleB_{t}$ as defined in \cref{def:evolution_rule_BDF2} is compatible with $\cB_t$ as in \cref{def:action},
    i.e., that for all $x, x' \in \Xx$, $(u_j)_{j \in \NN} \in \Xx$ and increasing sequences $(s_j)_{j \in \NN} \in \NN$
    such that $u_0 = x$, $\lim_{j \rightarrow \infty} u_{s_j} = x'$ and
    \begin{equation}
        u_{j+1} \in \argmin_{y \in \Xx} \left\{ E_t(y) + \frac{1}{\tau} d(y,u_{j})^2 - \frac{1}{4\tau} d(y,u_{j-1})^2 \right\},
    \end{equation}
    for all $j \in [0, \infty)_\Z$---where we set $u_{-1} = x$---, we have that
    \begin{equation*}
        E_t(x) - E_t(x') = \cB_t(x, x').
    \end{equation*}
    However, this follows directly by applying recursively applying \cref{eq:BDF2_curve_cost_bigger_than_energy_diff_phase_space} from \cref{lemma:energy_diff_BDF2} to the sequences
    \begin{equation*}
        \gamma^{(i)} \coloneqq \big(i(x), (x, x, u_1), (x, u_1, u_2), (u_1, u_2, u_3), \dots, (u_{s_j-2}, u_{s_j-1}, u_{s_j}), (u_{s_j-1}, u_{s_j}, u_{s_j}), i(u_{s_j})\big)
    \end{equation*}
    to see that
    \begin{equation*}
        \cB_t(x, u_{s_j}) \leq E_t(x) - E_t(u_{s_j}) + \CostB\big((u_{s_j-1}, u_{s_j}, u_{s_j})\big) + \CostB\big(i(u_{s_j})\big)
    \end{equation*}
    We finish the proof by using continuity of $E_t$ and $\cM_t$
    and noting that the convergence of $E_t(u_j)$ implies that
    $ \CostB_{t} \big((u_{s_j-1}, u_{s_j}, u_{s_j})\big) + \CostB_t \big(i(u_{s_j})\big) \rightarrow 0$ as $j \rightarrow \infty$.
\end{proof}
%
\section{Numerical experiments}\label{sec:num}
To illustrate our framework, we present an experiment where we simulate the breaking of an elastic rod (\cref{fig:key_frames})\footnote{Videos of the simulations, as well as the full code are available at \url{https://github.com/duesenfranz/quasistatic_evolutions_simulations}.}.
The rod is modeled as a chain of $n$ particles connected by $n-1$ springs;
each of the springs can either be intact or broken.
The total energy of the system is a trade-off of two energies:
On one hand, the potential energy $V$ of a spring of length $l$ is given as
$V(l) = \frac{1}{2} k (l - \bar{l})^2$ if the spring is intact, and zero otherwise---where $\bar{l}$ is the rest length of the spring. 
The surface energy $S$ of the $i$th spring, on the other hand, is modelled as $S = \sigma_i$
if the spring is broken, and zero otherwise---here, $\sigma_i \sim \mathcal{N}(\bar{\sigma}, \varepsilon)$ is the surface constant of the $i$th spring, where we add slight noise to break the symmetry.
The total energy of the system is then given by
\begin{equation}
    E(x) = \sum_{i \in \mathbb{B}} \sigma_i + \sum_{i \in \mathbb{I}} V(|x_i - x_{i+1}|),
\end{equation}
where $\mathbb{I} \subseteq [n-1]$ is the set of intact springs, $\mathbb{B} = [n-1] \setminus \mathbb{I}$ is the set of broken springs, and $x \in \R^{n}$ is the vector of the positions of the particles.
\begin{figure}[H]
    \centering
    \includegraphics[width=\textwidth]{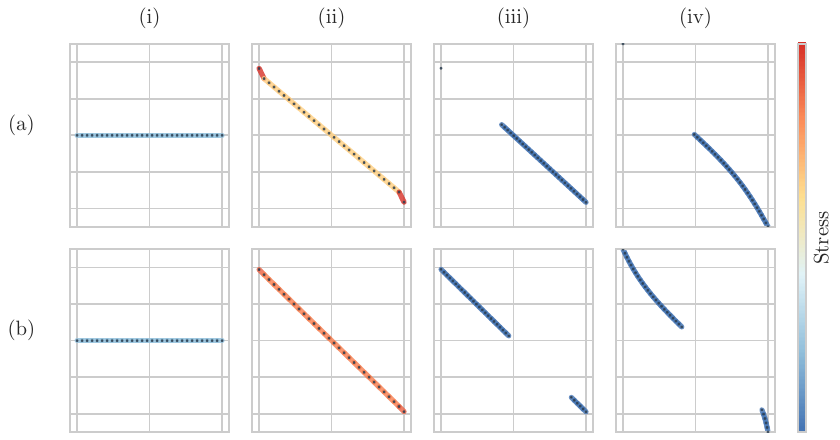}
    \caption{The simulation for an elastic rod, once for (a) $\delta=\frac{1}{15}$ and once for (b) $\delta=\frac{1}{240}$.
    From left to right, we plot (i) the initial configuration, (ii) the configuration right before the system transition which leads to the breakage of the rod, (iii) the configuration right after the transition, and (iv) the configuration at the end of the time horizon.
    In (a.ii), the stress is concentrated on the two end segments, which leads to the rod breaking early. This phenomenon is the result of the approximation error due to the large step size $\delta = \frac{1}{15}$.}
    \label{fig:key_frames}
\end{figure}
To model the transition between the intact and broken states,
we introduce a latent variable $z \in [0, 1]^{n-1}$ for each spring $i$.
The variable $z_i$ is a measure of how broken the $i$th spring is:
$z_i = 0$ means that the spring is intact, while $z_i = 1$ means that the spring is completely broken.
With this notation, the total energy of the system becomes
\begin{equation}
    E\big((x, z)\big) = \sum_{i=1}^{n-1} z_i \cdot \sigma_i + (1 - z_i) \cdot V(|x_i - x_{i+1}|).
\end{equation}
While $z \in [0, 1]^{n-1}$ theoretically allows for unphysical states where $z_i \not\in \{0, 1\}$,
such states will not appear in practice:
In the generic case, we expect that ${V(|x_i - x_{i+1}|) \neq \sigma_i}$, in which case a critical point must have $z_i \in \{0, 1\}$.

In the simulation, we fix the first and the last particle to be at position
${x_{1}(t) \coloneqq (0,  -t \cdot h)}$ and $x_{n}(t) \coloneqq (1, t \cdot h)$ respectively, at time $t$.
We set the state space to be $E \coloneqq \R^{{n-2}\times 2} \times [0, 1]^{[n-1]}$, since each system state is given by the positions of the inner particles $x' \coloneqq (x_2, \dots, x_{n-1}) \in \R^{n-2 \times 2}$ as well as the latent variables
$z \coloneqq z_1, \dots, z_{n-1} \in [0, 1]^{n-1}$. In this setup, the total energy is given by
\begin{align}
    E_t\big((\bar{x}, \bar{z})\big) & = \overbrace{ (1 - z_1) \cdot V(|x_2 - (0, -t \cdot h)|)}^{\text{Tension in first segment}}
    +  \overbrace{\sum_{i=2}^{n-2}(1 - z_i) \cdot V(|x_{i+1} - x_i|)}^{\text{Tension in inner segments}}                         \\
                            & + \underbrace{(1 - z_{n-1}) \cdot V(|x_{n-1} - (1, t \cdot h) |)}_{\text{Tension in last segment}}
    + \underbrace{\sum_{i=1}^{n-1} z_i \cdot \sigma_i}_{\text{Surface energy}}. \nonumber
\end{align}
Finally, we set the starting state to be the local minimum where the whole rod is intact and
the points $x$ are equidistantly spaced along the line from $x_{1, 0}$ to $x_{n, 0}$.
The discrete quasistatic evolution follows this local minimum over time,
positioning the points $x$ equidistantly between the endpoints $x_{1}(t)$ and $x_{n}(t)$.
Eventually, the potential energy of one of the springs exceeds the surface energy.
At this point, this local minimum disappears and the discrete quasistatic evolution jumps to a lower energy level---the one where the spring is broken.
Interestingly, this new local minimum of the rod being broken is not an isolated critical point, but a whole manifold.

Numerically, we use a BDF2 approximation of the gradient flow, corresponding to the action explained in \cref{subsec:elaborate_BDF2}, with $\tau=0.1$; we stop the gradient descent once the energy difference between two consecutive steps is less than $10^{-5}$.

We can observe the expected convergences in the quantities involved in the energy balance: First, $\mu^\delta$ converges to zero everywhere but at a single point---the point where the rod breaks---as $\delta$ goes to zero (\cref{fig:convergence_mu}). Furthermore, $\mathcal{D}^\delta$ converges to a piecewise smooth function, with a single jump point where the rod breaks (\cref{fig:convergence_D}).
The same convergences can be seen when comparing $E_t(\eta^\delta(t))$ and $\int_{0}^{t} \mathcal{D}^\delta(s) \diff s$ for different values of $\delta$ (\cref{fig:convergence_energy_sums}).

\begin{figure}[h]
    \centering
    \includegraphics[width=0.8\textwidth]{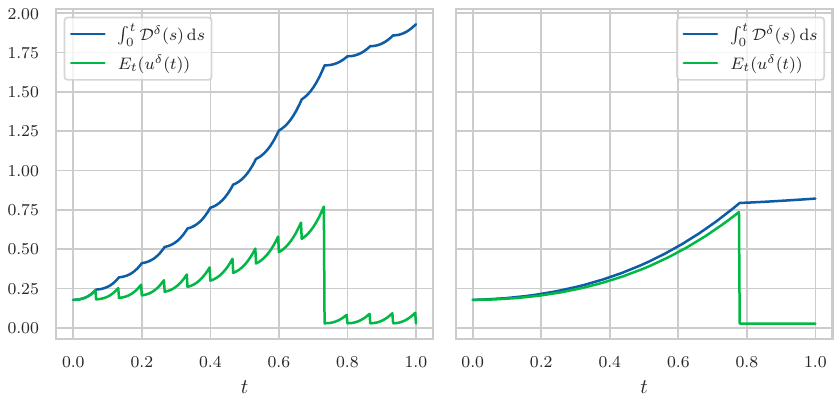}
    \caption{Comparing $E_t(\eta^\delta(t))$ and $\int_{0}^{t} \mathcal{D}^\delta(s) \diff s$, for $\delta=\frac{1}{15}$ on the left and $\delta=\frac{1}{240}$ on the right.}
    \label{fig:convergence_energy_sums}
\end{figure}

\begin{figure}[H]
    \centering
    \begin{subfigure}{0.48\textwidth}
        \centering
        \includegraphics[width=\linewidth]{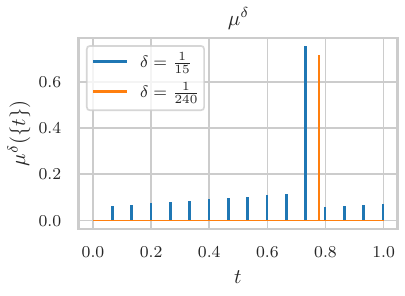}
        \caption{$\mu^\delta$ over the time horizon.}
        \label{fig:convergence_mu}
    \end{subfigure}
    \hfill
    \begin{subfigure}{0.48\textwidth}
        \centering
        \includegraphics[width=\linewidth]{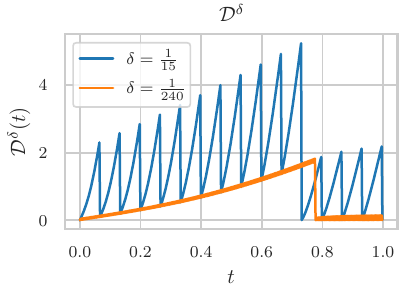}
        \caption{$\mathcal{D}^\delta$ over the time horizon.}
        \label{fig:convergence_D}
    \end{subfigure}
    \caption{The quantities involved in the energy balanced for the elastic rod simulation, for discrete quasistatic evolutions with different values of $\delta$.
    For $\delta$ going to zero, the support of $\mu^\delta$ collapses to the single point where the rod breaks at $t = 0.78$. Furthermore, $\mathcal{D}^\delta$ converges to a piecewise smooth function,
    with a single jump point at this breakage time. After the rod is broken,
    the potential energy is constant zero for steady states and $D^\delta$ converges to zero as $\delta$ goes to zero.}
    \label{fig:convergences}
\end{figure}

\appendix

\section{Structure of the measure in the energy balance}

In this part, we seek to establish that the non-negative measure $\bar \mu$ appearing in the energy balance in \cref{thm:main_result_complete} is purely atomic. 
Under suitable extra assumptions, this fact is expected, as it has already been observed for quasi\-static evolutions in finite-dimensional Hilbert spaces constructed with limiting arguments (see, e.g., \cite{AR17, Scilla_2018,minotti2018viscous}).
In this section, we adapt to our setting the arguments developed in \cite[Section~5.1]{Scilla_2018}.

We observe that, by virtue of \cref{ass:t_der_quotient}, we can rewrite  the energy balance for the quasi\-static evolution $\hat \eta\colon [0,T]\to\X$ as follows:
\begin{equation} \label{eq:energy_balance_quotient}
            \hat E\big(\hat\eta^+(t)\big) - \hat E\big(\hat\eta^-(s)\big) = \int_s^t \partial_t \hat E \big(\hat\eta(\tau)\big) \diff \tau - \bar \mu([s, t]).
\end{equation}
Moreover, we introduce the non-increasing function $f\colon [0,T] \to \R$ as
\begin{equation} \label{eq:def_ancill_funct_mu}
    f(t) \coloneqq \hat E\big( \hat \eta(t) \big) - \int_0^t \partial_t \hat E \big(\hat\eta(\tau)\big) \diff \tau,
\end{equation}
and we use the notation $f_+(t) \coloneqq \lim_{s\to t^+} f(s)$. We observe right away that
\begin{equation*}
    f_+(t) = \lim_{s\to t^+} \left(\hat E\big(\hat \eta(s) \big) - \int_0^s \partial_t \hat E \big(\hat\eta(\tau)\big) \diff \tau \right) = \hat E(\hat \eta^+(t)\big) - \int_0^t \partial_t \hat E \big(\hat\eta(\tau)\big) \diff\tau.
\end{equation*}

\begin{proposition} \label{prop:dini_der}
    Let us assume \cref{ass:conn_comp,ass:der_time,ass:metric_space,ass:path+Lipsch,ass:PL+time_der_slope,ass:reg_E,ass:unif_coerc,ass:t_der_quotient}.
    Let  $f\colon [0,T] \to \R$ be defined as in \cref{eq:def_ancill_funct_mu}. Then, for every $t\in [0,T)$, we have that
    \begin{equation*}
        D_+ f_+(t) \coloneqq \liminf_{h\to 0^+} 
        \frac{f_+(t+h)-f_+(t)}{h} \geq 0,
    \end{equation*}
    i.e., the Dini lower right derivative of $f_+$ is non-negative.
\end{proposition}

\begin{proof}
    For every $s\in [0,T]$, let us take $x_{s}\in \Xx$ such that $(s,x_{s})\in \hat \eta^+(s)$.
    Then, we fix $t\in [0,T)$ and we observe that
    \begin{equation*}
        \begin{split}
            f_+(t+h)-f_+(t) &=
            \hat E\big(\eta^+(t+h)\big) - \hat E\big(\eta^+(t)\big)
            - \int_t^{t+h} \partial_t \hat E \big(\hat\eta(\tau)\big) \diff \tau \\
            &= E_{t+h}(x_{t+h}) - E_t(x_t) - \int_t^{t+h} \partial_t \hat E \big(\hat\eta(\tau)\big) \diff \tau \\
            & = E_{t}(x_{t+h}) - E_t(x_t)
            + \int_t^{t+h} \left( 
            \partial_t E_\tau(x_{t+h})
            - \partial_t \hat E \big(\hat\eta(\tau)\big) 
            \right) \diff \tau
        \end{split}
    \end{equation*}
    for every $h>0$.
    Moreover, the expression in the intergal can be rewritten as follows:
    \begin{equation} \label{eq:app_int_term_identity}
        \begin{split}
            \partial_t E_\tau(x_{t+h})
            - \partial_t \hat E \big(\hat\eta(\tau)\big) &=
            \partial_t \hat E\big( [(\tau,x_{t+h})] \big)
            - \partial_t \hat E \big(\hat\eta(\tau)\big) \\
            & = \left( \partial_t \hat E\big( [(\tau,x_{t+h})] \big)
            -   \partial_t \hat E\big( \hat \eta^+(t+h) \big) \right) \\
            &\qquad +
            \left( \partial_t \hat E\big( \hat \eta^+(t+h) \big) 
            - \partial_t \hat E\big( \hat \eta^+(t) \big) \right) \\ 
            &\qquad +
            \left( \partial_t \hat E\big( \hat \eta^+(t) \big) 
            - \partial_t \hat E \big(\hat\eta(\tau)\big) \right).
        \end{split}
    \end{equation}
    Since $[(\tau,x_{t+h})] \to_{\X} [(t+h,x_{t+h})] = \hat \eta(t+h)$ as $\tau\to t+h$, it follows that the time derivative $\partial_t \hat E\big( [(\tau,x_{t+h})] \big) \to \partial_t \hat E\big( \hat \eta^+(t+h) \big)$ as $\tau\to t+h$, owing to \cref{ass:t_der_quotient}.
    Moreover, recalling that $\tau\in [t,t+h]$, that $\hat \eta^+$ is right-continuous and that $\hat \eta^+(\tau)\to_\X \hat \eta^+(t)$ as $\tau\to t^+$, by virtue of \cref{eq:app_int_term_identity} and \cref{ass:t_der_quotient}, we conclude that
    \begin{equation*}
        \int_t^{t+h} \left( 
            \partial_t E_\tau(x_{t+h})
            - \partial_t \hat E \big(\hat\eta(\tau)\big) 
            \right) \diff \tau = o(h),
    \end{equation*}
    yielding the identity
    \begin{equation} \label{eq:app_liminf_equiv}
        \liminf_{h\to 0^+} 
        \frac{f_+(t+h)-f_+(t)}{h} =
        \liminf_{h\to 0^+} 
        \frac{E_{t}(x_{t+h}) - E_t(x_t)}{h}.
    \end{equation}
    Before proceeding, we need to carefully choose the representative $(t,y)$ of $\hat \eta^+(t)$, and to replace $(t,x_t)$ if necessary.
    Indeed, let us firs consider $h_n\searrow 0$ as $n\to\infty$ such that
    \begin{equation} \label{eq:app_liminf_subseq}
        \liminf_{h\to 0^+} 
        \frac{E_{t}(x_{t+h}) - E_t(x_t)}{h}
        = \lim_{n\to\infty} \frac{E_{t}(x_{t+h_n}) - E_t(x_t)}{h_n},
    \end{equation}
    and such that there exits $x'\in \Xx$ for which $d(x_{t+h_n}, x')\to 0$ as $n\to\infty$. We recall that it is possible to find such a sequence by virtue of \cref{rmk:compactness_limit_traj}.
    Moreover, since $[(t+ h_n, x_{t+h_n})]= \hat \eta ^+(t+h_n) \to_\X \hat\eta^+(t)$ as $n\to\infty$, we deduce that $(t,x')\in \hat\eta^+(t)$.
    Hence, since $[(t,x')] = \hat\eta^+(t) = [(t,x_t)]$, from \cref{eq:app_liminf_equiv,eq:app_liminf_subseq} we obtain that
    \begin{equation} \label{eq:app_liminf_new_rep}
        \liminf_{h\to 0^+} 
        \frac{f_+(t+h)-f_+(t)}{h} =
        \lim_{n\to\infty} \frac{E_{t}(x_{t+h_n}) - E_t(x')}{h_n},
    \end{equation}
    where we used the fact that $E_t(x_t) = E_t(x')$, ensured by \cref{ass:conn_comp}.
    Then, we note that if $E_{t}(x_{t+h_n}) = E_t(x')$ for infinitely many $n\in \NN$, then the thesis follows immediately. Therefore, we assume that $E_{t}(x_{t+h_n}) \neq E_t(x')$ for $n$ large enough.
    From this fact, since $d(x_{t+h_n}, x')\to 0$ as $n\to \infty$ and since the connected components of critical points are assumed to be well-separated (see again \cref{ass:conn_comp}), it follows that eventually $|\partial E_{t}|(x_{t+h_n}) \neq 0$.
    Hence, we compute
    \begin{equation*}
            \frac{E_{t}(x_{t+h_n}) - E_t(x')}{h_n} =
            \frac{E_{t}(x_{t+h_n}) - E_t(x')}{|\partial E_{t}|(x_{t+h_n})} \frac{|\partial E_{t}|(x_{t+h_n})}{h_n}  \geq -\varepsilon_{x'}\big(  d(x_{t+h_n},x') \big) L,
    \end{equation*}
    where we used \cref{ass:PL+time_der_slope}, together with the fact that $|\partial E_{t+h_n}|(x_{t+h_n})=0$, due to $[(t+h_n,x_{t+h_n})]=\hat \eta^+(t+h_n)$.
    Hence, by taking the limit in the last inequality, we finally deduce that
    \begin{equation*}
        \liminf_{h\to 0^+} 
        \frac{f_+(t+h)-f_+(t)}{h} =
        \lim_{n\to\infty} \frac{E_{t}(x_{t+h_n}) - E_t(x')}{h_n} \geq 0
    \end{equation*}
    and we conclude the proof.
\end{proof}

We report below a well-known sufficient condition for monotonicity of continuous functions.

\begin{lemma} \label{lem:monotn_Dini}
    Let $g\colon [a,b] \to \R$ be continuous and such that the Dini upper right derivative of $g$ satisfies
    \begin{equation*}
        D^+g'(t) \coloneqq  \limsup_{h\to 0^+} 
        \frac{g(t+h)-g(t)}{h} \geq 0
    \end{equation*}
    for every $t\in (a,b)$. Then, $g$ is non-decreasing on $(a,b)$.
\end{lemma}
\begin{proof}
    See, e.g., \cite[Lemma~5.2]{Scilla_2018}.
\end{proof}

Now, we show that the non-negative measure $\bar \mu$ appearing in the energy balance in \cref{thm:main_result_complete} is purely atomic. 
Since we already took care of the peculiarity of our setting in \cref{prop:dini_der}, we report the the arguments in the next proof follows the lines of \cite[Theorem~5.4]{Scilla_2018}.
We detail the proof for the sake of completeness.

\begin{proposition}\label{prop:mu_atomic}
    Let us assume \cref{ass:conn_comp,ass:der_time,ass:metric_space,ass:path+Lipsch,ass:PL+time_der_slope,ass:reg_E,ass:unif_coerc,ass:t_der_quotient}.
    Let $\bar \mu$ be the non-negative measure appearing in the energy balance in \cref{thm:traj_convergence}. Then, we have that $\mathrm{supp}\, \mu = J$, where $J$ is the jump set of the limiting trajectory $\hat\eta\colon [0,T]\to \X$.
\end{proposition}

\begin{proof}
    We consider the non-increasing function $f\colon [0,T]\to \R$ defined as in \cref{eq:def_ancill_funct_mu}, and, owing to \cref{thm:traj_convergence} and to the enhanced energy balance in \cref{eq:energy_balance_quotient}, we observe that the distributional derivative of $f$ satisfies $\bar \mu = -df$, and we recall $\bar \mu$ is a positive measure.
    To see that $\mathrm{supp}\, \bar \mu = J$, we introduce the non-increasing function
    \begin{equation*}
        f^J(t) \coloneqq \sum_{s\in [0,t]}\big( f_+(s) - f_-(s) \big),
    \end{equation*}
    and we observe that it is right-continuous. Moreover, the set of dicontinuity points for $f^J$ is exactly $J$, and the distributional derivative satisfies $d(f^J) = (df)^J$, where  $(df)^J$ denotes the jump part of the measure $df$. Then, we observe that $\bar \mu = -df \geq -d(f^J) \eqqcolon \bar \mu^J$.
    In addition, recalling that $ f^J(t)-f^J(t+h) \geq 0$ for every $h\geq 0$, we have that
    \begin{equation} \label{eq:dini_specific}
        \liminf_{h\to 0} 
        \frac{(f_+-f^J)(t+h)-(f_+-f^J)(t)}{h} \geq
        \liminf_{h\to 0} 
        \frac{f_+(t+h) - f_+(t)}{h} \geq 0.
    \end{equation}
    We further notice that the function $f_+-f^J$ is continuous. Indeed, on the one hand, both $f_+$ and $f^J$ are right-continuous. 
    On the other hand, we have $\lim_{h\to 0^+} f_+(t-h) = f_-(t)$ and
    \begin{equation*}
        \lim_{h\to 0^+} f^J(t-h) = \sum_{s\in [0,t)}\big( f_+(s) - f_-(s) \big) = f^J(t) - \big( f_+(t) - f_-(t) \big)
    \end{equation*}
    for every $t\in(0,T]$, which combined provide the left-continuity of $f_+-f^J$.
    Hence, by using \cref{eq:dini_specific} and \cref{lem:monotn_Dini}, we deduce that $f_+-f^J$ is non-decreasing.
    Therefore, from the inequality $f_+(t)-f^J(t)\geq f_+(0)-f^J(0)$, recalling that $f^J(t) = - \bar \mu^J([0,t])$ and that $f_+(0)-f^J(0)=f_-(t)$, we obtain that 
    \begin{equation*}
        \hat E\big(\hat\eta^+(t)\big) + \bar \mu^J ([0, t]) \geq \hat E\big(\hat\eta^-(0)\big) + \int_0^t \partial_t \hat E \big(\hat\eta(\tau)\big) \diff \tau 
    \end{equation*}
    for every $t\in [0,T]$.
    Recalling that $\bar \mu^J, \bar \mu$ are positive, from the last inequality and from the balance in \cref{eq:energy_balance_quotient} we deduce that $\bar \mu^J \geq \bar \mu$. Since by construction $\bar \mu \geq \bar \mu^J$, this shows that $\bar \mu = \bar \mu^J$ and concludes the proof.
\end{proof}

\section{Gradient flow vs.~minimizing movement scheme actions}
In this section, we relate the actions $\cB_t$ and $\cM_t$, corresponding to the gradient descent and to the minimizing movement scheme, respectively.
From here on, we fix a time $t \in [0, T]$ and consequently drop the subscript $t$ from the notation:
We concern ourselves with two actions $\cB$ and $\cM$
and an energy functional $E$ on $\Xx$.
We begin by relating $\CostGF$ and $\CostM$, the instantaneous costs  of the gradient flow and of the minimizing movement scheme, respectively.
\begin{lemma}\label{lemma:instanteneous_cost_relation}
    Let us assume \cref{ass:conn_comp,ass:der_time,ass:metric_space,ass:path+Lipsch,ass:PL+time_der_slope,ass:reg_E,ass:unif_coerc}. Furthermore, let us assume that $L < \frac{1}\tau$, where $L$ is the Lipschitz constant of $|\partial E|$. Then, setting $\varepsilon_L \coloneqq L \cdot \tau$, we have that, for every $x \in \Xx$,
    \begin{equation}\label{eq:app2:relation_between_instantaneous_cost}
        (1 - \varepsilon_L) \CostM\big(i(x)\big) \leq \tau \CostGF\big(i(x)\big) \leq (1 + \varepsilon_L) \CostM\big(i(x)\big).
    \end{equation}
\end{lemma}
\begin{proof}
    We start with the first inequality in \eqref{eq:app2:relation_between_instantaneous_cost}.
    To this end, we fix $x \in \Xx$, pick $x' \in \Xx$ and $\varepsilon > 0$, and set $\underline{d} \coloneqq d(x, x')$. We use \cref{ass:metric_space} to choose an absolutely continuous curve $\gamma\colon [0, \underline{d} + \varepsilon] \to \Xx$ such that $\dot \gamma \equiv 1$, $\gamma(0) = x$ and $\gamma(\underline{d} + \varepsilon) = x'$.
    Since $|\partial E|$ is a strong upper gradient---see \cref{rmk:slope_strong_upper_gradient}---we have
    \begin{align*}
        E(x) - E(x') - \frac{1}{2\tau} \underline{d}^2 &\leq \int_0^{\underline{d}+\varepsilon}|\partial E|(\gamma(s)) \diff s - \frac{1}{2\tau} \underline{d}^2
        \leq \int_0^{\underline{d} + \varepsilon} \big(|\partial E|(x) + L s\big) \diff s - \frac{1}{2\tau} \underline{d}^2 \\
        &\leq (\underline{d} + \varepsilon) |\partial E|(x) + \frac{1}{2}L(\underline{d} + \varepsilon)^2 - \frac{1}{2\tau} \underline{d}^2.
    \end{align*}
    Taking $\varepsilon \to 0$, we get
    \begin{equation}\label{eq:app2:bound_over_energy_diff}
        E(x) - E(x') - \frac{1}{2\tau} \underline{d}^2 \leq \underline{d} |\partial E|(x) + L\frac{1}{2}\underline{d}^2 - \frac{1}{2\tau} \underline{d}^2 = \underline{d} |\partial E|(x) - \frac{1}{2\tau} (1 - \varepsilon_L)  \underline{d}^2.
    \end{equation}
    The right hand side above is a quadratic equation whose maximum is attained at $\underline{d} = \tau|\partial E|(x) \frac{1}{(1 - \varepsilon_L)}$; substituting this choice of $\underline{d}$ into \cref{eq:app2:bound_over_energy_diff} gives
    \begin{equation*}\label{eq:app2:uniform_bound_over_energy_diff}
        E(x) - E(x') - \frac{1}{2\tau} \underline{d}^2 \leq {\tau} \frac{1}{1 - \varepsilon_L} \frac{1}{2} |\partial E|(x)^2
    \end{equation*}
    As $\CostM\big(i(x)\big)$ is the supremum of the left hand side of the above equation over all $x' \in \Xx$, we have
    \begin{equation*}
        \CostM\big(i(x)\big) \leq {\tau} \frac{1}{1 - \varepsilon_L} \frac{1}{2} |\partial E|(x)^2 = {\tau} \frac{1}{1 - \varepsilon_L} \CostGF(i(x)), 
    \end{equation*}
    which finishes the proof of the first inequality.\\
    For the second inequality in \eqref{eq:app2:relation_between_instantaneous_cost}, we start by letting $\gamma\colon [0, \infty) \to \Xx$ be a curve of maximum slope such that $\gamma(0)=x$, which in particular implies that the following equation holds for all $s \geq 0$:
    \begin{equation}\label{eq:gamma_max_slope}
        E(x) - E\big(\gamma(s)\big) = \int_0^s |\partial E|\big(\gamma(r)\big)|\dot \gamma| \diff r;
    \end{equation}
    such a curve of maximum slope exists by \cref{rmk:slope_strong_upper_gradient} and \cite[Prop.~2.2.3, Theorem~2.3.3]{greenbook}.
    For $s > 0$, we set $l_s \coloneqq \int_0^s |\dot \gamma(\sigma) | \diff \sigma$ and we define $l_\infty \coloneqq \int_0^\infty |\dot \gamma(\sigma) | \diff \sigma$.
    Note that, as $\gamma$ is a curve of maximum slope, in addition to \cref{eq:gamma_max_slope}, the following holds for almost all $s > 0$:
    \begin{equation*}
        |\dot \gamma|(s) = |\partial E|(\gamma(s)) \leq |\partial E|(x) - L \int_0^s |\dot \gamma|(r) \diff r.
    \end{equation*}
    Using Grönwall's lemma, we deduce from the above that $|\dot \gamma|(s) \geq |\partial E|(x) e^{-L s}$,
    and thus that
    \begin{equation}\label{eq:l_infty}
        l_\infty \geq \frac{1}{L} |\partial E|(x).
    \end{equation}
    As a next step, we reparametrize $\gamma$ to obtain a curve $\widetilde{\gamma}\colon [0, l_\infty] \to \Xx$ such that $\widetilde{\gamma}(0) = x$,
    $\dot{\widetilde{\gamma}}(s) \equiv 1$ and such that \cref{eq:gamma_gradient_flow} holds for $\widetilde{\gamma}$
    in place of $\gamma$.
    We calculate, for $s \in [0, l_\infty]$,
    \begin{align*}
        \CostM\big(i(x)\big) &\geq E(x) - E\big(\widetilde{\gamma}(s)) - \frac{1}{2\tau} d\big(x, \widetilde{\gamma}(s)\big)^2 = \int_0^s |\partial E|\big(\widetilde{\gamma}(r)\big) \diff r - \frac{1}{2\tau} d\big(x, \widetilde{\gamma}(s)\big)^2 \\
        &\geq \int_0^s \left(|\partial E|(x) - L r\right) \diff r - \frac{1}{2\tau} s^2 = s |\partial E|(x) - \frac{1}{2}L s^2 - \frac{1}{2\tau} s^2 \\
        &= s |\partial E|(x) - \frac{1}{2} \left(L + \frac{1}{\tau}\right) s^2.
    \end{align*}
    Evaluating the above at $s^\star \coloneqq \frac{|\partial E|(x)}{L + \frac{1}{\tau}}$---justified by \cref{eq:l_infty}--- gives
    \begin{align*}
        \CostM\big(i(x)\big) &\geq \frac{1}{2} |\partial E|(x)^2 \tau \frac{1}{1 + \varepsilon_L} = \tau \CostGF\big(i(x)\big) \frac{1}{1 + \varepsilon_L},
    \end{align*}
    which finishes the proof of the second inequality.
\end{proof}
As a next step, we show that we can bound the cost of a single step of the minimizing movement scheme
by the cost of the gradient descent action, as long as step is small enough.
\begin{lemma}\label{lemma:estimate_one_step_of_mms_by_gf}
    Let us assume \cref{ass:conn_comp,ass:der_time,ass:metric_space,ass:path+Lipsch,ass:PL+time_der_slope,ass:reg_E,ass:unif_coerc}. Furthermore, let us assume that $L < \frac{1}\tau$, where $L$ is the Lipschitz constant of $|\partial E|$. Then, setting $\varepsilon_L \coloneqq L \cdot \tau$, we have, for each $x, x' \in \Xx$ such that $d(x, x') \leq \tau \cdot |\partial E|(x)$,
    \begin{equation}\label{eq:app2:bound_of_single_step_with_error}
        \begin{split}
            (1 - \varepsilon_L) \CostM\big((x, x')\big) &\leq 2 \cGF(x, x')  + |\partial E|(x) \big(\tau |\partial E|(x) - d(x, x')\big) \\
            &\leq 2 \cGF(x, x') + 2 \tau \CostGF\big(i(x)\big).
        \end{split}
    \end{equation}
    In particular, if $d(x, x') = \tau |\partial E|(x)$, the above inequality reduces to
    \begin{equation}\label{eq:app2:bound_of_single_step_without_error}
        (1 - \varepsilon_L) \CostM\big((x, x')\big) \leq 2 \cGF(x, x').
    \end{equation}
\end{lemma}
\begin{proof}
    Both the second inequality in~\eqref{eq:app2:bound_of_single_step_with_error} and~\eqref{eq:app2:bound_of_single_step_without_error} follow immediatly. In the rest of this proof, we concern ourselves the first inequality in~\eqref{eq:app2:bound_of_single_step_with_error}.
    Under the above assumptions, we have, using \cref{lemma:instanteneous_cost_relation}:
    \begin{equation}\label{eq:action_estimation_1}
        \begin{split}
            (1 - \varepsilon_L) \CostM\big((x, x')\big) &= (1 - \varepsilon_L) \bigg( \CostM\big(i(x)\big) + \frac{1}{2\tau} d(x, x')^2\bigg) \\
            &\leq \tau \CostGF\big(i(x)\big) + \frac{1}{2\tau}\tau^2|\partial E|(x)^2 = 2\tau\CostGF\big(i(x)\big).
        \end{split}
    \end{equation}
    On the other hand, we can bound $\cGF(x, x')$ from below: By virtue of \cref{lemma:equivalent_minimization} it suffices to find a lower bound for $\int_0^s |\partial E|(\gamma(\sigma)) \diff \sigma$, uniformly over $s > 0$ and all absolutely continuous curves $\gamma\colon [0, s] \to \Xx$ such that $\dot \gamma \equiv 1$, $\gamma(0) = x$ and $\gamma(s) = x'$. Let such $s$ and $\gamma$ be given. We note that $s \geq d(x, x') \eqqcolon \underline{d}$ and write
    \begin{equation*}
        \begin{split}
            \int_0^s |\partial E|(\gamma(\sigma)) \diff \sigma &\geq \int_0^{\underline{d}} |\partial E|(\gamma(\sigma)) \diff \sigma  \geq \underline{d} |\partial E|(x) - \frac{1}{2} L \underline{d}^2 \\
            &\geq \frac{1}{2} |\partial E|(x) \underline{d} = \frac{1}{2} \tau |\partial E|(x)^2 - \frac{1}{2} |\partial E|(x) \big(\tau |\partial E|(x) - \underline{d}\big).
        \end{split}
    \end{equation*}
    Taking the infimum over $\gamma$ in the inequality above, we obtain
    \begin{equation}\label{eq:action_estimation_2}
        \cGF(x, x') \geq \tau \CostGF\big(i(x)\big) - \frac{1}{2} |\partial E|(x) \big(\tau |\partial E|(x) - \underline{d}\big).
    \end{equation}
    Combining \cref{eq:action_estimation_1,eq:action_estimation_2} gives
    \begin{align*}
        (1 - \varepsilon_L) \CostM\big((x, x')\big) &\leq 2 \cGF(x, x') + |\partial E|(x) \big(\tau |\partial E|(x) - \underline{d}\big)
    \end{align*}
    as required.
\end{proof}
We continue by proving a inverse relation to \cref{lemma:estimate_one_step_of_mms_by_gf}:
\begin{lemma}\label{lemma:estimate_gf_by_one_step_of_mms}
    Let us assume \cref{ass:conn_comp,ass:der_time,ass:metric_space,ass:path+Lipsch,ass:PL+time_der_slope,ass:reg_E,ass:unif_coerc}. Furthermore, let us assume that $L < \frac{1}\tau$, where $L$ is the Lipschitz constant of $|\partial E|$. Then, setting $\varepsilon_L \coloneqq L \cdot \tau$, we have, for every $x, x' \in \Xx$:
    \begin{equation*}
        \cGF(x, x') \leq 4 \frac{1 + \varepsilon_L}{1 - \varepsilon_L} \,\CostM\big((x, x')\big)
    \end{equation*}
\end{lemma}
\begin{proof}
    The proof is similar to the first part of the proof of \cref{lemma:instanteneous_cost_relation}.
    We show the inequality by choosing, for an arbitrary $\varepsilon > 0$, as a competitor for the infimum in \cref{eq:equivalent_minimization_2}
    an absolutely continuous curve $\gamma\colon [0, \underline{d} + \varepsilon] \to \Xx$ such that $\dot \gamma \equiv 1$, $\gamma(0) = x$ and $\gamma(\underline{d} + \varepsilon) = x'$; the existence of such a curve is guaranteed by \cref{ass:metric_space}.
    With this choice, we have, setting $\underline{d} \coloneqq d(x, x')$:
    \begin{equation*}
        \cGF(x, x') \leq \int_0^{\underline{d}+\varepsilon} |\partial E|\big(\gamma(s)\big) \diff s 
        \leq (\underline{d} + \varepsilon) |\partial E|(x) + \frac{1}{2}L (\underline{d} + \varepsilon)^2.
    \end{equation*}
    Taking the limit $\varepsilon \to 0$, we get
    \begin{equation}\label{eq:gradient_cost_crude_estimate}
        \cGF(x, x') \leq \underline{d}|\partial E|(x) + \frac{1}{2}L \underline{d}^2 = \underline{d}|\partial E|(x) + \frac{1}{2\tau} \underline{d}^2 - \frac{1}{2\tau} ( 1 - \varepsilon_L)  \underline{d}^2.
    \end{equation}
    If $\underline{d} \leq \frac{2\tau}{1 - \varepsilon_L} |\partial E|(x)$, we use \cref{lemma:instanteneous_cost_relation} to reduce the above to
    \begin{align*}
        \cGF(x, x') &\leq \frac{2\tau}{1 - \varepsilon_L}|\partial E|^2(x) + \frac{1}{2\tau} \underline{d}^2 = \frac{4}{1 - \varepsilon_L}\tau \CostGF(i(x)) + \frac{1}{2\tau} \underline{d}^2 \\
        &\leq  4 \frac{1 + \varepsilon_L}{1 - \varepsilon_L}\, \CostM\big(i(x)\big) + \frac{1}{2\tau} \underline{d}^2 \leq 4 \frac{1 + \varepsilon_L}{1 - \varepsilon_L} \,\CostM\big((x, x')\big).
    \end{align*}
    On the other hand, if $\underline{d} >  \frac{2\tau}{1 - \varepsilon_L} |\partial E|(x)$, 
    \cref{eq:gradient_cost_crude_estimate} also gives
    \begin{align*}
        \cGF(x, x') &\leq \underline{d}|\partial E|(x) + \frac{1}{2\tau} \underline{d}^2 - \frac{1}{2\tau} ( 1 - \varepsilon_L)   \frac{2\tau}{1 - \varepsilon_L} |\partial E|(x) \underline{d}  \\
        &= \frac{1}{2\tau} \underline{d}^2 \leq \CostM\big((x, x')\big) \leq 4 \frac{1 + \varepsilon_L}{1 - \varepsilon_L} \, \CostM\big((x, x')\big).
    \end{align*}
\end{proof}
We are now in a position to prove a relation between the actions $\cB$ and $\cM$.
\begin{lemma}
    Let us assume \cref{ass:conn_comp,ass:der_time,ass:metric_space,ass:path+Lipsch,ass:PL+time_der_slope,ass:reg_E,ass:unif_coerc}. Furthermore, let us assume that $L < \frac{1}\tau$, where $L$ is the Lipschitz constant of $|\partial E|$. Then, setting $\varepsilon_L \coloneqq L \cdot \tau$, we have, for any $x \neq x' \in \Xx$
    \begin{equation}\label{eq:app2:relation_of_actions}
        \frac{1-\varepsilon_L}{4(1+\varepsilon_L)} \cGF(x, x') \leq \cM(x, x') - \CostM\big(i(x)\big) \leq \frac{2}{1-\varepsilon_L} \cGF(x, x') + \frac{2 \tau R}{1-\varepsilon_L},
    \end{equation}
    where $R \coloneqq \sup_{x'' \in K_\Xx}\CostGF\big(i(x'')\big)$, $K \subseteq \PpGF$ is a chosen compact such that \cref{prop:curves_coercive} for $\cGF$ is fulfilled for $K$, and $K_\Xx$ denotes the projection of $K$ onto $\Xx$.
\end{lemma}
\begin{proof}
    We start with the first inequality. By taking the infimum, it suffices to show that for all competitors $\gamma \in \Gamma\MMS\big(i(x), i(x')\big)$,
    we have that
    \begin{equation}\label{eq:app2:upper_bound_costGF_proof}
        \frac{1-\varepsilon_L}{4(1+\varepsilon_L)} \cGF(x, x') \leq \sum_{j=a+1}^{b} \CostM\big(\gamma(j)\big) - \CostM\big(i(x)\big) = \sum_{j=a+1}^{b-1} \CostM\big(\gamma(j)\big),
    \end{equation}
    where $a < b \in \Z$ are chosen such that $\dom(\gamma) = [a, b]_\Z$; for the second equality above, note that necessarily $\gamma(b) = i(x')$. However,~\eqref{eq:app2:upper_bound_costGF_proof} follows by repeatedly applying \cref{lemma:estimate_gf_by_one_step_of_mms}, using the triangle inequality on $\cGF$ and noting that $\gamma_2(j) = \gamma_1(j+1)$ for all $j \in [a, b-1]_\Z$ and that $\gamma_1(a+1) = x$, $\gamma_2(b-1) = x'$. \\
    We now turn to the second inequality in~\eqref{eq:app2:relation_of_actions}, which is a little more intricate to prove. We start by picking some $\varepsilon > 0$ and a curve $\gamma \in \Gamma\GF\big(i(x), i(x')\big)$ such that
    $\int_a^b\CostGF\big(\gamma(s)\big) \diff s \leq \cGF(x, x') + \varepsilon$,
    where $a < b \in \R$ are such that $\dom(\gamma) = [a, b]_\R$, and such that $\im(\gamma) \subseteq K$---\cref{prop:curves_coercive} guarantees that we can indeed pick such a curve. In the rest of this proof, we will choose appriate control points along the curve $\gamma$ to repeatedly apply \cref{lemma:estimate_one_step_of_mms_by_gf} to the segments of $\gamma$. \\
    Recalling the indexing conventions on $\gamma$ introduced in \cref{lemma:gradient_flow_summary_lemma}, we pick some arbitrary $\varepsilon' > 0$ and iteratively choose controlpoints
    $a=t_1 < \dots < t_n \in [a, b]$ such that---setting $y_i \coloneqq \gamma_x(t_i)$---$d(y_i, y_{i+1}) = \max\{\tau |\partial E|(y_i), \varepsilon'\}$ for $1 \leq i < n$ and such that $d(y_n, x') \leq \tau |\partial E|(y_n)$. Denoting the length of $\gamma_x$ by $l$, we have that $n \leq \frac{l}{\varepsilon'}$. For each $1 \leq i < n$, if 
    $\tau |\partial E|(y_i) \leq \varepsilon'$, we can use \cref{lemma:instanteneous_cost_relation} to see that
    \begin{align*}
        (1 - \varepsilon_L) \, \CostM\big((y_i, y_{i+1})\big) &\leq \tau \CostGF\big(i(y_i)\big) + \frac{1-\varepsilon_L}{2 \tau} \varepsilon'^2 = \frac{1}{2\tau}\big(\tau|\partial E|(y_i)\big)^2 + \frac{1-\varepsilon_L}{2 \tau} \varepsilon'^2 \\
        &\leq \frac{1}{2\tau} \varepsilon'^2 + \frac{1}{2\tau} \varepsilon'^2 = \frac{1}{\tau} \varepsilon'^2.
    \end{align*}
    On the other hand, if $\tau |\partial E|(y_i) \geq \varepsilon'$, we can use \cref{lemma:estimate_one_step_of_mms_by_gf} to see that $(1 - \varepsilon_L) \, \CostM\big((y_i, y_{i+1})\big) \leq 2 \cGF(y_i, y_{i+1})$. Thus, either way,
    \begin{equation*}
        (1 - \varepsilon_L) \, \CostM\big((y_i, y_{i+1})\big) \leq 2 \cGF(y_i, y_{i+1}) + \frac{1}{\tau} \varepsilon'^2 \leq 2 \int_{t_i}^{t_{i+1}} \CostGF\big(\gamma(s)\big) \diff s + \frac{1}{\tau} \varepsilon'^2;
    \end{equation*}
    furthermore, also using \cref{lemma:estimate_one_step_of_mms_by_gf}, we also have that
    \begin{equation*}
        (1 - \varepsilon_L) \, \CostM\big((y_n, x')\big) \leq 2 \cGF(x, x') + 2 \tau \CostGF\big(i(x)\big) \leq 2 \int_{t_n}^{b} \CostGF\big(\gamma(s)\big) \diff s + 2 \tau \CostGF\big(i(x)\big).
    \end{equation*}
    Using the curve $\big(i(x), (y_1, y_2), (y_2, y_3), \dots, (y_{n-1}, y_n), (y_n, x'), i(x')\big) \in \Gamma\MMS\big(i(x), i(x')\big)$ as a competitor for the infimum in $\cM(x, x')$, we finally get
    \begin{align*}
        \cM(x, x') &\leq \sum_{i=1}^{n-1} \CostM\big((y_i, y_{i+1})\big) + \CostM\big((y_n, x')\big) + \CostM\big(i(x')\big) \\
        &\begin{aligned}
            \leq \frac{1}{1 - \varepsilon_L}\bigg(2 \sum_{i=1}^{n-1} \int_{t_i}^{t_{i+1}} \CostGF\big(\gamma(s)\big) \diff s + (n-1) \frac{1}{\tau} \varepsilon'^2 + &\\
            2 \int_{t_n}^{b} \CostGF\big(\gamma(s)\big) \diff s + 2 \tau \CostGF\big(i(x)\big)&\bigg) + \CostM\big(i(x')\big)
        \end{aligned} \\
        &\leq \frac{1}{1 - \varepsilon_L}\bigg( 2 \int_a^b \CostGF\big(\gamma(s)\big) \diff s + \frac{l \varepsilon'}{\tau} + 2\tau R\bigg) + \CostM\big(i(x')\big) \\
        &\leq \frac{2}{1-\varepsilon_L}\cGF(x, x') + \varepsilon + \frac{l \varepsilon'}{\tau} + \frac{2 \tau R}{1-\varepsilon_L} + \CostM\big(i(x')\big).
    \end{align*}
    As both $\varepsilon > 0$ and $\varepsilon' > 0$ were arbitrary, we can take the limit $\varepsilon \to 0$ and $\varepsilon' \to 0$ in the above inequality to obtain the thesis.
\end{proof}

\section*{Acknowledgments}
The work of SA was partially funded by the Austrian Science Fund through the project 10.55776/P35359, by the University of Naples Federico II through the FRA Project ``ReSinApas", by the MUR - PRIN 2022 project ``Variational Analysis of Complex Systems in Materials Science, Physics and Biology'', No.~2022HKBF5C, funded by European Union NextGenerationEU, and by the Gruppo Nazionale per l'Analisi Matematica, la Probabilit\`a e le loro Applicazioni (GNAMPA-INdAM, Project 2025: DISCOVERIES - Difetti e Interfacce in Sistemi Continui: un'Ottica Variazionale in Elasticit\`a con Risultati Innovativi ed Efficaci Sviluppi).
AS acknowledges partial support from GNAMPA-INdAM. MF acknowledges the support of the Munich Center for Machine Learning and the ERC Advanced Grant NEITALG, grant agreement No. 101198055.

\bigskip
\begin{center}
  \FundingLogos
  
  \vspace{0.5em}
  \begin{tcolorbox}\centering\small
    Funded by the European Union. Views and opinions expressed are however those of the author(s) only and do not necessarily reflect those of the European Union or the European Research Council Executive Agency. Neither the European Union nor the granting authority can be held responsible for them.
    This project has received funding from the European Research Council (ERC) under the European Union’s Horizon Europe research and innovation programme (grant agreement No. 101198055, project acronym NEITALG).
    
  \end{tcolorbox}
\end{center}

\bibliographystyle{siam}
\bibliography{biblio}

\begin{thebibliography}{10}

\bibitem{Agostiniani_2012}
{\sc V.~Agostiniani}, {\em Second order approximations of quasistatic evolution
  problems in finite dimension}, Discrete and Continuous Dynamical Systems, 32
  (2012), pp.~1125--1167.

\bibitem{AR17}
{\sc V.~Agostiniani and R.~Rossi}, {\em Singular vanishing-viscosity limits of
  gradient flows: the finite-dimensional case}, J. Differential Equations, 263
  (2017), pp.~7815--7855.

\bibitem{ARS_2015}
{\sc V.~Agostiniani, R.~Rossi, and G.~Savar{\'e}}, {\em On the transversality
  conditions and their genericity}, Rendiconti del Circolo Matematico di
  Palermo (1952 -), 64 (2015), pp.~101--116.

\bibitem{Almi-Belz-AMPA}
{\sc S.~Almi and S.~Belz}, {\em Consistent finite-dimensional approximation of
  phase-field models of fracture}, Ann. Mat. Pura Appl. (4), 198 (2019),
  pp.~1191--–1225.

\bibitem{Almi-Belz-Negri}
{\sc S.~Almi, S.~Belz, and M.~Negri}, {\em Convergence of discrete and
  continuous unilateral flows for {Ambrosio–Tortorelli} energies and
  application to mechanics}, ESAIM: Math. Model. Numer. Anal., 53 (2019),
  pp.~659--699.

\bibitem{Almi-Negri2020}
{\sc S.~Almi and M.~Negri}, {\em Analysis of staggered evolutions for nonlinear
  energies in phase field fracture}, Arch. Ration. Mech. Anal., 236 (2020),
  pp.~189--252.

\bibitem{greenbook}
{\sc L.~Ambrosio, N.~Gigli, and G.~Savaré}, {\em Gradient Flows in Metric
  Spaces and in the Space of Probability Measures}, Lectures in Mathematics
  {{ETH Zürich}}, {Birkhäuser}, 2. ed~ed.
\newblock OCLC: 254181287.

\bibitem{ACFS17}
{\sc M.~Artina, F.~Cagnetti, M.~Fornasier, and F.~Solombrino}, {\em Linearly
  constrained evolutions of critical points and an application to cohesive
  fractures}, Mathematical Models and Methods in Applied Sciences, 27 (2017),
  pp.~231--290.

\bibitem{Bourdin}
{\sc B.~Bourdin}, {\em Numerical implementation of the variational formulation
  for quasi-static brittle fracture}, Interfaces Free Bound., 9 (2007),
  p.~411–430.

\bibitem{dalmaso2006quasistatic}
{\sc G.~Dal~Maso, A.~DeSimone, and M.~G. Mora}, {\em Quasistatic evolution
  problems for linearly elastic--perfectly plastic materials}, Arch. Rat. Mech.
  Anal., 180 (2006), pp.~237--291.

\bibitem{dalmaso2008vanishing}
{\sc G.~Dal~Maso, A.~DeSimone, M.~G. Mora, and M.~Morini}, {\em A vanishing
  viscosity approach to quasistatic evolution in plasticity with softening},
  Arch. Rat. Mech. Anal., 189 (2008), pp.~469--544.

\bibitem{efendiev2006rate}
{\sc M.~A. Efendiev and A.~Mielke}, {\em On the rate-independent limit of
  systems with dry friction and small viscosity}, Journal of Convex Analysis,
  13 (2006), p.~151.

\bibitem{fiaschi2009vanishing}
{\sc A.~Fiaschi}, {\em A vanishing viscosity approach to a quasistatic
  evolution problem with nonconvex energy}, 26 (2009), pp.~1055--1080.

\bibitem{Fornasier-Tradeoff}
{\sc M.~Fornasier, J.~Klemenc, and A.~Scagliotti}, {\em Trade-off invariance
  principle for minimizers of regularized functionals}, Journal of Optimization
  Theory and Applications, accepted, in press (2026).

\bibitem{Knees-NegriM3AS}
{\sc D.~Knees and M.~Negri}, {\em Convergence of alternate minimization schemes
  for phase-field fracture and damage}, Math. Models Methods Appl. Sci., 27
  (2017), pp.~1743--1794.

\bibitem{Kos}
{\sc C.~Kosniowski}, {\em A First Course in Algebraic Topology}, Cambridge
  University Press, 1980.

\bibitem{Matthes-Plazotta}
{\sc D.~Matthes and S.~Plazotta}, {\em A variational formulation of the {BDF}2
  method for metric gradient flows}, ESAIM Math. Model. Numer. Anal., 53
  (2019), pp.~145--172.

\bibitem{MR2525194}
{\sc A.~Mielke, R.~Rossi, and G.~Savar\'e}, {\em Modeling solutions with jumps
  for rate-independent systems on metric spaces}, Discrete Contin. Dyn. Syst.,
  25 (2009), pp.~585--615.

\bibitem{MR2887927}
\leavevmode\vrule height 2pt depth -1.6pt width 23pt, {\em B{V} solutions and
  viscosity approximations of rate-independent systems}, ESAIM Control Optim.
  Calc. Var., 18 (2012), pp.~36--80.

\bibitem{MR3531671}
\leavevmode\vrule height 2pt depth -1.6pt width 23pt, {\em Balanced viscosity
  ({BV}) solutions to infinite-dimensional rate-independent systems}, J. Eur.
  Math. Soc. (JEMS), 18 (2016), pp.~2107--2165.

\bibitem{Roubicek-Mielke}
{\sc A.~Mielke and T.~Roub\'{i}\v{c}kek}, {\em Rate-{I}ndependent {S}ystems},
  Springer New York, NY, 2015.

\bibitem{minotti2018viscous}
{\sc L.~Minotti and G.~Savar{\'e}}, {\em Viscous corrections of the time
  incremental minimization scheme and visco-energetic solutions to
  rate-independent evolution problems}, Archive for Rational Mechanics and
  Analysis, 227 (2018), pp.~477--543.

\bibitem{Riva-Scilla-Solombrino-1}
{\sc F.~Riva, G.~Scilla, and F.~Solombrino}, {\em The notions of inertial
  balanced viscosity and inertial virtual viscosity solution for
  rate-independent systems}, Adv. Calc. Var., 16 (2023), pp.~903--934.

\bibitem{Riva-Scilla-Solombrino-2}
\leavevmode\vrule height 2pt depth -1.6pt width 23pt, {\em Inertial balanced
  viscosity {(IBV)} solutions to infinite-dimensional rate-independent
  systems}, J. Funct. Anal., 288 (2025), p.~110830.

\bibitem{Scilla_2018_2}
{\sc G.~Scilla and F.~Solombrino}, {\em Delayed loss of stability in singularly
  perturbed finite-dimensional gradient flows}, Asymptotic Analysis, 110
  (2018), pp.~1--19.

\bibitem{Scilla_2018}
{\sc G.~Scilla and F.~Solombrino}, {\em Multiscale analysis of singularly
  perturbed finite dimensional gradient flows: the minimizing movement
  approach}, Nonlinearity, 31 (2018), p.~5036.

\bibitem{Scilla_2019}
{\sc G.~Scilla and F.~Solombrino}, {\em A variational approach to the
  quasistatic limit of viscous dynamic evolutions in finite dimension}, Journal
  of Differential Equations, 267 (2019), pp.~6216--6264.

\bibitem{W35}
{\sc H.~Whitney}, {\em Abstract for ``{A} function not constant on a connected
  set of critical points''}, Bull. Am. Math. Soc., 41 (1935), p.~796.
\newblock Abstract only; published in \textit{Duke Math.~J.}~\textbf{1}:4
  (1935), pp. 514--517. JFM:61.0262.07.

\bibitem{Zanini2007}
{\sc C.~Zanini}, {\em Singular perturbations of finite dimensional gradient
  flows}, Discrete and Continuous Dynamical Systems, 18 (2007), pp.~657--675.

\end{thebibliography}

\end{document}